%% file: FlammParreau_NonArchimedeanHilbertGeometry.tex
\documentclass[final,twoside,a4paper,reqno]{amsart}

\usepackage[UKenglish]{babel}

\usepackage{todonotes}
\newcommand{\comm}[1]{\medskip {\tt #1}}
\newcommand{\compl}[1]{\footnote{Compl: #1}}
\renewcommand{\comm}[1]{}
\renewcommand{\compl}[1]{}

\usepackage[utf8]{inputenc}   
\usepackage[TS1,T1]{fontenc}
\usepackage{lmodern}

\usepackage{amsmath,amssymb,amstext,amsthm,amsfonts}
\usepackage{mathtools} 

\usepackage{enumerate}
\usepackage{enumitem}
\usepackage{graphicx}
\usepackage{graphics}
\usepackage{float}
\usepackage{hyperref}
\usepackage{svg}
\hypersetup{
  colorlinks,
  linkcolor={blue!80!black},
 citecolor={green!90!black},
 urlcolor={red!80!black},
}
\usepackage{cleveref}
\usepackage{stmaryrd}

\usepackage[toc,page]{appendix}
\DeclareRobustCommand{\SkipTocEntry}[5]{}

\input{macros}

 \graphicspath{{figures/}}

\usepackage{pgfplots,tikz}
\usetikzlibrary{backgrounds}
\usetikzlibrary{shadings, fadings}

\newtheorem*{theoremIntro*}{Theorem}
\newtheorem{theoremIntro}{Theorem}

\newtheorem*{lemmaIntro*}{Lemma}
\newtheorem{lemmaIntro}[theoremIntro]{Lemma}
\newtheorem*{corollaryIntro*}{Corollary}
\newtheorem{corollaryIntro}[theoremIntro]{Corollary}

\newtheorem{proposition}{Proposition}[section]
\newtheorem{lemma}[proposition]{Lemma}

\newtheorem{theorem}[proposition]{Theorem}

\theoremstyle{definition}
\newtheorem{definition}[proposition]{Definition}

\newtheorem{example}[proposition]{Example}
\newtheorem{examples}[proposition]{Examples}
\newtheorem{remark}[proposition]{Remark}
\newtheorem{remarks}[proposition]{Remarks}
\newtheorem*{remark*}{Remark}

\newtheorem*{question*}{Question}

\title[Non-Archimedean Hilbert geometry]{Non-Archimedean Hilbert geometry and degenerations of real Hilbert geometries}

\date{\today}

\author{Xenia Flamm}
\address{Xenia Flamm, Institut des Hautes \'Etudes Scientifiques, France 
\newline
\indent Max Planck Institute for Mathematics in the Sciences, Germany}
\email{flamm@ihes.fr\\
xenia.flamm@mis.mpg.de}

\author{Anne Parreau}
\address{Anne Parreau, Institut Fourier, UMR 5582, Laboratoire de Mathématiques, Université Grenoble Alpes, CS 40700, 38058 Grenoble cedex 9, France}
\email{anne.parreau@univ-grenoble-alpes.fr}

\makeatletter
\providecommand\@dotsep{5}
\renewcommand{\listoftodos}[1][\@todonotes@todolistname]{%
  \@starttoc{tdo}{#1}}
\makeatother

\newcommand{\aptodo}[2][]{\todo[color=blue!40,#1]{#2}}
\newcommand{\xftodo}[2][]{\todo[color=magenta!80,#1]{#2}}

\renewcommand{\todo}[1]{}
\renewcommand{\aptodo}[1]{}
\renewcommand{\xftodo}[1]{}

\newcommand{\xfchanged}[1]{{\color{magenta} #1}}

\renewcommand{\xfchanged}[1]{{#1}}

\begin{document}
\def\subjclassname{\textup{2020} Mathematics Subject Classification}
\subjclass{
53C60, 
53C15,  
54A20,  	
12J15,  	
12J25.	
}

\maketitle

\begin{abstract}
We develop a theory of Hilbert geometry over general ordered valued
fields, associating with an open convex subset of the
projective space a quotient Hilbert metric space.
Under natural non-degeneracy assumptions,  we prove that the
ultralimit of a sequence of rescaled real Hilbert geometries 
is isometric to the Hilbert metric space 
of an open convex projective subset over a Robinson field. 
This result allows us to prove that ideal points
of the space of convex real projective structures on a closed manifold 
arise from actions on 
non-Archimedean Hilbert geometries without global fixed point.
We explicitly describe the Hilbert metric space 
of a non-Archimedean bounded polytope $\polyt$ defined over a subfield
of the valuation ring as the geometric realization of the flag complex
of $\polyt$ modeled on a Weyl chamber.
As an application, we obtain a complete description of Gromov--Hausdorff
limits of a real polytope with rescaled Hilbert metric. 
\end{abstract}

\section{Introduction}
Hilbert geometry is a generalization of hyperbolic geometry, 
first introduced by Hilbert in his letter to Klein
\cite{Hilbert_StraightLinesShortestConnection}, that has
been extensively studied in many different contexts of mathematics.
It has been in particular successfully used in the context 
of group actions on real projective spaces 
and  convex real projective structures on manifolds,  see e.g.\
\cite{Benzecri_VarietiesLocalementAffinesProjectives,
  Vinberg_DiscreteLinearGroupsGeneratedReflections}.
More recent developments include for example
\cite{Benoist_ConvexesDivisiblesI, 
KarlssonNoskov_HilbertMetricGromovHyp, 
ColboisVernicosVerovic_AreaIdealTrianglesHilbertGeom,  
ColboisVernicosVerovic_HilbertGeometryConvexPolygonalDomains, 
CooperLongTillmann_ConvexProjManifoldsCusps, 
DancigerGueritaudKassel_ConvexCocompactActionsRealProjGeom}.
For a thorough introduction we recommend \cite{Vernicos_IntroHilbertGeom, PapadopoulosTroyanov_HandbookHilbertGeometry},  as well as \cite{Benoist_SurveyDivisibleConvexSets, Quint_DivisibileConvexSets, Marquis_AroundGroupsHilbertGeometry} in the context of group actions.

The goal of this paper is 
to develop the foundations 
of  Hilbert geometry 
over general (possibly non-Archimedean) ordered valued fields,
and to deduce some applications to degenerations of convex real
projective geometries.
%
We are motivated in particular by
degenerations of representations of a finitely generated group in $\PSL(d,\RR)$ 
that preserve open \emph{properly convex} subsets of the real
projective space,  i.e.\ subsets that are convex and bounded in an
affine chart. 
Representations over non-Archimedean (ordered) valued fields 
arise naturally in the study of degenerations of representations, see e.g.\
\cite{
MorganShalen_ValuationsTreesDegHypStructures,
Brumfiel_RSCTeichmullerSpace,
Parreau_CompEspReprGroupesTypeFini,
BurgerIozziParreauPozzetti_RSCCharacterVarieties2}.
They appear for example 
by taking asymptotic cones of actions on symmetric spaces
and can be seen as representations over Robinson fields,  which
are obtained by taking ultralimits of $\RR$ endowed with a rescaled
absolute value.
Degenerations of convex real projective structures on manifolds,
have been studied by numerous authors, mostly in the case of surfaces,
see for example
\cite{Loftin_CompModuliSpaceConvexProjStrSurfaces,
Alessandrini_CompactificationConvexProjStrMnfd,
Nie_HilbGeomSimplTitsSets,
Zhang_DegConvexRP2StrSurf,
Parreau_InvariantWeaklyConvexCocompactSubspacesSurfaceGroupsA2Buildings,
Nie_LimitPolygonsConvexDomainsProjPlane,
OuyangTamburelli_LimitsBlaschkeMetric,
LoftinTamburelliWolf_LimitsCubicDifferentialsBuildings,
Reid_LimitsConvexProjectiveSurfacesFinslerMetrics}.
In the case of hyperbolic structures on a closed surface,
which corresponds to the case  where the convex set is the unit disk, 
Brumfiel interpreted degenerations as hyperbolic structures over 
non-Archimedean ordered valued fields \cite{Brumfiel_RSCTeichmullerSpace,
  Brumfiel_TreeNonArchimedeanHyperbolicPlane}. 
Generalizing these ideas to convex real projective
structures on a closed manifold,
Cooper and Delp suggested   
to study their degenerations using ultralimits 
and convex projective structures over the non-standard field 
of hyperreals and announced results in this direction at conferences 
around 2008.

Let us now explain the objects and results of the present article in more detail.

\subsection{Hilbert geometry over ordered valued fields} 
An \emph{ordered valued field} $\F$ is a field together with a total
order $<$ and a non-trivial absolute value $\abs{\cdot}\from \F^\times\to\R$ 
that are compatible in the sense that $0 \leq x\leq y$ implies that $\abs{x}\leq \abs{y}$ for all $x,y \in \F$.
We note $\Lambda\coloneqq \log \abs{\F^\times} \subseteq \R$ the value group of $\abs{\cdot}$. 
Such fields are endowed with a natural topology induced from the order,  or equivalently from the absolute value).
The field  $\F$ is called \emph{non-Archimedean} if there exists $x \in \F$ with $x>n$ for all $n \in \N$,  or equivalently if $\abs{\cdot}$ satisfies the ultrametric triangle inequality.

Let $V$ be a finite-dimensional $\F$-vector space.
Since $\F$ is ordered there is no difficulty in defining convex subsets of the projective space $\PP V$ as in the real case.
Let thus $\Omega$ be a convex subset of $\PP V$, bounded in some affine chart.
Classically (when $\F=\R$), 
given $x,y\in\Omega$ their Hilbert distance is defined 
as $d_\Omega (x,y) \coloneqq  \log |\CR{a,x,y,b}|$
where $a, b$ are the unique intersection points of $\partial\Omega$ with a
projective line $L_{xy}$ through $x$ and $y$,
and $\CR{a,x,y,b}$ is the cross-ratio, with convention
$\CR{0,1,t,\infty}=t$ in affine charts.
Over general ordered valued fields, $\partial\Omega$ might be empty and $a,b$ not defined.
One can  circumvent this problem using the following alternative
description of the Hilbert distance, 
which makes sense for general valued fields \cite{Guilloux_PAdicHypDisc,FalbelGuillouxWill_HilbertGeometryWithoutConvexity, FalbelGuillouxWill_HilbertMetricBoundedSymmetricDomains}.

\begin{definition}[Hilbert pseudo-distance]
Let $\Omega \subset \PP V$ be a convex set.
We define the function $d_\Omega \from \Omega \times \Omega \to \R \cup \{\infty\}$ by
\[d_\Omega (x,y) \coloneqq \sup_{a,b \in L_{xy}\setminus \Omega} \log |\CR{a,x,y,b}|\]
where $L_{xy}$ is a projective line  through $x$ and $y$.
\end{definition}

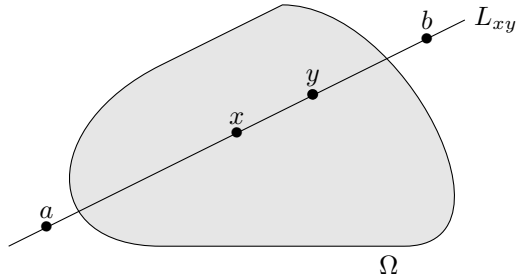
\begin{figure}[h]
\begin{tikzpicture}
\draw [scale=0.8,rotate=90, fill=gray!20]
(0,0) .. controls +(0,2) and +(-1,2) .. (3,0)
..controls +(0,0) and +(0,0)..(4,-2)
.. controls +(0,-2) and +(0,-2) .. (0,-4)
.. controls +(0,4) and +(0,0) .. (0,0);
\draw (-2,0)--(4,3);
\draw (3,0) node[below]{$\Omega$} ;
\draw (-1.5,0.25) node{$\bullet$} ;
\draw (-1.5,0.25) node[above]{$a$} ;
\draw (1,1.5) node{$\bullet$} ;
\draw (1,1.5) node[above]{$x$} ;
\draw (2,2) node{$\bullet$} ;
\draw (2,2) node[above]{$y$} ;
\draw (2,2) node{$\bullet$} ;
\draw (3.5,2.75) node[above]{$b$} ;
\draw (3.5,2.75) node{$\bullet$} ;
\draw (4,3) node[right]{$L_{xy}$} ;
\end{tikzpicture}
\caption{The Hilbert pseudo-distance on $\Omega$.}
\end{figure}

\noindent This function is symmetric,  non-negative and satisfies the triangle inequality,  see \Cref{lem:PseudoDistdOmega}.
We show in \Cref{lem:PropertiesdOmega:Finite} that if $\Omega$ is open,  then $d_\Omega$ is finite.
In \Cref{lem:PropertiesdOmega} we revisit some more basic properties of $d_\Omega$,  such as the inequality between nested convex sets,  the restriction to projective subspaces,  and the  invariance under a projective transformation preserving $\Omega$,  that are classical in the real case.
We also prove that $d_\Omega$ is additive on segments (\Cref{lem:PropertiesdOmega:AdditivityOnSegments}),  which implies that projective lines are shortest paths.

However contrary to the real case,  if $\F$ is non-Archimedean,  then $d_\Omega$ does not separate points and is thus only a pseudo-distance,  see e.g.\ \Cref{example:ball} for an example where all points have distance zero.
We define the quotient metric space.

\begin{definition}[Associated Hilbert metric space]
\label{dfn_AssociatedMetricSpace}
Let $\Omega \subseteq \PP V$ be open, convex and contained in an affine chart with Hilbert pseudo-distance $d_\Omega$.
The \emph{associated Hilbert metric space} is 
\[ X_\Omega \coloneqq \Omega/_\sim, \, \textnormal{ where } x \sim y \iff d_\Omega(x,y) = 0,\]
endowed with the distance induced by $d_\Omega$.
\end{definition}
We denote by $\overline{x}$ the class in $X_\Omega$ of $x\in \Omega$,  and write $\pi \from \Omega \to X_\Omega$ for the projection.

Let us now turn to some first examples.
When $\Omega$ is the interior of the unit ball, i.e.\ the non-Archimedean analogue of hyperbolic space,  and $\F$ is additionally \emph{real closed},  meaning that $\F[\sqrt{-1}]$ is algebraically closed,  Brumfiel in dimension two \cite{Brumfiel_TreeNonArchimedeanHyperbolicPlane},   and Bouzoubaa \cite{Bouzoubaa_CompRealSpectrumCharVarSOn1} in general,  showed that  $X_\Omega$ is a $\Lambda$-tree,  a generalization of real trees \cite{Chiswell_IntroductionLambdaTrees}.

On the other hand of the spectrum,  we show that the simplices play,  as in the real case,  the role of flats.
To a basis  $\eb = (e_1,\ldots,e_d)$ of $V$ one associates the closed $(d-1)$-dimensional \emph{simplex}
\begin{align*}
\simplex_{\eb} &\coloneqq \PP\big(\big\{ x=\sum_{i=1}^d x_i e_i \bigm| x_i \in \F, \, x_i \geq 0 \textrm{ for all } i=1,\ldots,d\big\}\big),
\end{align*}
and we denote by $\simplex^o_\eb$ its interior in $\PP V$.
To simplify for the introduction,  we suppose from now that $\Lambda=\R$,  
but the general statements are given and proven later on.
As in the real case \cite{delaHarpe_HilbertMetricSimplices},  the map that assigns to a point in $\simplex^o_\eb$ the logarithm of the absolute value of its coordinates,  induces an isometry from $(X_{\simplex^o_\eb},d_{\simplex^o_\eb})$ to $(\Aa,\dhex)$,  see \Cref{propo:Simplex}.
In fact,  in any affine chart containing $\simplex_\eb$,  every simplex of the barycentric subdivision of $\simplex^o_\eb$ is mapped to the closed Weyl chamber $\Aap \coloneqq \{ [\alpha] \in \R^d/\R(1,\ldots,1) \mid \alpha_1\geq \ldots\geq \alpha_d \} \subset \Aa$ of the root system $A_{d-1}$.

\begin{figure}[H]
\begin{tikzpicture}[scale=2]
\filldraw[fill=gray!40]  (1/2,0) --({1/2*cos(60)},{1/2*sin(60)})-- ({1/2*cos(120)},{1/2*sin(120)})--(-1/2,0)-- (-{1/2*cos(60)},-{1/2*sin(60)})--(-{1/2*cos(120)},-{1/2*sin(120)})-- cycle;
\fill [black!40,path fading=north,fading transform={rotate=-60}] (0,0)--(1,0) --({cos(60)},{sin(60)})-- cycle;
\draw (1/2,1/2) node{$\Aap$};
\draw (-1,0)--(1,0);
\draw (-{cos(60)},-{sin(60)})--({cos(60)},{sin(60)});
\draw (-{cos(120)},-{sin(120)})--({cos(120)},{sin(120)});
\end{tikzpicture}
\caption{The model flat $(\Aa,\dhex)$ with its unit ball and the closed Weyl chamber $\Aap$ for $d=3$.}
\label{fig:ModelFlat}
\end{figure}
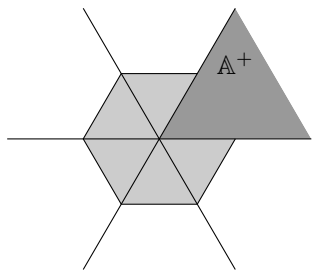

\noindent 
This allows for example to see that when $\Omega$ 
is 
the projective model of the symmetric space of $\SL(n,\F)$, 
namely the convex cone of symmetric positive definite matrices, 
then $X_\Omega$ 
can be equivariantly identified 
with  the affine building of $\SL(n,\F)$,
see \Cref{propo:SymSpace}. 

\subsection{Ultralimits of Hilbert geometries} 
\label{subsec:IntroUlim}
Our first main result motivates the study of such non-Archimedean Hilbert geometries as we prove that they allow to describe ultralimits of rescaled
real Hilbert geometries.
Roughly speaking,  taking an ultralimit of metric spaces
is a geometric construction
generalizing the pointed Gromov--Hausdorff convergence, that
associates to a sequence of pointed metric spaces $(X_n,d_n,o_n)$ a
pointed metric space $(X_\omega,d_\omega,o_\omega)$ constructed using a non-principal
ultrafilter $\omega$ on $\NN$.
More precisely, 
$X_\omega$ is the quotient space
\[X_\omega \coloneqq \{x=(x_n)_{n\in \N} \in \prod_{n\in \N} E_n \colon \lim_\omega d_n(o_n,x_n) < \infty\}_{/\sim},\]
where $x\sim y$ if and only if $\lim_\omega d_n(x_n,y_n)=0$, 
endowed with the metric $d_\omega$ induced by the pseudo-metric 
$d_\omega(x,y) \coloneqq \lim_\omega d_n(o_n,x_n)$.
The class in $X_\omega$ of a sequence $(x_n)_{n \in \N}$ will be called its
ultralimit (in $X_\omega$) and denoted by $\ulim_{X_\omega}x_n$.
For precise definitions and more details, see
\Cref{subsection:UltralimitsMetricSpaces}.

Let $V$ be a finite-dimensional real vector space and $(\Omega_n)_{n \in \N}$ a
sequence of non-empty open properly convex subsets in the projective space $\PP V$.
Let $o_n \in \Omega_n$ and $(\lambda_n)_{n \in \N}$ a scaling sequence,
namely a sequence  of real numbers  such that $\lambda_n\geq 1$ and $\lambda_n \to \infty$. 
Let $\omega$ be a non principal ultrafilter on $\NN$.
Let $X_\omega$ be the ultralimit of the sequence of pointed metric spaces
$(\Omega_n,\tfrac{1}{\lambda_n} d_{\Omega_n},o_n)$, where $d_{\Omega_n}$ is the Hilbert
metric on $\Omega_n$.
Let $\PVom$  be the asymptotic cone of the projective space $\PP V$
for the scaling sequence $(\lambda_n)_{n \in \N}$, namely the ultralimit of $\PP V$
with respect to the metric $d^{1/_{\lambda_n}}$ where $d$ is the spherical
metric, 
see \cite{Parreau_InvariantWeaklyConvexCocompactSubspacesSurfaceGroupsA2Buildings}.
It can be identified with the projective space over $\RRom$
 of an $\RRom$-vector space,  
where $\RRom$ is a \emph{Robinson field}---an ordered valued and non-Archimedean field that can be described
 as the ultralimit of the sequence of metric spaces
 $(\R,\abs{\cdot}^{1/\lambda_n},0)$, see \Cref{subsection:RobinsonField}.
As the ultralimit $\ulim \Omega_n \coloneqq \{\ulim_{\PVom}x_n \mid  (x_n)_{n\in\N} \in \prod_{n \in \N}\Omega_n\}$
 is always a closed subset of $\PVom$,  
we consider its interior $\Omega_\omega$ in $\PVom$ in order to define a Hilbert pseudo-distance that is finite.
The set $\Omega_\omega$ is an open convex subset of $\PVom$ (possibly empty).
Let $\oom=\ulimPVom{o_n} \in \ulim \Omega_n$ be the ultralimit of the sequence $(o_n)$ in $\PVom$. 
With this setup we can now state our first result.

\begin{theoremIntro}[\Cref{thm:UltralimitsHilbertGeometries}]
\label{thm:Intro:UltralimitsHilbertGeometries}
Suppose that $\ulim \Omega_n$ is contained in an affine chart
and $\oom \in \Omega_\omega$. 
\begin{enumerate}
\item 
There is an isometry $\overline{\Phi}\from \Xom \to X_{\Omega_\omega}$, 
satisfying
$\overline{\Phi}(\ulimXom x_n)=\overline{\ulimPVom x_n}$.
\item 
Let $\Gamma$ be a finitely generated group, and $\rho_n \from \Gamma\to\PGL(V)$ be a representation preserving $\Omega_n$ for each $n \in \N$.
Suppose that for every $\gamma$ in a finite symmetric generating set of $\Gamma$, we have
\[\lim_\omega\, \norm{\rho_n(\gamma)}^{1/\lambda_n} < \infty \]
for any fixed norm $\norm{\cdot}$ on $\End(V)$.
Then $(\rho_n)_{n \in \N}$ induces a well-defined isometric action of $\Gamma$ on $\Xom$ and on $X_{\Omega_\omega}$, for which the map $\overline{\Phi}$ is equivariant.
\end{enumerate}
\end{theoremIntro}

To prove this theorem,  we analyze ultralimits of convex sets,  their boundaries,  and how ultralimits of projective lines intersect those.
The two conditions in the theorem allow to avoid pathological behavior,
for example 
a sequence of scales that grows too fast or too slow in comparison with the convex sets $\Omega_n$, 
or with the distance between the observation
point $o_n$ and the boundary of $\Omega_n$. 
Assuming this does not happen,  we show that the boundary of the
ultralimit is the ultralimit of the sequence of boundaries
(\Cref{prop:ultralimitsOfConvex:boundary}),  as well as that
ultralimits are \emph{well-bordered},  
meaning that in any affine chart 
the intersection with a projective line is either contained in the
boundary or a closed interval (\Cref{prop-UltralimitConvexAndLines}).

\subsection{Degenerations of convex projective structures on manifolds}
Motivated by the study of degenerations of convex projective structures on manifolds,  we use \Cref{thm:Intro:UltralimitsHilbertGeometries} to associate to a closed point in the boundary of the real spectrum compactification of the space of convex real projective structures a properly convex open subset of the projective space over a Robinson field.

Let $M$ be a closed connected oriented topological manifold of dimension $d-1 \geq 2$ with hyperbolic fundamental group $\Gamma=\pi_1(M)$.
Denote by $\PS(M)$ the space of marked convex real projective structures on $M$.
Via the holonomy representations,  the space $\PS(M)$ can be identified with a subset of the $\PSL(d,\R)$-character variety of reductive representations modulo conjugation
\[\mathcal{X}(\Gamma, d) \coloneqq \Hom_\textnormal{red}(\Gamma,\PSL(d,\R))/\PSL(d,\R).\]
In fact,  $\PS(M)$ is a union  of connected components of $\mathcal{X}(\Gamma,d)$ \cite{Koszul_DeformationsConnexionsLocalementPlates, Benoist_ConvexesDivisiblesIII},  and thus a closed and semi-algebraic subset of $\mathcal{X}(\Gamma,d)$.
As such $\PS(M)$ admits a compactification using the real spectrum,  defined as the closure of $\PS(M)$ inside the compact Hausdorff space $\overline{\mathcal{X}(\Gamma,d)}^{\RSp,\cl}$; see \Cref{subsection:ConvexRealProjStr} or \cite{BurgerIozziParreauPozzetti_RSCCharacterVarieties2} for more details.
We write $\partial^{\RSp} \PS(M)\coloneqq \overline{\PS(M)} \setminus \PS(M)$ for its boundary points.
Points in $\overline{\mathcal{X}(\Gamma,d)}^{\RSp,\cl}$ can be \emph{represented} by representations $\rho \from \Gamma \to \PSL(d,\F)$,  where $\F$ is a real closed field containing $\R$ (see \Cref{definition:RepresentingPtsinRSC}).

\begin{theoremIntro}[\Cref{thm:DegenerationsConvProjStr}]
\label{thm:Intro:DegenerationsConvProjStr}
Let $\alpha \in \partial^{\RSp} \PS(M)$.
Then there exist a non-Archimedean,  real closed valued field $\F$ and a representation $\rho \from \Gamma \to \PSL(d,\F)$ representing $\alpha$,  such that $\rho(\Gamma)$ preserves a non-empty,  open,  properly convex,  well-bordered (\Cref{dfn:WellBordered}) subset $\Omega \subset\PP(\F^d)$,  and $\Gamma$ acts on the associated Hilbert metric space $X_\Omega$ without global fixed point.
\end{theoremIntro}

The proof of this theorem is based on an accessibility result for the closed points in the boundary of the real spectrum compactification \cite[Theorem~7.16]{BurgerIozziParreauPozzetti_RSCCharacterVarieties2}.
It implies that there is a sequence $(\rho_n)_{n \in \N}$ of real representations whose ultralimit $\rho_\omega$ is a representation to $\PSL(d,\RRom)$,  where $\RRom$ is a Robinson field,  that is conjugated to a representation representing $\alpha$.
Now every real representation $\rho_n$ preserves a properly convex domain $\Omega_n \subset \PP(\R^d)$, and we then define $\Omega$ as the interior in $\PP (\RRom^d)$ of the ultralimit of the sequence $(\Omega_n)_{n \in \N}$.
To prove that $\Omega$ is non-empty we use the Tarski--Seidenberg transfer principle \cite[\S 1.4]{BochnakCosteRoy_RealAlgebraicGeometry}.
To show that it is contained in an affine chart,  we construct a simplex containing $\Omega$ as the ultralimit of a sequence of simplices containing $\Omega_n$.
For the second part of the statement,  we relate the translation length to the eigenvalues of the representation $\rho_\omega$.

In order to investigate the metric properties of the Hilbert metric spaces associated with degenerations of convex projective structures,  in particular of those with positive systole,  the we study of polytopal Hilbert geometries,  as we expect them to give a model for the local structure of the Hilbert metric spaces in this case.

\subsection{Hilbert geometry of non-Archimedean integral polytopes}
For a general polytope in real projective space its Hilbert distance is Lipschitz-equivalent to a normed vector space,  and isometric to it if and only if the polytope is a simplex \cite{FoertschKarlsson_HilbertMetricsMinkowskiNorms,Bernig_HilbertGeometryPolytopes, ColboisVernicosVerovic_HilbertGeometryConvexPolygonalDomains, Vernicos_HilbertGeometryConvexPolytopes}.

When $\F$ be a non-Archimedean ordered valued field,
the next result gives a complete
description of the associated Hilbert metric space of an
\emph{integral polytope}, i.e.\ a bounded convex polytope in $\F^{d-1}\subset\PP(\F^d)$ 
whose vertices have coefficients in a subfield $\K$ of the valuation ring $\mathcal{O}=\{x \in \F \mid \abs{x}\leq 1\}$.

We have already seen as a first example that an open simplex is isometric to $(\Aa,\dhex)$,
and that every simplex of a barycentric subdivision is mapped isometrically to $\Aap$.

Let now $\Fpolyt$ be an integral polytope.
Denote by $\Model$ the \emph{geometric realization of the flag complex of $\Fpolyt$ modeled on $(\Aap,\dhex)$}.
More precisely,  $\Model$ is the polyhedral fan obtained by gluing together copies of $\Aap$ according to the combinatorial data of the flag complex of $\Fpolyt$ (or equivalently a barycentric subdivision of $\Fpolyt$) around a point $o \in \Model$,  called the \emph{cone point}.
For example if $d=3$,  $\Model$ is a fan of $2n$ Weyl chambers,  where $n$ is the number of vertices of $\Fpolyt$ (\Cref{fig:ModelMetricSpace6gon}),  and if $\Fpolyt$ is a simplex,  we recover $(\Aa, \dhex)$.
We endow $\Model$ with the distance induced by $\dhex$.
For more details on these definitions,  we refer to \Cref{subsection:DefinitionModelSpace}.

\begin{figure}[h]
\begin{tikzpicture}[scale=.7]
\fill[fill=gray!80] (0,0) --(-4,0) --  (-4,-1.5) -- cycle;
\fill[fill=gray!80] (0,0) --(3,-1) --  (3,-1.5) -- cycle;
\fill[fill=gray!80] (0,0) --(5,1) --  (4,1.5) -- cycle;
\draw (0,0)--(-4,-1.5);
\draw (0,0)--(3,-1.5);
\draw (0,0)--(5,1);
\fill[fill=gray!40] (-4,0)--(-3.8,0)--(-3.7,1.5)--(0.2,1)--(4,2) --(5,1)--(4,1.5)--(4.7,.1)--(3,-1.5)--(3,-1)--(1,-2.5)--(-1,-2.3)-- cycle;
\draw (0,0)--(-4,0);
\draw (0,0)--(-3.7,1.5);
\draw (0,0)--(0.2,1);
\draw (0,0)--(4,2);
\draw (0,0)--(4,1.5);
\draw (0,0)--(4.7,.1);
\draw (0,0)--(3,-1);
\draw (0,0)--(1,-2.5);
\draw (0,0)--(-1,-2.3);
\draw (0,0) node{$\bullet$};
\draw (-.15,0) node[above]{$o$};
\end{tikzpicture}
\caption{The geometric realization of the flag complex of a 6-gon.}
\label{fig:ModelMetricSpace6gon}
\end{figure}
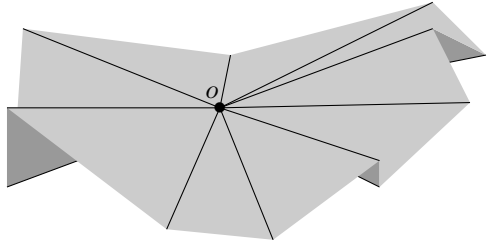

\begin{theoremIntro}[\Cref{thm:NAPolytopalHilbertMetricSpace}]
\label{thm:Intro:NAPolytopalHilbertMetricSpace}
Let $\Fpolyt \subset \F^{d-1} \subset \PP(\F^d)$ be an integral polytope over a subfield $\KK \subset \mathcal{O}$,  
and let $\Omega \subset \PP(\F^d)$ be its interior.
Then there exists an isometry
\[ \overline{\Psi} \from (X_\Omega,d_\Omega)\to (\Model,\dmodel).\]
Moreover $\overline{\Psi}$ sends the image of the  barycenter of $\Fpolyt$ to the cone point $o \in K$,  and,  if $S$ is a maximal simplex in the barycentric subdivision of $\Fpolyt$,  then $\overline{\Psi}$ maps $\pi(S\cap \Omega)$ isometrically onto $(\Aap,\dhex)$.
\end{theoremIntro}

The proof of the above theorem follows in its structure the proof of the Lipschitz-equivalence between real polytopal Hilbert geometries and normed vector space \cite{Vernicos_HilbertGeometryConvexPolytopes}.
However,  the local maps defined on the barycentric subdivision of $\Fpolyt$,  that in the real case are only Lipschitz-equivalences,  descend to local isometries in the non-Archimedean case,  that glue together to a global isometry $\overline{\Psi}$.
To prove this,  we use an approximation lemma of independent interest,  which is essentially due to the ultrametricity of the absolute value (\Cref{lem:NonArchimedeanPropertiesCR}).

\begin{lemmaIntro}[Flag simplex sandwich lemma, \Cref{lem:FlagSimplexSandwichLemma}]
\label{lem:Intro:FlagSimplexSandwichLemma}
Let $\Omega\subset \PP V$ be an open convex subset,  and $\simplex_\eb$, $\simplex_{\eb'}$ two simplices such that $\simplex_\eb \subset \overline{\Omega} \subset \simplex_{\eb'}$.
Let $M = \Mat_{\eb'}(\eb)$ be the matrix of $\eb$ in $\eb'$.
If $M$ is upper triangular and in $\GL(d,\mathcal{O})$,  then
\[d_\Omega = d_{\simplex^o_\eb} = d_{\simplex^o_{\eb'}} \textnormal{ on } (\simplex \cap \Omega) \times (\simplex \cap \Omega),  \]
where $S \coloneqq \PP(\{ x=\sum_{i=1}^d x_i e_i | x_1 \geq \ldots \geq x_d \geq 0 \})$.
In particular,  $\pi(\simplex \cap \Omega) \subset X_\Omega$ is isometric to $(\Aap,\dhex)$.
\end{lemmaIntro}

\begin{figure}[H]
\begin{tikzpicture}[scale=1.2]
\filldraw[fill=gray!40] (0,2) --(3,0) --  (5,3) -- cycle;
\draw (1,1.8) node{$\simplex_{\eb'}$} ;

\filldraw[fill=gray!30] (3,0)-- (4.5,2.25)  .. controls +(-.5,.2) and +(0,0) ..   (2.5,2.3)-- (1.5,1.875) .. controls +(-.1,-.2) and +(0,0) ..  cycle;
\draw (2,1.7) node{$\overline{\Omega}$} ;

\filldraw[fill=gray!20] (3,0) --(4.3,1.95) --  (2.3,1.8)-- cycle;
\draw (3.7,1.6) node{$\simplex_\eb$} ;

\filldraw[fill=blue!20] (3,0) -- (3.872,1.308) -- (2.86,1.625) -- cycle;
\draw (3.2,1) node{$\simplex$} ;

\draw[dotted] (3.872,1.308)--(2.3,1.8);
\draw[dotted] (4.3,1.95)--(2.41,1.52);
\draw[dotted] (2.86,1.625)--(2.83,1.85);
\end{tikzpicture}
\caption{Sandwiching $\overline{\Omega}$ between two simplices.}
\label{fig_SandwichLemmaProof}
\centering
\end{figure}
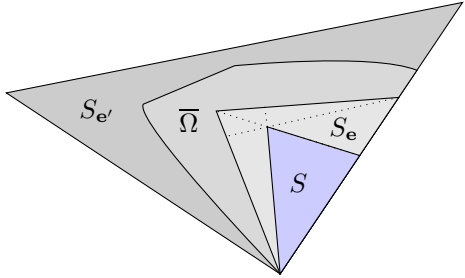

In other words,  locally on $S\cap \Omega$ we can approximate $\Omega$ by the interior of a simplex,  whose associated Hilbert metric space we understand well (\Cref{propo:Simplex}).
The condition $\Mat_{\eb'}(\eb) \in \GL(d,\mathcal{O})$ is for example satisfied if the vertices of the simplices $\Sin$ and $\Sout$ have coefficients in a subfield $\K \subseteq \mathcal{O}$.
Furthermore,  we prove in \Cref{lem:SandwichLem} that in the case that $\Omega$ is the interior of a polytope, there exist $\Sin$ and $\Sout$ that satisfy the assumptions of the above lemma,  and such that $\PP(\sum_{i=1}^d e_i)$ is a barycenter of the polytope.
In the case that the polytope is integral,  also the vertices of the simplices $\Sin$ and $\Sout$ have coefficients in the same subfield $\K \subseteq \mathcal{O}$,  and thus the condition $\Mat_{\eb'}(\eb) \in \GL(d,\mathcal{O})$ 
is satisfied.
This already proves that the local maps defined in the proof of \Cref{thm:Intro:NAPolytopalHilbertMetricSpace} descend to local isometries.

Combining the results of \Cref{thm:Intro:UltralimitsHilbertGeometries} and \Cref{thm:Intro:NAPolytopalHilbertMetricSpace} allows us to understand asymptotic cones of real polytopal Hilbert geometries.

\subsection{Asymptotic cones of real polytopal Hilbert geometries}
The goal is to describe asymptotic cones of a fixed real Hilbert geometry using non-Archimedean Hilbert geometries.
An asymptotic cone of a metric space $(X,d)$ 
are the ultralimit of $(X,\frac{1}{\lambda_n}d,o_n)$ 
for a scaling sequence $\lambda_n \to\infty$ and a sequence of observation points
$o_n \in X$, with respect to some ultrafilter $\omega$.
For example,  when $X$ is hyperbolic space,  then
\Cref{thm:Intro:UltralimitsHilbertGeometries}
gives a
new proof of the classical result that the asymptotic cone of
hyperbolic space is a real tree.
Similarly, in the case where $X$
is the symmetric space of $\SL(n,\RR)$,  
endowed with the Hilbert metric on the convex projective model,
then
\Cref{thm:Intro:UltralimitsHilbertGeometries} gives a
new proof that the asymptotic cone of $X$ is
the affine building of $\SL(n,\cdot)$ over a Robinson field 
\cite{KlLe97, Parreau_CompEspReprGroupesTypeFini}.

For real polytopal Hilbert geometries,  putting together \Cref{thm:Intro:UltralimitsHilbertGeometries} and \Cref{thm:Intro:NAPolytopalHilbertMetricSpace}, we obtain a complete description of the asymptotic cones of a fixed real polytope.

\begin{theoremIntro}[\Cref{thm:AsymptoticConesPolytopalHilbertGeometries}]
\label{thm:Intro:AsymptoticConesPolytopalHilbertGeometries}
Let $\polyt \subset \RR^{d-1} \subset\PP(\RR^d)$ be a bounded real convex polytope with non-empty interior $\Omega\subset \RR^{d-1} \subset\PP(\RR^d)$.
Let $\omega$ be a non-principal ultrafilter on $\NN$,  and let $(X_\omega,d_\omega,o_\omega)$ be the asymptotic cone of $(\Omega, d_\Omega)$ for the scaling sequence $(\lambda_n)_{n \in \N}$ and fixed observation points $o_n=x_0\in \Omega$.
Then $(X_\omega,o_\omega)$ is isometric to $(\Model,o)$.
\end{theoremIntro}

Since $\Omega$ and $\Model$ are proper metric spaces, this implies that we
have in fact Gromov--Hausdorff convergence (see \Cref{prop:UltralimitsAndGH}):

\begin{corollaryIntro}
\label{corol:AsymptoticConesPolytopalHilbertGeometries:GHConv}
Let $\polyt \subset \RR^{d-1} \subset\PP(\RR^d)$ be a bounded real convex polytope with non-empty interior $\Omega\subset \RR^{d-1} \subset\PP(\RR^d)$, 
and let $(\lambda_n)_{n \in \N}$ be a scaling sequence.
Then $(\Omega,\tfrac{1}{\lambda_n} d_{\Omega},x)_{n \in \N}$ is converging to $(\Model,o)$ in the pointed Gromov--Hausdorff topology.
\end{corollaryIntro}

\subsection{Structure of the paper}
The article is organized as follows.
In \Cref{section:Preliminaries} we recall background on ordered valued fields,  convex sets in projective spaces over such fields and the cross-ratio.
\Cref{section:HilbertGeometry} introduces the Hilbert pseudo-distance and the associated Hilbert metric space,  and we prove basic properties of $d_\Omega$.
We investigate families of examples of such metric spaces,  e.g.\ simplices and symmetric spaces over non-Archimedean real closed fields.

We then turn to the proofs of the main results.
We first prove the results on non-Archimedean integral polytopes in \Cref{section:Polytopes},  and we postpone the proofs of the results on ultralimits to the last section of this article.
This section is,  except maybe for \Cref{subsection:AsymptoticConeRealPolytope},  independent of \Cref{section:Polytopes},  which allows for an independent study.

In \Cref{section:Polytopes},  we develop the theory of polytopes over ordered valued fields and show \Cref{lem:Intro:FlagSimplexSandwichLemma}.
We use this to prove \Cref{thm:Intro:NAPolytopalHilbertMetricSpace} in \Cref{subsection:ProofHilbertMetricSpacePolytope}.
\Cref{section:UltralimitsHilbertGeometry} explains how to view ultralimits of real Hilbert geometries as non-Archimedean Hilbert geometries.
We recall background on ultralimits in Sections \ref{subsection:UltrafiltersProducts}-\ref{subsection:UlimInLinearGroup},  which can be omitted by experts.
We show \Cref{thm:Intro:UltralimitsHilbertGeometries} in \Cref{subsection:ProofUlimHilbMetrics}.
We end this article with the proofs of \Cref{thm:Intro:AsymptoticConesPolytopalHilbertGeometries}
in \Cref{subsection:AsymptoticConeRealPolytope} and \Cref{thm:Intro:DegenerationsConvProjStr} in \Cref{subsection:ConvexRealProjStr}.

\addtocontents{toc}{\SkipTocEntry}
\subsection*{Acknowledgments}
The authors thank the IHES and the Institut Fourier for excellent
working conditions,  where a lot of this work was undertaken, and
Fr\'ed\'eric Paulin for his helpful comments.
The first author expresses her gratitude to Jacques Audibert,  Marc Burger,  Fanny Kassel,  and Nicolas Tholozan for their interest in this work and fruitful discussions.
The second author thanks Daryl Cooper and Kelly Delp for interesting 
discussions.

Both authors were  partially supported by 
ANR-23-CE40-0012 HilbertXfield, 
ANR-22-CE40-0001 GAPR,
and 
ANR-23-CE40-0001 GALS.
X.F.\ is funded by the Max-Planck Institute for Mathematics in the Sciences and the Labex Carmin project.
A.P.\ is partially supported by 
ANR-22-CE40-0004 GoFR.

\tableofcontents

\section{Preliminaries}
\label{section:Preliminaries}

\subsection{Valued fields}
\label{subsection:NAValuedFields}
\label{dfn_Balls}
A \emph{valued field} $\F$ is a field together with a  multiplicative and sub-additive map $\abs{\cdot}\from \F \to \R_{\geq 0}$
  that only maps $0$ to $0$, called the \emph{absolute value}.
For an introduction to valued fields we recommend \cite{EnglerPrestel_ValuedFields}.
\compl{From the definition we immediately obtain that 
$|1|=1$ and $|x^{-1}|=|x|^{-1}$, 
thus $|\cdot| \from \F^{\times} \to \R_{>0}$ is a group homomorphism.
We also have that if $\xi \in \F^{\times}$ is a root of unity then $|\xi|=1$, hence $|-x|=|x|$ for all $x\in \F$.}%

\begin{example}
\label{exa:ValuedFields}
Any field $\F$ can be endowed with the \emph{trivial} absolute value, setting $|x|=1$ for all $0 \neq x \in \F$.
Non trivial examples of valued fields are $\R$ with its usual absolute value and $\Q_p$ with the $p$-adic absolute value.
If $\K$ is any field,  we can define an absolute value on the field of rational functions $\K(X)$ by setting $\left|\frac{P(X)}{Q(X)}\right| \coloneqq e^{\deg P - \deg Q}$ for every $P,Q \in \K[X]$.
\end{example}

The absolute value induces a metric structure on $\F$ by setting $d(x,y) \coloneqq |x-y|$ for all $x,y \in \F$.
The resulting topology turns $\F$ into a Hausdorff topological field, 
which is discrete if and only if the absolute value is trivial.
Note that in general valued fields are not locally compact.
For example $\K(X)$ is locally compact if and only if $\K$ is finite.
%
For $x \in \F$ and $r \in \R_{\geq 0}$ we write
\[ B(x,r) \coloneqq \{ y \in \F \mid |x-y|<r \}, \, \overline{B}(x,r) \coloneqq \{ y \in \F \mid |x-y| \leq r \}.\]
for the \emph{open} and \emph{closed balls} of radius $r$ centered at $x$.
%

We say that $\abs{\cdot}$ is \emph{non-Archimedean} if 
the \emph{ultrametric} triangle inequality holds, 
namely  when $|x+y| \leq \max\{|x|,|y|\}$ for all $x, y \in \F$.
%
For example, the fields $\QQ_{p}$ and $\KK(X)$ are non-Archimedean.
An easy but important consequence  is that 
$|x+y|= \max\{|x|,|y|\}$ whenever $|x|\neq |y|$.
Note that non-Archimedean absolute values are trivial on the prime subfield, namely the subfield generated by $1$.

For a non-Archimedean valued field $\F$,  we denote by $\mathcal{O}\coloneqq \{x \in \F \mid \abs{x} \leq 1 \}$ its \emph{valuation ring} and by $\mathcal{O}^\times\coloneqq \{x \in \F \mid \abs{x}=1 \}$ the set of units in $\mathcal{O}$.
The additive subgroup $\log \abs{\F^\times}$ of $\R$ is called the
\emph{value group} of $\F$ and will be denoted by $\Lambda$.
The absolute value is called \emph{discrete} if $\Lambda$ is discrete.
In the case $\F=\Q_p$ we have $\mathcal{O} = \Z_p$ and $\Lambda=\Z$.
In the case $\F=\K(X)$ we have $\mathcal{O} =\left\{\frac{P(X)}{Q(X)} \in \K(X) \mid
  \deg P \leq \deg Q\right\}$ and again $\Lambda=\Z$, so these two examples
have discrete absolute value.
Examples of non-Archimedean fields with non-discrete absolute value (in fact, with full value group $\Lambda=\RR$) are given by Robinson fields,  which will be described in \Cref{subsection:RobinsonField}.

If $\abs{\cdot}$ is non-Archimedean,  the balls $B(x,r)$ and
$\overline{B}(x,r)$ are both open and closed 
(in particular  $\overline{B}(x,r)$ is not the topological closure of $B(x,r)$).
As a consequence, $\F$ is totally disconnected.
Furthermore, in the non-Archimedean setting,  any point in a ball serves as center,  and hence,  if two balls intersect, one is contained in the other.

\subsection{Ordered and real closed fields}
\label{K(X)defOrder}
We refer to \cite[\S 1]{BochnakCosteRoy_RealAlgebraicGeometry} for a
detailed introduction to ordered and real closed fields,  and for their properties recalled below.
%
An \emph{ordered field} is a field together with a total order compatible with the field operations.
%
The field of real numbers is ordered.
The $p$-adic numbers cannot be ordered.
If $\KK$ is any ordered field,  then we can define an order on $\K(X)$ extending the order of $\K$ by defining $X>\lambda$ for all $\lambda \in \KK$.
Note that an ordered field has characteristic zero and hence any ordered field contains a copy of $\Q$.
Given an ordered field $\F$ and $a < b \in \F \cup \{\pm\infty\}$,  the sets
\begin{align*}
]a,b[ &\coloneqq \{x \in \F \mid a < x<b\}, \,]a,b] \coloneqq \{x \in \F \mid a < x\leq b\},\\
[a,b[ &\coloneqq \{x \in \F \mid a \leq x<b\},\,[a,b] \coloneqq \{x \in \F \mid a \leq x\leq b\}
\end{align*}
are called \emph{segments}.

The order on $\F$ defines the \emph{order topology}, which has a subbasis of segments $(a,\infty)$,  $(-\infty,b)$ for all $a,b \in \F$.
We have $\Int{[a,b]} = ]a,b[$,  $\overline{]a,b[}=[a,b]$ and hence $\partial[a,b]=\{a,b\}$,  
where $\Int{X}$,  $\overline{X}$ and $\partial X$ denote the interior respectively the closure respectively the boundary of subset $X$ of a topological space.

We write $\F_{\geq 0}$ and $\F_{>0}$ for the subsets of non-negative respectively positive elements in $\F$.
The \emph{$\FF$-valued absolute value} is defined by $|x|_\F \coloneqq \max\{x,-x\}>0$ for all $x\in \F$.

\begin{definition}
An ordered field is \emph{real closed} if there is no algebraic field extension to which the order extends,  or equivalently if every positive element has a square root and every odd degree polynomial has a root.
\end{definition}
The order on $\F$ is then unique, positive elements being given by squares.
Every ordered field $\K$ has a (unique up to isomorphism) \emph{real closure}, i.e.\ an algebraic field extension of $\K$ that is real closed and whose order extends the order on $\K$ \cite[Theorem 1.3.2]{BochnakCosteRoy_RealAlgebraicGeometry}.
Thus up to passing to an algebraic extension we can restrict ourselves to real closed fields.

\begin{examples}
First examples of real closed fields are $\RR$, and the field $\overline{\Q}^{r}$ of real algebraic numbers, which is the real closure of $\Q$.
For $\K$ an ordered field,  the field $\K(X)$ endowed with the order defined above is not real closed,  since $X$ is positive but $X$ does not have a square root.
Its real closure may be seen as a subfield of the field of real Puiseux series \cite[Example 1.2.3]{BochnakCosteRoy_RealAlgebraicGeometry},  which is real closed.
Further examples of real closed fields are the hyperreals \cite[\S 3]{Goldblatt_LecturesHyperreals} and Robinson fields;  see \Cref{subsection:RobinsonField}.
\end{examples}

As a final note we remark that an ordered field with the least upper
bound property is order isomorphic to $\R$.
Thus in an ordered field different from $\R$,  we do not have the tool of supremum at hand.

\subsection{Ordered valued fields}
\label{subsection:NAOrderedFields}
%
An \emph{ordered valued field} is an ordered field $(\F,\leq)$ 
together with an absolute value $\abs{\cdot}$
which is \emph{order-compatible}, 
namely such that  for all $x,y \in \F$
\[ 0 \leq x \leq y \implies |x| \leq |y|.\]
If  $(\F,\leq)$ is additionally real closed we will call it a
\emph{real closed valued field}.
Note that in this case, the value group $\Lambda=\log |\F^\times|$ always contains
$\Q$,  in particular it is dense in $\R$.

Examples of ordered valued fields are given by $\RR$ and its subfields, $\K(X)$ for any ordered field $\KK$ (with the order and the absolute value defined above in \Cref{K(X)defOrder} and \Cref{exa:ValuedFields}), and Robinson fields.

Let us now state  an important observation concerning the interplay of the order and the absolute value.

\begin{proposition}
\label{lem:AbsoluteValueNA:EqualityPos}
Let $\F$ be an ordered valued field and $x, y \in \F_{\geq 0}$.
\begin{enumerate}
\item If $|x|<|y|$,  then $x<y$.
\item If $\F$ is non-Archimedean,  then $|x+y|= \max\{|x|,|y|\}$.\qed
\end{enumerate}
\compl{
\begin{proof}
(1): If no then  $0\leq y\leq x$ and $|y| \leq |x|$.
(2): We already know equality if $|x| \neq |y|$, so let us assume the two values are equal.
From the ultrametric triangle inequality we already know $\leq$.
Assume $|x+y|<|x|=|y|$.
Since $y \geq 0$ we have $ x \leq x+y$, and since $|\cdot|$ is order compatible, we obtain $|x| \leq |x+y| <|x|$, a contradiction.
\end{proof}
}
\end{proposition}

As a consequence we have the following interplay between balls and segments in an ordered valued field $\F$  
\[B(x,|a|) \subset (x-a, x+a) \subset \overline{B}(x,|a|)\]
for all $x,a \in \F$.
%
In particular, if the absolute value $\abs{\cdot}$ is non-trivial,   the order topology and the topology induced by the absolute value agree.
\compl{Since $|F|$ has then arbitrary small elements.}

\subsection{Vector and projective spaces}
\label{subsection:VectorSpacesNAOrderedFields}
Let $\F$ be a valued field and $V$ an $\F$-vector space of dimension $d$.
We fix a basis  $(e_1,\ldots,e_d)$ of $V$.
%
Then we can endow $V$ with the associated 
\emph{$l^\infty$-norm} $\|\cdot\|$ defined by
\[\|x\| \coloneqq \max_{1\leq i \leq d}\{|x_i|\}  \in \R_{\geq 0}, \, \textrm{ for
  } x=\sum_ix_ie_i \textrm{ with } (x_1,\ldots,x_d) \in \F^d, \]
\compl{This is a standard norm}
and for $x, y \in V$ we write
$|xy| \coloneqq \|x-y\| \in \R_{\geq 0}$ for their ($\R$-valued) distance.
\compl{this is a usual metric}%

As in \Cref{subsection:NAValuedFields} we define for $x \in V$ and $r \in \R_{\geq 0}$ the \emph{open} and \emph{closed balls} of radius $r$ centered at $x$ as
\[ B(x,r) \coloneqq \{ y \in V \mid |xy| < r \}, \, 
\overline{B}(x,r) \coloneqq \{ y \in V \mid |xy| \leq r \}.\]
The open balls form a subbasis for the (product) topology on $V$.
Note that the topology does not depend on the choice of basis.

When $\F$ is additionally ordered (compatible with $\abs{\cdot}$),  we also define the \emph{$\F$-valued $l^\infty$-norm} $\|\cdot\|_\F$ on $V$ 	by
\[\|x\|_\F \coloneqq \max_{1\leq i \leq d}\{|x_i|_\F\} \in \F_{\geq 0},\]
and associated \emph{$\F$-valued distance} $d_\F(x,y) \coloneqq \|x-y\|_\F \in \F_{\geq 0}$.
An \emph{$\F$-valued distance} is any function $d \from \F\times \F \to \F_{\geq 0}$ that satisfies $d(x,y)=0$ if and only if $x=y$,  $d(x,y )=d(y,x)$ and
$d(x,y) \leq d(x,z) +d(z,y)$ for all $x,y,z \in \F$.
An \emph{$\F$-valued norm} is any function $N \from V \to \F_{\geq 0}$ that
satisfies $N(v)=0$ if and only if $v=0$,  $N( \alpha v )=|\alpha|_\F N(v)$ and
$N(v+w) \leq N(v) +N(w)$ for all $v,w \in V$ and $\alpha \in \F$.
This generalizes the $\F$-valued absolute value $\abs{\cdot}_\F$ on $\F$.
For any $\F$-valued norm $N$ on $V$,  we define for $R \in \F_{\geq 0}$ the \emph{open and  closed $N$-balls} of radius $R$ centered at $x$ as
\begin{align*}
 B_{N}(x,R) \coloneqq \{ y \in \F \mid N(x-y)<R \}, \\
\overline{B}_{N}(x,R) \coloneqq \{ y \in \F \mid N(x-y) \leq R \},
\end{align*}
which generalize the notion of segments in dimension one.
Note that $B_{\|\cdot\|_\F}(0,R)=[-R,R]^d$
and that $\|\cdot\|_\F$ refines the $\R$-valued norm in the following sense: 
we have $\|x\|= |\|x\|_\F|$ and $|xy|=|d_\F(x,y)|=|\|x-y\|_\F|$.
%
A subset $A$ of $V$ is \emph{bounded} 
if there exists $R \in \F$ such that $A \subset [-R,R]^d$, or equivalently if $A$ is contained in some ball $A \subset \overline{B}(x,r)$ with $x \in V$ and $r \in \R_{\geq 0}$.

Let $\F$ be a field and $V$ a finite-dimensional $\F$-vector space.
We denote by $\PP V$ its projectivization, i.e.
\[ \PP V \coloneqq (V \setminus \{0\})/\F^{\times},\]
endowed with the quotient topology.
We write $\PP$ for the projection map $V\setminus \{0\} \to \PP V $.
We set $V^*$ to be its dual vector space, and $\PP V^*$ its dual projective space.
The projective general linear group $\PGL(V)$ naturally acts on $\PP V$.

If $\eb=(e_1,\ldots,e_d)$ is a basis of $V$, the \emph{associated affine chart} is the map
\[
\function{f_\eb}{\F^{d-1}}{\PP V}%
{(x_1 , \ldots , x_{d-1})}{\PP\big(\,\sum_{i=1}^{d-1} x_ie_i +e_d \,\big).}
\]
Given a basis $\eb =(e_1,\ldots,e_d)$ of $V$, when   $(x_1,\ldots,x_d)\in \FF^d$
are \emph{homogeneous coordinates} 
of $x$ in the basis $\eb$, namely when $\PP(\sum_{i=1}^{d} x_ie_i
) = x$, we write $x = [x_1:\ldots:x_d]$.

We now define the dual of a subset in $\PP V$.
Note that the definition slightly differs from the usual one in the
real case.
This allows to be consistent with \cite{FalbelGuillouxWill_HilbertGeometryWithoutConvexity, FalbelGuillouxWill_HilbertMetricBoundedSymmetricDomains}.

\begin{definition}
Let $\Omega \subseteq \PP V$ be any subset.
We define its \emph{dual set} $\Omega^* \in \PP V^*$ by
\[\Omega^*\coloneqq \PP \left(\{ \tilde{\varphi} \in V^* \mid \tilde{\varphi}(\tilde{x}) \neq 0 \textrm{ for a lift } \tilde{x} \textrm{ of all } x \in \Omega \} \right) \subset \PP V^*.\]
\end{definition}

Note that $\Omega$ is contained in an affine chart if and only if $\Omega^*$ is non-empty.

Assume now that $\F$ is ordered.
As in the real case we then have the following.
\begin{proposition}
\label{lem:ProperAffineChart}
If $\Omega \subseteq \PP V$ is bounded in an affine chart,  then $\overline{\Omega}$ is contained in an affine chart.
\end{proposition}
\begin{proof}
Pick a basis $\eb$ of $V$.
Up to acting by an element of $\PGL(V)$ we can assume that $\Omega$ is contained in the associated affine chart $f_\eb \from \F^{d-1} \to \PP V$,  $(x_1,\ldots,x_{d-1}) \mapsto \PP(\,\sum_{i=1}^{d-1} x_ie_i +e_d),$
i.e.\ $\Omega \subseteq f_\eb(\F^{d-1})$.
Since $\Omega$ is bounded,  there exists $R \in \F_{>0}$ such that $\Omega \subset f_\eb([-R,R]^{d-1})$.
Now $f_\eb([-R,R]^{d-1}) \subset \PP V$ is closed,  since its preimage
\begin{align*}
\PP^{-1}(f_\eb([-R,R]^{d-1}) = \{(x_1,\ldots,x_d) &\in V\setminus \{0 \} \mid -Rx_d \leq x_i \leq Rx_d \\
& \textnormal{ for all } i=1,\ldots,d-1 \}
\end{align*}
is closed in $V \setminus \{0 \}$.
Thus $\overline{\Omega} \subset f_\eb([-R,R]^{d-1})\subset  f_\eb(\F^{d-1})$.
\end{proof}

\subsection{The cross-ratio}

Let $\F$ be a field.
\begin{definition}
Let $a,x,y$ and $b \in \F \cup \{\infty\}$,  and assume that no three of the four points agree.
We define their \emph{cross ratio} $\CR{\cdot,\cdot,\cdot,\cdot}$ by
\[ \CR{a,x,y,b} = \frac{y-a}{x-a} \cdot \frac{x-b}{y-b} \in \F \cup \{\infty\},\]
with the conventions $\tfrac{0}{0}=1$,  $\tfrac{\infty}{\infty}=1$, $\tfrac{z}{0}=\infty$ and $\tfrac{z}{\infty}=0$ for all $0\neq z \in \F$.
\end{definition}
In the above convention we have $\CR{0,1,z,\infty}=z$ for any $z \in \F \cup \{\infty\}$.
If $L$ is a projective line,  we can define the cross ratio of four
points on $L$ using some (any) identification of $L$ with $\F \cup \{\infty\}$ by an
affine chart.

Assume now that $\F$ is an ordered field.
The order on $\F$ induces a natural cyclic order on $\F \cup \{\infty\}$.
This does not induce a cyclic order on any projective line $L$,  but there is a well-defined notion of positively oriented tuples on $L$.
Given a projective line $L$ we say that four points $a,x,y,b \in L$ are \emph{positively oriented} if in any affine chart $(a,x,y,b)$ or $(b,y,x,a)$ is cyclically ordered.
As in the real case we have the following monotonicity property.

\begin{proposition}[Monotonicity of the cross ratio]
\label{lem:PropertiesCR:Monotonicity} 
Let $a' \leq a \leq x \leq x' \leq y' \leq y \leq b \leq b' \in \F\cup \{\infty\}$ be such that the following two cross ratios are defined.
Then
\[ \CR{a',x',y',b'} \leq \CR{a,x,y,b}.\]
\end{proposition}

Let us additionally assume that $\F$ is endowed with an order-compatible non-Archimedean absolute value $|\cdot| \from \F \to \R$.
We now list some of the special properties of the absolute value of the cross-ratio in the setting of non-Archimedean ordered fields.

\begin{proposition}
\label{lem:NonArchimedeanPropertiesCR}
Let $a' \leq a \leq x \leq y \leq b \leq b' \in \F \cup \{ \infty\}$ be such that all of the following cross ratios are defined.
\begin{enumerate}
\item 
\label{lem:NonArchimedeanPropertiesCR:CRb}
If $|xy| \leq |xa|$, then $\abs{\CR{a,x,y,b}}=\tfrac{|xb|}{|yb|}$.
This is for example satisfied if $|xb|\leq |xa|$.
\item
\label{lem:NonArchimedeanPropertiesCR:CRa}
If $|xy| \leq |yb|$, then $\abs{\CR{a,x,y,b}}=\tfrac{|ya|}{|xa|}$.
This is for example satisfied if $|ya|\leq |xa|$.
\item
\label{lem:NonArchimedeanPropertiesCR:CR=1}
If $|xy| \leq \min \{|xa|,|yb|\}$, then $\abs{\CR{a,x,y,b}}=1$.
\item
\label{lem:NonArchimedeanPropertiesCR:CR=}
If $|bb'| \leq |yb|$ and $|aa'|\leq |xa|$, then $\abs{\CR{a,x,y,b}}=\abs{\CR{a',x,y,b'}}$.
This is for example satisfied if $\abs{\CR{b',x,y,b}}=\abs{\CR{a,x,y,a'}}=1$.
\end{enumerate}
\end{proposition}
\begin{proof}
By definition of the cross ratio and the multiplicity property of $\abs{\cdot}$ we have 
\[\abs{\CR{a,x,y,b}}= \frac{|ya|}{|xa|}\cdot \frac{|xb|}{|yb|}.\]
Because the points $a,x,y,b$ are aligned in this order,  we have $|ya|=\max\{|yx|,|xa|\}$ and $|xb|=\max\{|yx|,|xb|\}$ by \Cref{lem:AbsoluteValueNA:EqualityPos}.
This immediately implies (\ref{lem:NonArchimedeanPropertiesCR:CRb}) and (\ref{lem:NonArchimedeanPropertiesCR:CRa}).
Putting these together gives (\ref{lem:NonArchimedeanPropertiesCR:CR=1}).
A similar argument proves (\ref{lem:NonArchimedeanPropertiesCR:CR=}) using the order-compatibility of $\abs{\cdot}$.
\end{proof}

Let $\F$ be again any field and $V$ a finite-dimensional $\F$-vector space.
The cross-ratio of four aligned points is well-defined.
We now also define the cross-ratio of two points and two hyperplanes.

\begin{definition}
Let $x,y \in \PP V$ and $\varphi, \psi \in \PP V^*$, and assume that neither one of the following occurs: $x=y \in \ker(\varphi)$,  or $x=y \in \ker(\psi)$,  or $x \in \ker(\varphi)=\ker(\psi)$,  or $y \in \ker(\varphi)=\ker(\psi)$.
We define the \emph{cross ratio }$\CR{\cdot,\cdot,\cdot,\cdot}$ of two points and two hyperplanes by
\[\CR{\varphi,x,y,\psi} \coloneqq \frac{\tilde{\psi}(\tilde{x})}{\tilde{\psi}(\tilde{y})} \cdot \frac{\tilde{\varphi}(\tilde{y})}{\tilde{\varphi}(\tilde{x})} \in \F\cup \{ \infty \},\]
where $\tilde{x}, \tilde{y} \in V$ and $\tilde{\varphi}, \tilde{\psi} \in V^*$ are representatives of $x, y$ respectively $\varphi, \psi$.
\end{definition}

Note that this definition does not depend on the choice of lifts.
Like in the real case we have the following properties.

\begin{proposition} \label{lem:PropertiesCR}
Let $x, y,z \in \PP V$ and $\varphi, \psi,\alpha \in \PP V^*$ be such that all of the following cross ratios are defined.
Then
\begin{enumerate}
\item \label{lem:PropertiesCR:Multiplicativity}
\begin{align*}
\CR{\varphi,x,y,\psi} &= \CR{\varphi,x,z,\psi}\CR{\varphi,z,y,\psi}\\
&=\CR{\varphi,x,y,\alpha}\CR{\alpha,x,y,\psi}.
\end{align*}
\item \label{lem:PropertiesCR:Dim1}
Let $L$ be a projective line through $x$ and $y$.
If we set $a\coloneqq L \cap \PP(\ker(\varphi))$ and $b \coloneqq L \cap \PP(\ker(\psi))$,  then $\CR{\varphi, x,y,\psi} = \CR{a,x,y,b}$.
\item
\label{lem:PropertiesCR:InvProjTransf}
The cross ratio is invariant by projective transformations, namely for $g \in \PGL(V)$ we have $\CR{ga,gx,gy,gb}=\CR{a,x,y,z}$.
\qed
\end{enumerate}
\end{proposition}

Assume now that $\F$ is ordered.
In this case one can compute the cross-ratio using any $\F$-valued norm on an affine chart.
\begin{proposition}
Let $a,x, y,b \in L\subset \PP V$ be aligned and such that no three of the four points agree, i.e.\ their cross ratio is defined.
Take any affine chart $f=f_\eb \from \F^{d-1} \to \PP V$ such that $a=f(a')$,  $x=f(x')$, $y=f(y')$,  $b=f(b')$ for $a',x',b',y' \in \F^{d-1}$,  and let $N$ be any $\F$-valued norm on $\F^{d-1}$.
If $a,x,y,b$ are a positively oriented on $L$,  then 
\[\CR{a,x,y,b} = \frac{N(y'-a')}{N(x'-a')} \cdot \frac{N(x'-b')}{N(y'-b')}.\]
\hfill $\square$
\end{proposition}
\compl{
\begin{proof}
Since $a,x,y,b$ lie on $L$ there exist $v \in \F^{d-1}$ and $t_a, t_x, t_y, t_b \in \F$ such that $a'=t_av, x'=t_x v,  y'=t_y v,  b'=t_b v$.
One obtains
\begin{align*}
\frac{N(y'-a')}{N(x'-a')} \cdot \frac{N(x'-b')}{N(y'-b')}&= \frac{N((t_y-t_a)v)}{N((t_x-t_a)v)} \cdot \frac{N((t_x-t_b)v)}{N((t_y-t_b)v)}\\
&= \frac{|t_y-t_a|_\F}{|t_x-t_a|_\F} \cdot \frac{|t_x-t_b|_\F}{|t_y-t_b|_\F}\\
&=\CR{a,x,y,b},
\end{align*}
where the last equality holds,  since $a,x,y,b$ are positively ordered.
\end{proof}}

\subsection{Convex sets}
\label{subs:ConvexSets}
Let $\F$ be an ordered field and  $V$ a finite-dimensional $\F$-vector space.

\begin{definition}[Convex sets in vector spaces]
\label{dfn:ConvexSetsVectorSpaces}
A subset $C$ of $V$ is \emph{convex} if for all $v, w \in C$ the segment $[v,w] \coloneqq \{ tv + (1-t)w \mid 0\leq t\leq1, \, t \in \F\}$ is included in $C$.
\end{definition}

Note that, unlike in the real case, convex subsets of $\FF$ are not
necessarily segments:  
for example the balls as defined in \Cref{dfn_Balls} are convex.
%
It is easy to verify that half-spaces are convex, as well as linear
subspaces and intersections and products of convex sets.
For any $\F$-valued norm $N$ on $V$,  we have that $B_N(x,R)$ is convex for all $x \in V$ and $R \in \F_{\geq 0}$.
Recall that $V$ is endowed with the product topology.
The interior,  the relative interior and the closure of a convex set are convex.
\compl{
\begin{proof}
Let $v,w \in \Int{C}$.
Then by convexity of $C$ the segment $[v,w]$ is included in $C$.
Fix $0<t<1$. 
We have $tC+(1−t)C\subseteq C$ by convexity, so $t\Int{C}+(1−t)\Int{C} \subseteq C$. 
But $t\Int{C}$ is open, so $t\Int{C}+(1−t)\Int{C}= \bigcup_{x \in (1-t)\Int{C}} (t\Int{C}+x)$ is open,  and thus $t\Int{C} +(1-t)\Int{C} \subseteq \Int{C}$,  hence $\Int{C}$ is convex.

We turn to the convexity of $\overline{C}$.
Let $v,w \in \overline{C}$.
By definition of $\overline{C}$ we have that for all $s \in \F_{>0}$ there exist $v_0,w_0 \in C$ such that $\|v-v_0\|,  \|w-w_0\| < s$.
Since $C$ is convex $[v_0,w_0] \subset C$.
Take $u = tv+(1-t)w$ with $0\leq t \leq 1$.
We claim that $u \in \overline{C}$.
Indeed for all $s \in \F_{>0}$ we have that for $u_0 = tv_0+(1-t)w_0 \in C$
\[\|u-u_0\| \leq t\|v-v_0\| + (1-t) \|w-w_0\| <s,\]
which shows that  $u \in \overline{C}$.
ALTERNATIVE: Without using $\norm{\cdot}$.
Let $C$ be a nonempty convex subset of $X$.
For $x,y \in C$ and $\lambda \in [0,1]$,  we prove that any neighborhood of $z=\lambda x + (1-\lambda) y$ intersects $C$. 
So let $W$ be a neighborhood of $0$. 
Because $(u,v) \to  \lambda u + (1-\lambda) v$ is continuous, there exist open subsets $U$ and $V$ such that $\lambda U +(1-\lambda) V \subset W$; $x + U$ (resp. $y+V$) is an open neighborhood of $x$ (resp. of $y$) so there exists $x_1 \in (x+U)\cap C$ (resp. $y_1 \in (y+V)\cap C$).
Therefore,  $z_1 = \lambda x_1 + (1-\lambda y_1) \in C$ because $C$ is convex and $z_1 \in \lambda(x+U)+(1-\lambda)(y+v) \subset z+W$, i.e. $(z+W) \cap C \neq \emptyset$.
\end{proof}
}

As any intersection of convex subsets is convex, any subset $A\subset
V$ has a well-defined \emph{convex hull} $\conv(A)$,  namely the
intersection of all convex sets in $V$ containing $A$.
This is the same as the set of all convex combinations of points in $A$.
\compl{Every convex set containing $A$ contains all convex combinations of points in $A$, thus the set of all such convex combinations is contained in $\conv(A)$.
Conversely,  the set of all convex combinations is convex and contains $A$,  thus contains the intersection of all such sets.}
By Caratheodory's theorem (which holds for arbitrary ordered fields),  it suffices to look at convex combinations of at most $d+1$ points,  where $d$ is the dimension of $V$.
If we consider ordered field extensions $\F\subseteq \L$ we write $\conv_\L(A)$ to specify the $\L$-convex hull, i.e.\ the set of all $\L$-convex combinations of points in $A$,  where we view $A$ as a subset of the $\L$-vector space $V \otimes \L$.
In the case that $A$ is a finite set of vectors in $V$,  its convex hull is called an \emph{affine polytope}.

We now turn to convex sets in projective spaces.

\begin{definition}[Convex sets in projective lines]
If $\PP V$ is a projective line,  a subset $\Omega \subseteq \PP V$ is \emph{convex} if either $\Omega = \PP V$ or $\Omega$ is convex in one (or any) affine chart.
\end{definition}

Note that this is well-defined because the change of affine chart is a monotonic map.

\begin{definition}[Convex sets in projective spaces]
Let $\Omega \subseteq \PP V $ be a subset.
We say that $\Omega$ is \emph{convex} if the intersection of $\Omega$ with every projective line in $\PP V$ is convex.
\end{definition}

Note that if $\Omega$ is contained in an affine chart,  then $\Omega$ is convex in $\PP V$ if and only if $\Omega$ is convex in the affine chart (in the sense of \Cref{dfn:ConvexSetsVectorSpaces}).
Observe that, like in the real case,  if $\Omega$ is convex then $\Omega^*$ is a convex subset of $\PP V^*$.

\begin{definition}
Let $\Omega \subseteq \PP V$ be a subset.
We say that $\Omega$ is \emph{properly convex} if 
there exists an affine chart in which $\Omega$ is convex and bounded.
Equivalently,  $\Omega$ is the image of a properly convex cone.
\end{definition}

\Cref{lem:ProperAffineChart}
together with the observation that the closure of a convex set is convex,  implies in particular 
that if $\Omega$ is properly convex,  then so is $\overline{\Omega}$.

\section{Hilbert geometry over ordered valued fields and first examples}
\label{section:HilbertGeometry}

In this section we define the basics of Hilbert geometry over general ordered valued fields and present some examples.
From now on let $V$ be a finite-dimensional vector space over an ordered valued field $\F$.

\subsection{Hilbert pseudo-distance and associated Hilbert metric space}
With the above definitions we can now introduce the Hilbert pseudo-distance.

\begin{definition}[Hilbert pseudo-distance]
Let $\Omega \subset \PP V$ be a convex set.
We define the function $d_\Omega \from \Omega \times \Omega \to \R \cup \{\infty\}$ by
\[d_\Omega (x,y) \coloneqq \sup_{\varphi, \psi \in \Omega^*} \log |\CR{\varphi,x,y,\psi}|.\]
\end{definition}

The supremum of the empty set is defined to be $\infty$.
Note that this definition makes sense more generally for all valued fields (not necessarily ordered),  
when $\Omega$ is taken to be any subset; 
compare also to \cite{Guilloux_PAdicHypDisc} in the case of the $p$-adic numbers and to \cite{FalbelGuillouxWill_HilbertGeometryWithoutConvexity,FalbelGuillouxWill_HilbertMetricBoundedSymmetricDomains} in the case of complex numbers.

By \Cref{lem:PropertiesCR}~(\ref{lem:PropertiesCR:Dim1}), we have
the following alternative expression for $d_\Omega$, namely 
\begin{equation}
d_\Omega (x,y) = \sup_{a, b \in L \setminus \Omega} \log \abs{\CR{a,x,y,b}}, 
\end{equation}
where $L$ is any line passing through $x$ and $y$.

We recover the following properties of $d_\Omega$ which are classical in the case $\F=\R$.

\begin{proposition}[Pseudo-distance]
\label{lem:PseudoDistdOmega}
Let $\Omega \subset \PP V$ be a convex set with $\Omega^* \neq \emptyset$.
Then $d_\Omega$ is a pseudo-distance, i.e.\ it has the following properties: For all $x,y,z \in \Omega$ we have
\begin{enumerate}
	\item 
	\label{lem:PropertiesdOmega:Zero}
	$d_\Omega(x,x)=0$,
	\item  \label{lem:PropertiesdOmega:Pos}
	$d_\Omega(x,y)\geq 0$,
	\item  \label{lem:PropertiesdOmega:Sym}
	$d_\Omega(x,y)=d_\Omega(y,x)$,  and
	\item  \label{lem:PropertiesdOmega:TriangleInequality}
	$d_\Omega(x,y) \leq d_\Omega(x,z)+d_\Omega(z,y)$.
	\qed
\end{enumerate}
\end{proposition}

The above proposition follows directly from the definition and neither uses that $\F$ is an ordered field nor that $\Omega$ is convex.
In fact,  they hold more generally for valued fields (not necessarily ordered) and arbitrary subsets $\Omega$ (see e.g.\ \cite[Proposition 2.6]{FalbelGuillouxWill_HilbertGeometryWithoutConvexity} in the complex case for an analogue of \Cref{lem:PseudoDistdOmega}~(\ref{lem:PropertiesdOmega:Zero})-(\ref{lem:PropertiesdOmega:TriangleInequality})).

It cannot be excluded that $d_\Omega$ takes the value $\infty$.
The following gives a sufficient condition to ensure that $d_\Omega$ is finite.

\begin{proposition}[Finiteness]
\label{lem:PropertiesdOmega:Finite}
Let $\Omega \subset \PP V$ be a convex set with $\Omega^* \neq \emptyset$.
If $\Omega$ is open,  then $d_\Omega(x,y)< \infty$ for all $x,y \in \Omega$.
\end{proposition}

\compl{
Alternatively :
\begin{proposition}[Finiteness of $d_\Omega$]
Let $\Omega \subset \PP V$ be a convex set with $\Omega^* \neq \emptyset$.
For any $x \neq y \in \Omega$,  we have $d_\Omega(x,y) = \infty$ if and only if $x \in \partial\Omega$ or $y\in \partial\Omega$.
In particular,  if $\Omega$ is open,  then $d_\Omega(x,y)<\infty$ for all $x,y \in \Omega$.
\end{proposition}
\begin{proof}
Assume first that $d_\Omega(x,y)=\infty$. 
Then we have $x\neq y$ and, denoting $L$ the line passing through $x$ and $y$,
there exist two sequences $a_n, b_n \in L \setminus \Omega$ 
 with $\log|\CR{a_n,x,y,b_n}| \to \infty$,  or equivalently $\CR{a_n,x,y,b_n} \to \infty$ (in $\F$).
Choose some $z \in L \setminus \Omega$. 
By multiplicativity of the cross-ratio,  see \Cref{lem:PropertiesCR} (\ref{lem:PropertiesCR:Multiplicativity}),  we have that 
\[\CR{a_n,x,y,z} \CR{z,x,y,b_n}=\CR{a_n,x,y,b_n}  \to \infty,\]
which implies that either $\CR{a_n,x,y,z} \to \infty$ or $\CR{z,x,y,b_n} \to \infty$.
Let us assume we are in the first case.
By definition of the cross-ratio this implies that $a_n\to x$. 
Since $a_n \in \Omega^c$,  we have $x \in \partial\Omega$.
In the second case,  we argue similarly and obtain that $y \in \partial\Omega$.

Conversely,  assume,  up to exchanging $x$ and $y$,  that $x \in \partial\Omega$.
Let $L$ be the line through $x$ and $y$.
Take a sequence $a_n \in \Omega^c \cap L$ with $a_n \to x$.
Then $\CR{a_n,x,y,b} \to \infty$ for any $b_0 \in \Omega^c$ different from any $a_n$.
In particular,  $d_\Omega(x,y) = \sup_{a,b \in L \cap \Omega^c} \log |\CR{a,x,y,b}| \geq \log |\CR{a_n,x,y,b_0}| \to \infty$.

If $\Omega$ is open,  then $\Omega$ is disjoint from $\partial\Omega$,  which concludes.
\end{proof}
}

\begin{proof}
Assume there exists points $x,y \in \Omega$ with $d_\Omega(x,y)=\infty$. 
Then we have $x\neq y$ and, denoting $L$ the line passing through $x$ and $y$,
there exist two sequences $a_n, b_n \in L \setminus \Omega$ 
 with $\log|\CR{a_n,x,y,b_n}| \to \infty$,  or equivalently $\CR{a_n,x,y,b_n} \to \infty$ (in $\F$).
Choose some $z \in L \setminus \Omega$. 
By multiplicativity of the cross-ratio,  see \Cref{lem:PropertiesCR}~(\ref{lem:PropertiesCR:Multiplicativity}),  we have that 
\[\CR{a_n,x,y,z} \CR{z,x,y,b_n}=\CR{a_n,x,y,b_n}  \to \infty,\]
which implies that either $\CR{a_n,x,y,z} \to \infty$ or $\CR{z,x,y,b_n} \to \infty$ as the quadruples are positively oriented.
Let us assume we are in the first case.
By definition of the cross-ratio this implies that $a_n\to x$. 
\compl{$\CR{a_n,x,y,z}=\CR{x,a_n,z,y}=\CR{y,z,a_n,x}=\CR{0,1,t_n,\infty}=t_n$}
Since $a_n \in \Omega^c$,  we have $x \in \partial\Omega$,  which contradicts the fact that $x \in \Omega$ and $\Omega$ is open.
In the second case,  we argue similarly and obtain that $y \in \partial\Omega$.
\end{proof}

Thus if $\Omega$ is open, convex and contained in an affine chart (or equivalently $\Omega^* \neq \emptyset$) the function $d_\Omega$ is a pseudo-distance,  called the \emph{Hilbert pseudo-distance} on $\Omega$.
If $\F$ is non-Archimedean,  $d_\Omega$ does not distinguish points, i.e.\ there are distinct points at distance zero.
Indeed,  for $\Omega = ]-1,1[ \subset \F$ we have $d_\Omega(0,x) = 0$ for all $x \in ]-1,1[\cap\Q \subset \F$.
Hence $d_\Omega$ does not define a distance on $\Omega$ (in contrast to the real case),  and we define the quotient metric space.

\begin{definition}[Associated Hilbert metric space]
\label{dfn_AssociatedMetricSpace}
Let $\Omega \subseteq \PP V$ be open, convex and contained in an affine chart with Hilbert pseudo-distance $d_\Omega$.
The \emph{associated Hilbert metric space} is 
\[ X_\Omega \coloneqq \Omega/\{d_\Omega = 0\},\]
endowed with the distance induced by $d_\Omega$.
\end{definition}

We denote by $\pi \from \Omega \to X_\Omega$ the projection from $\Omega$ to the associated Hilbert metric space.
For $x \in \Omega$ we write $\overline{x}\coloneqq \pi(x)$ for its image in $X_\Omega$.

We describe two Hilbert metric spaces in the one-dimensional case.
Remark that in the case that $\F$ is non-Archimedean,  $\F$ contains more open convex sets than just segments (see e.g.\ \Cref{example:ball} or \Cref{examp:ConvexSetNEQSegment}).

\begin{example}[One-dimensional segment]
\label{lem:X(OneDimensionalSegment)}
Let $\F$ be any ordered valued field with $\Lambda \coloneqq \log |\F^\times| \subseteq \R$, and let $\Omega = ]a,b[ \subset \F$ with $a<b \in \F$.
Then the map $\Phi \from ]a,b[ \to \R$,  $x \mapsto \log \big|\tfrac{x-a}{x-b}\big|$ descends to an isometry from $(X_\Omega, d_\Omega)$ to $(\Lambda, |\cdot|_\R)$, where $|\cdot|_\R$ denotes the standard absolute value on $\R$ and its induced distance.
Clearly $\Phi$ is surjective onto $\Lambda$.
Furthermore,  for all $x\leq y \in \Omega$ we have
\[|\Phi(y)-\Phi(x)|_\R = \left| \log \tfrac{|y-a|}{|y-b|} - \log \tfrac{|x-a|}{|x-b|} \right|_\R = \log \abs{\CR{a,x,y,b}} = d_\Omega(x,y),\]
hence $\Phi$ preserves pseudo-distances.
Thus $\Phi$ descends to an isometry $X_\Omega \to \Lambda$.
\end{example}

This generalizes the case $\F=\R$ with the standard absolute value to the non-Archimedean setting.
The following however is in strong contrast with the real case.

\begin{example}[Balls]
\label{example:ball}
Assume now that $\F$ is non-Archimedean.
For all $r \in \R_{>0}$ and $x \in \F$,  we have $X_{B(x,r)} =\{\overline{x}\}$.
Indeed,  let $y \in B(x,r)$ with $x \neq y$.
We claim that $d_{B(x,r)}(x,y) = 0$.
By definition 
\[d_{B(x,r)}(x,y) = \sup_{a,b \in \F:\, |ax|,|bx| \geq r} \log \abs{\CR{a,x,y,b}}.\]
Since $|xy| < r$ it follows by \Cref{lem:NonArchimedeanPropertiesCR}~(\ref{lem:NonArchimedeanPropertiesCR:CR=1}) that $\abs{\CR{a,x,y,b}}=1$ for all $a,b \in \F$  with $|ax|,|bx| \geq r$.
Thus $d_{B(x,r)}(x,y) =0$.
\end{example}

\begin{remarks}
\label{rem:Segment}
\noindent\begin{enumerate}
\item
In the setting of \Cref{lem:X(OneDimensionalSegment)},  the Hilbert metric space $X_\Omega$ is complete if and only if $\Lambda$ is either discrete in $\R$ (e.g.\ for $\R(X)$,  see \Cref{exa:ValuedFields}) or equal to $\R$ (e.g.\ for Robinson fields,  see \Cref{subsection:RobinsonField}).
If $\Lambda$ is dense in $\R$,  e.g.\ if $\F$ is real closed, then the metric completion of $X_\Omega$ is isometric to $\R$.

\item \label{rem:Segment:CompleteGeod}
In the real case,  through any two points in an open properly convex set $\Omega$ there passes a complete geodesic.
For general ordered valued fields,  this might no longer be true,  see e.g.\ \Cref{example:ball}. 
However,  if $x, y$ are on a projective line $L$ and 
$\Omega \cap L = ]a,b[$ for some $a<b \in \F$, then the image of $]a,b[$ in $X_\Omega$ is a complete \emph{$\Lambda$-geodesic},  i.e.\ an isometric embedding of $(\Lambda,|\cdot|_\R)$ in $X_\Omega$.
\end{enumerate}
\end{remarks}

In turns out that convex sets as in Remark \ref{rem:Segment}~(\ref{rem:Segment:CompleteGeod}) are ``well-behaved'',  which motivates the definition of well-bordered convex sets (\Cref{dfn:WellBordered}) in the following section.
However it is good to keep examples with pathological behavior in mind (see \Cref{example:ball}) 
to not be misguided by our intuition in the Archimedean case.

\subsection{Basic properties}
\label{subsection:BasicProperties}
In this subsection we collect basic properties of $d_\Omega$ and $X_\Omega$.
We also define well-bordered convex sets on which one can define an $\F$-valued multiplicative distance on $\Omega$,  and we relate them to $\Lambda$-metric spaces.

Again let $\F$ be an ordered valued field,  and $V$ a finite-dimensional $\F$-vector space.

\begin{proposition}[Basic properties of $d_\Omega$]
\label{lem:PropertiesdOmega}
Let $\Omega \subset \PP V$ be a convex set with $\Omega^* \neq \emptyset$.
Then $d_\Omega$ has the following properties: 
\begin{enumerate}
	\item \label{lem:PropertiesdOmega2:Sandwich}
  	Let $\Omega'$ be a convex set with $\Omega' \subseteq \Omega $.
  	Then for all $x,y \in \Omega'$ we have $d_\Omega(x,y) \leq d_{\Omega'}(x,y)$.
  	\item \label{lem: intersectingWithSubspace}
  	Let $\PP W$ be a projective subspace of $\PP V$,  then $d_\Omega = d_{\Omega \cap \PP W}$ on $\Omega \cap \PP 		W$.  
	\item \label{lem:InvProjTrans}
	If $g \in \PGL(V)$ preserves $\Omega$,  then $d_\Omega(x,y)=d_\Omega(gx,gy)$ for all $x, y \in \Omega$.
  	\qed
\end{enumerate}
\end{proposition}

This implies that given $g \in \PGL(V)$ that preserves $\Omega$,  then $g$ descends to an
isometry of $X_\Omega$.

\begin{proposition}[Additivity on segments]
\label{lem:PropertiesdOmega:AdditivityOnSegments}
Let $\Omega \subset \PP V$ be a convex set contained in an affine chart.
For all $x,y,z \in \Omega$, if $z \in [x,y]$ (in the affine chart) then $d_\Omega(x,y) = d_\Omega(x,z)+d_\Omega(z,y)$.
\end{proposition}
\begin{proof}
It suffices to check the one-dimensional case $\Omega \subset \F$ (see \Cref{lem:PropertiesdOmega}~(\ref{lem: intersectingWithSubspace})).
Let $a,a',b,b' \in \Omega^{*}\simeq\F\cup\{\infty\} \setminus \Omega$  with $a<x<y<b$ and $a'<x<y<b'$.
Then for $a''\coloneqq \max\{a,a'\}$ and $b'' \coloneqq \min\{b,b'\}$,  by monotonicity and multiplicativity of the cross-ratio,  see \Cref{lem:PropertiesCR:Monotonicity} and \Cref{lem:PropertiesCR}~(\ref{lem:PropertiesCR:Multiplicativity}), we have
\begin{align*}
\CR{a,x,z,b}\CR{a',z,y,b'} &\leq \CR{a'',x,z,b''}\CR{a'',z,y,b''}\\
&=\CR{a'',x,y,b''} \in \F_{\geq 1}.
\end{align*}
Hence taking the logarithm of the absolute values,  we obtain
\begin{align*}
\log \abs{\CR{a,x,z,b}} + \log \abs{\CR{a',z,y,b'}}\leq \log \abs{\CR{a'',x,y,b''}} \leq d_\Omega(x,y)
\end{align*}
Now taking suprema on $a,b$ and on $a',b'$ we get
\[d_\Omega(x,z) + d_\Omega(z,y)\leq d_\Omega(x,y),\]
which concludes using the triangle inequality.
\end{proof}

Recall that in \Cref{example:ball} we saw that the intersection of $\Omega$
with a projective line is not necessarily a segment --- an observation which lead to some degenerate behavior.
This motivates the following.

\begin{definition}
\label{dfn:WellBordered}
We say that an open and convex set $\Omega \subseteq V$ is \emph{well-bordered} if the intersection of $\Omega$
with every affine line in $V$ is an open segment $]a,b[$ with $a,b \in V$.

An open and convex subset of the projective space $\PP V$ is \emph{well-bordered} if it is contained in some affine chart and well-bordered in it (note that this do not depend on the affine chart considered).
\end{definition}

\begin{example}
\label{examp:WellBordered}
Semi-algebraic, i.e.\ solutions to a finite set of polynomial equalities and inequalities,  open convex sets over real closed fields $\F$ are well-bordered. 
Indeed intersections with lines are semi-algebraic convex subsets of $\F$,  hence segments \cite[Proposition 2.1.7]{BochnakCosteRoy_RealAlgebraicGeometry}.
We will see later that ultralimits of convex subsets are always well-bordered (\Cref{prop-UltralimitConvexAndLines}).
\end{example}

Note that $\Omega \cap L = ]a,b[$ implies $\overline{\Omega} \cap L =[a,b]$.
\compl{This is slightly weaker than ``the intersection of the closure with every line is a closed segment'', which may not be true for ultralimits,  see Example  \ref{exa:UltralimitConvexAndLines:NotSegment}.}
Thus if $\Omega$ is well-bordered,  we can get rid of the supremum in the definition of the Hilbert pseudo-metric and retrieve the usual real case definition:
for $x,y \in \Omega$,  for any line $L$ passing through $x$ and $y$, we have by monotonicity of the cross-ratio that 
\[d_\Omega (x,y) = \log \abs{\CR{a,x,y,b}},\]
where $a,b$ are the endpoints of the segment $\Omega \cap L$, and $a,x,y,b$ are positively oriented.

We finish this section with two remarks and a few examples.
\begin{remark}
\label{subs:FvaluedMultHilbertDistance}
If $\Omega$ is well-bordered,  we may also define the \emph{$\F$-valued (multiplicative) Hilbert distance} on $\Omega$ by
\[D_\Omega (x,y) \coloneqq \CR{a,x,y,b} \in \F_{\geq 1},\]
where $a,b$ are the endpoints of the segment $\Omega \cap L$, and $a,x,y,b$ are positively ordered.
Then $D_\Omega$ has the following properties for all $x,y,z \in \Omega$:
\begin{enumerate}
\item $D_\Omega (x,y) \leq D_\Omega (x,z)D_\Omega (z,y) $;
\item $D_\Omega (x,y)=1 \Leftrightarrow x=y $;
\item $d_\Omega (x,y)=\log |D_\Omega (x,y)|$. 
\end{enumerate}
Note that contrary to the $\R$-valued pseudo-distance $d_\Omega$, the $\F$-valued multiplicative distance separates points.
Since a supremum is not available in ordered fields different from $\R$,  compare the end of \Cref{subsection:NAOrderedFields},  the description of the Hilbert pseudo-distance in terms of the $\F$-valued multiplicative Hilbert distance does not hold for convex sets which are not well-bordered.
\end{remark}

Well-bordered convex sets give rise to $\Lambda$-metric spaces;  refer to  \cite[Chapter 1.2]{Chiswell_IntroductionLambdaTrees} for an introduction.
A \emph{$\Lambda$-metric space} is a metric space $(X,d)$ for which $d$ takes values in $\Lambda \subset \R$.
It is \emph{geodesic} if between any two points $x,y \in X$ there exists a \emph{$\Lambda$-segment} with endpoints $x,y$,  i.e.\ the image of an isometric map $f \from [p,q]_\Lambda \to X$ with $f(p)=x$ and $f(q) = y$,  where $[p,q]_\Lambda \coloneqq [p,q] \cap \Lambda$ for $p\leq q \in \Lambda$.

\begin{remark}
If $\Omega$ is well-bordered,  then $X_\Omega$ is a geodesic $\Lambda$-metric space.
\compl{The expression of the Hilbert pseudo-distance for well-bordered sets (\Cref{subsection:WellBordered}) implies that $X_\Omega$ is a $\Lambda$-metric space.
That $X_\Omega$ is geodesic follows since we can restrict to the one-dimensional case (\Cref{lem:PropertiesCR} (\ref{lem:PropertiesCR:Dim1})) and in this case we have a complete description of the Hilbert metric space (\Cref{lem:X(OneDimensionalSegment)}).}
However,  if $\Lambda \neq \R$ and $\Omega$ is not well-bordered, then $d_\Omega$ can
take values in $\R \setminus \Lambda$ as the following example shows.
\end{remark}

\begin{example}
\label{examp:ConvexSetNEQSegment}
Let $\F$ be a non-Archimedean ordered valued field with $\Lambda=\Q$.
Let $\epsilon=e^{-\sqrt{2}}\in \R$.
Consider $\Omega = ]-1,1[\, \cup B(1,\epsilon) \subset \F$.
Then $d_\Omega(0,1) =\sqrt{2}\notin \Q$.

More generally for $\epsilon \in \R_{>0}$, $z \in \F_{>0}$ consider the following convex subsets of $\F$:
\begin{align*}
N_\epsilon(]-z,z[) &\coloneqq B(-z,\epsilon) \cup\, ]-z,z[ \, \cup  B(z,\epsilon), \\
N^+_\epsilon(]-z,z[) &\coloneqq ]-z,z[ \, \cup B(z,\epsilon).
\end{align*}
If $|z| < \epsilon$,  then $X_{N_\epsilon(]-z,z[)}$ and $X_{N^+_\epsilon(]-z,z[)}$ are isometric to a point.
If $|z| \geq \epsilon$ and $\Lambda$ is dense in $\R$,  then $X_{N_\epsilon(]-z,z[)}$ is isometric to $]-\log(|z|/\epsilon),\log(|z|/\epsilon)[_\Lambda \cup \{ \pm \log(|z|/\epsilon)\}$,
and $X_{N^+_\epsilon(]-z,z[)}$ is isometric to $]-\infty, \log (|z|/\epsilon)[_\Lambda \cup \{\log (|z|/\epsilon)\}$.
This follows by carefully combining the arguments in \Cref{lem:X(OneDimensionalSegment)} and \Cref{example:ball}.
In particular,  $X_{N_\epsilon(]-z,z[)}$ and $X_{N^+_\epsilon(]-z,z[)}$ are $\Lambda$-metric spaces if and only if $\epsilon \in |\F|$.
Furthermore,  if $\abs{\cdot}$ is surjective, we obtain using the above constructions $]-\infty,a]$, $[a,\infty[$, $[-a,a]$ for all $a \in \R_{\geq 0}$ and $]-\infty,\infty[$ (\Cref{lem:X(OneDimensionalSegment)}) as associated Hilbert metric spaces of non-Archimedean one-dimensional convex sets.

\compl{If $|z| < \epsilon$,  then $N_\epsilon(]-z,z[)=B(0,\epsilon)$ and the result follows from \Cref{propo:X(OneDimensionalBalls)}.
Assume now that $\delta \coloneqq |z| \geq \epsilon$.
We start by computing $d_{N_\epsilon(]-z,z[)}(0,x)$ for $x \in N_\epsilon(]-z,z[)$,  $y>0$.
We have 
\[d_{N_\epsilon(]-z,z[)}(0,x) = \sup_{a,b \in \F_{>0}: \,  a,b > z, \,\delta \geq |az|,|bz| \geq \epsilon} \log |\CR{-a,0,x,b}|,\]
where we can assume $\delta \geq |az|,|bz|$ by monotonicity of the cross ratio \Cref{lem:PropertiesCR:Monotonicity}.
In this case, by ultrametricity and alignment of the points (see \Cref{lem:AbsoluteValueNA:EqualityPos}) we obtain $|a| = \max\{|az|,|z|\} = \delta$ and analogously $|b|=\delta$.
Thus $|\CR{-a,0,x,b}|=\tfrac{|x(-a)|}{|xb|}$.
Furthermore,  $|x(-a)|=\max\{|x|,|z|,|az|\}=\delta$.
Putting everything together we can write
\begin{align*}
d_{N_\epsilon(]-z,z[)}(0,x) &= \sup_{b>z,\,\delta \geq |bz| \geq \epsilon} \log \tfrac{\delta}{|xb|} \\
&= \log \delta - \inf_{b>z,  \, \delta \geq |bz| \geq \epsilon} \log |xb|.
\end{align*}
We now consider three cases.

\textbf{Case 1:}
Let us first assume that $|xz|<\epsilon$, i.e.  $x \in B_z(\epsilon)$.
If $x>z$ then $\epsilon \leq |bz|=\max\{|bx|,|xz|\}=|bx|$,  since $|xz|<\epsilon$.
If $x<z$ then $|bx|=\max\{|bz|,|zx|\}=|bz|$,  since $|bz| \geq \epsilon$.
In both cases $\inf_{b>z,  \, \delta \geq |bz| \geq \epsilon} \log |xb| = \inf_{b>z,  \, \delta \geq |bz| \geq \epsilon} \log |bz|$.
Since $\Lambda$ is dense in $\R$,   $\inf_{b>z,  \, \delta \geq |bz| \geq \epsilon} \log |bz|=\epsilon$.
Thus $d_{N_\epsilon(]-z,z[)}(0,x) = \log(\delta/\epsilon)$ for all $x \in B_z(\epsilon)$.
In particular,  by additivity on segments, see \Cref{lem:PropertiesdOmega} (\ref{lem:PropertiesdOmega:AdditivityOnSegments}), we obtain $d_{N_\epsilon(]-z,z[)}(x,y) = 0$ for all $x,y \in B_z(\epsilon)$.

\textbf{Case 2:}
Assume now that $|xz|=\epsilon$.
By the same argument as above we have $|bx|=\max\{|bz|,|zx|\}=|bz|$,  and hence $d_{N_\epsilon(]-z,z[)}(0,x) = \log(\delta/\epsilon)$.

\textbf{Case 3:}
Assume $|xz|>\epsilon$.
By density of $\Lambda$ in $\R$, there exists $b \in \F_{>0}$ with $|xz|>|bz|\geq \epsilon$.
Thus $\inf_{b>z,  \, \delta \geq |bz| \geq \epsilon} \log |xb|=\log|xz|$.
We obtain $d_{N_\epsilon(]-z,z[)}(0,x) = \log(\delta/|xz|)$.

By symmetry,  we obtain the same results for $x<0$.
We define
\begin{align*}
\Phi \from N_\epsilon(]-z,z[) \to ]-\log(|z|/\epsilon),\log(|z|/\epsilon)[_\Lambda \cup \{ \pm \log(|z|/\epsilon)\\
x \mapsto \begin{cases}
\log \tfrac{|x(-z)|}{|xz|},  &\textnormal{ if } x \in ]-z,z[\setminus B(z,\epsilon)\\
\log \tfrac{\delta}{\epsilon},  &\textnormal{ if } x \in B(z,\epsilon)\setminus]-z,z[\\
-\log \tfrac{\delta}{\epsilon},  &\textnormal{ if } x \in B_\epsilon(-z)\setminus]-z,z[.
\end{cases} 
\end{align*}
Then $\Phi$ preserves distances,  is surjective and induces hence an isometry from $X_{N_\epsilon(]-z,z[)}$ to $]-\log(|z|/\epsilon),\log(|z|/\epsilon)[_\Lambda \cup \{ \pm \log(|z|/\epsilon)\}$ (compare the proof of \Cref{lem:X(OneDimensionalSegment)} for more details).}
\end{example}

Let us now move to some higher dimensional examples.

\subsection{Simplices and model flats}
\label{subsection:ModelFlats}
In classical Hilbert geometry,  simplices play the role of flat subspaces,  since they are the only domains whose Hilbert geometry is isometric to a normed vector space \cite{Vernicos_HilbertGeometryConvexPolytopes}.
They also generalize the example of the one-dimensional segment (\Cref{lem:X(OneDimensionalSegment)}) to higher dimensions.
Before we can study the non-Archimedean counterpart of simplices,  we need to introduce the model flat,  given by a Cartan subalgebra of the Lie algebra of $\SL(d,\R)$ endowed with a norm invariant by the Weyl group of the root system $A_{d-1}$.

Let $\Aa \coloneqq \R^d / \R(1,\ldots,1)$.
For $\alpha=(\alpha_1,\ldots,\alpha_d) \in \R^d$,  we write $[\alpha]$ for its class in $\Aa$.
We endow $\Aa$ with the metric $\dhex$ induced by the norm on $\Aa$ defined as
\[\nhex{[\alpha]} \coloneqq \max_{1 \leq i,j \leq d} |\alpha_i - \alpha_j|.\]
On the closed model Weyl chamber 
$\Aap \coloneqq \{ [\alpha] \in \Aa \mid \alpha_1 \geq \ldots \geq \alpha_d \}$ we have $\nhex{[\alpha]} = \alpha_1 - \alpha_d$.
We call $\Aa$ together with the metric $\dhex$ the $(d-1)$-dimensional \emph{model flat}.

For $\Lambda < \R$ an ordered abelian subgroup of $(\R,+)$ we define
analogously $\Aa_\Lambda \coloneqq \Lambda^d/\Lambda(1,\ldots,1)$ and
 $\Aap_\Lambda\coloneqq\Aa_\Lambda \cap \Aap$.
Then
$\Aa_\Lambda  \subset \Aa$
\compl{Assume $(y_1,\ldots,y_d)+\R(1,\ldots,1)=(x_1,\ldots,x_d)+\R(1,\ldots,\R)$ with $(x_i),(y_i) \in \Lambda^d$.
Then there exists $\alpha \in \R$ with $(y_1,\ldots,y_d)+\alpha=(x_1,\ldots,x_d)$.
Since $\Lambda$ is a group,  we get that $\alpha \in \Lambda$.
}
and the restriction of $\dhex$ to $\Aa_\Lambda$ 
induces a metric (with values in $\Lambda$) on $\Aa_\Lambda$,  also denoted by $\dhex$.
Then $\Aa_\Lambda$ together with the metric $\dhex$ is the $(d-1)$-dimensional \emph{$\Lambda$-model flat}.

\begin{remark}
Let $\Lambda^d_0 \coloneqq \{\alpha \in \Lambda^d \mid \sum_{i=1}^d \alpha_i =0 \}\subseteq \Lambda^d$.
Then $\Aa_\Lambda \cong \Lambda^d_0$ if and only if $\Lambda$ is $d$-divisible, i.e.\ for all $\lambda \in \Lambda$ there exists $\eta \in \Lambda$ with $d\eta = \lambda$.
In fact we always have an injective map $\Lambda^d_0 \to \Aa_\Lambda$.
However for $\Lambda=\ZZ$,  the element $[(1,0,\ldots,0)] \in \Aa_\ZZ$ does not lie in the image of this map.
For real closed valued fields with value group $\Lambda$ we always have that $\Q \subseteq \Lambda$, and hence $\Lambda$ is $d$-divisible for every $d$.
\end{remark}

Let us now turn to the study of simplices.
Let $\F$ be an ordered valued field with value group $\Lambda =\log |\F^\times|$,  and $V$ a $d$-dimensional $\F$-vector space. 
To a basis  $\eb = (e_i)$ of $V$ we associate the following closed
and open  $(d-1)$-dimensional \emph{simplices}
\begin{align*}
\simplex_{\eb} &\coloneqq \PP\big(\big\{ x=\sum_{i=1}^d x_i e_i \bigm| x_i \in \F_{\geq 0} 
\textrm{ for all } i=1,\ldots,d\big\}\big),  \textnormal{ and }\\
\simplex_{\eb}^{o} &\coloneqq \PP\big(\big\{ x=\sum_{i=1}^d x_i e_i \bigm|x_i \in \F_{> 0} 
\textrm{ for all } i=1,\ldots,d\big\}\big) \subset \PP V.
\end{align*}
Then there is a well-defined map from the open simplex to the $\Lambda$-model flat
\[
      \log_{\eb} \from \simplex_{\eb}^{o} \to \Aa_\Lambda, \;
\PP(x) \mapsto   [\log|x_1|, \ldots, \log|x_d|].
\]

\begin{proposition}
\label{propo:Simplex}
Let  $\eb = (e_i)$ be a basis of $V$ and 
let $\simplex\coloneqq  \simplex_{\eb}^{o} \subset \PP V$ be the associated open $(d-1)$-dimensional simplex.
Then the map $\log_{\eb} \from \simplex \to \Aa_\Lambda$ sends the Hilbert
pseudo-metric $d_\simplex$ to $\dhex$
and thus descends to a well-defined
isometry  
\[\log_{\eb} \from (X_\simplex,d_\simplex) \to (\Aa_\Lambda,\dhex)\]
from the associated Hilbert metric space to the model flat.
\end{proposition}
\begin{proof}
Let $\PP(x),\PP(y) \in \simplex$ be two distinct points,  and $\eb^*=(e_i^*)$ the dual basis of $\eb$.
Note that $e_i^* \in \simplex^*$ for all $i=1,\ldots,d$.
In fact,  they correspond to the hyperplanes bounding $\simplex$,  and hence the maximum of the cross-ratio will be attained for the linear forms in the dual basis.
For all $i, j \in \{1,\ldots,d\}$ we have
\[ \CR{e_i^*,\PP(x),\PP(y),e_j^*}=\frac{y_i}{x_i}\cdot \frac{x_j}{y_j},\]
and thus
\[d_\simplex(\PP(x),\PP(y)) = \max_{1\leq i,j\leq d} (\log\abs{y_i}-\log\abs{x_i}) - (\log\abs{y_j}-\log\abs{x_j}),\]
which is equal to $\nhex{\log_{\eb}(y)-\log_{\eb}(x)}$.
Thus $\log_{\eb}$ preserves pseudo-distances.
Furthermore,  $\log_{\eb}$ is surjective onto $\Aa_\Lambda$,  and induces hence the desired isometry.
\compl{
Alternative proof:
The map $f \from \simplex \to \Aa$,  $\PP(x) \mapsto [\log|x_1|, \ldots, \log|x_d|]$ is well-defined and surjective onto $\Aa_\Lambda$.
Furthermore $f$ preserves distances,  and induces thus the isometry $\log_{\eb} \from X_\simplex \to \Aa_\Lambda$.
Indeed, let $x_0 = \sum_{i=1}^d e_i$ and $x =\sum_{i=1}^d x_i e_i$ with $x_1 \geq \ldots \geq x_d >0$.
A computation shows that the line through $\PP(x_0)$ and $\PP(x)$ intersects the closed simplex $\overline{S}$ in $\PP(a)$ and $\PP(b)$,  where
\[a\coloneqq \sum_{i=1}^{d-1} \left(-1+\tfrac{x_i}{x_d}\right)e_i\; \textnormal{ and } b\coloneqq \sum_{i=2}^d \left(1-\tfrac{x_i}{x_1}\right)e_i\; .\]
The computation of the cross-ratio yields that
\[
d_{S}\left(\PP(x_0), \PP(x)\right)=\log \abs{\CR{\PP(a),\PP(x_0),\PP(x),\PP(b)}} =\log \frac{\abs{x_1}}{\abs{x_d}},
\]
which is equal to $\max_{i\neq j} \log|x_i| -\log |x_j|$, since $x_1 \geq \ldots \geq x_d >0$.
Every projective transformation preserving $S$ preserves distances. 
Thus given any two points $x_1, x_2 \in \simplex$, we can put them in the positions of $x_0$ and $x$,  which finishes the proof.
}
\end{proof}

In particular,  the map $\log_{\eb}$ maps the points $\PP(x) \in \simplex$ with $x_1\geq \ldots \geq x_d$ to $\Aap_{\Lambda}$.
Furthermore,  for $D \subset \{1,\ldots,d-1\}$,  the points of
\[\simplex^D_{\eb} \coloneqq \PP\big( \big\{ x=\sum_{i=1}^d x_i e_i \mid x_1\geq  \ldots \geq x_{d}>0\text{ and }x_j=x_{j+1} \, \forall \,  j \in \Delta \setminus D \big\}\big) \subset \simplex_\eb^o\]
are mapped to $\Aa_{D,\Lambda}$,  where
\[\Aa_{D,\Lambda} \coloneqq \{ [\alpha] \in \Aa_\Lambda \mid
\alpha_j=\alpha_{j+1} \textnormal{ for all } j  \in \Delta \setminus D\}.\]
We have $D' \subseteq D \subseteq \Delta$ if and only if $\Aa_{D',\Lambda} \subseteq \Aa_{D,\Lambda}$.

The results on simplices can be used in the study of non-Archimedean symmetric spaces.

\subsection{Symmetric spaces over real closed valued fields}
\label{subsection:SymSpaces}
For this subsection we assume that $\F$ is a real closed valued field with value group $\Lambda =\log |\F^\times| $.

Let us first consider the hyperbolic space over $\F$ as introduced and studied by \cite{Brumfiel_TreeNonArchimedeanHyperbolicPlane, Bouzoubaa_CompRealSpectrumCharVarSOn1}.
For this let $V$ be a finite-dimensional $\F$-vector space,  and let
$\B \subset \PP V$ be the open unit ball in an affine chart,  i.e.\
\[\B \coloneqq \big\{[x:1] \in f_\eb(\F^{d-1}) \mid \norm{x}_{2,\F} < 1\big\} \subset  \PP V, \]
where $\|x\|_{2,\F} \coloneqq \sqrt{x_1^2+\ldots+x_{d-1}^2} \in \F_{\geq 0}$ is the $\F$-valued \emph{$l^2$-norm} on $\F^{d-1}$.
Then $\B$ is a properly convex subset of the projective space over $\F$,  and
the associated Hilbert metric space $X_\B$ is a $\Lambda$-tree (see \cite[\S 2.1]{Chiswell_IntroductionLambdaTrees} for a definition).
Indeed,  the two-dimensional case follows from \cite[Theorem (28)]{Brumfiel_TreeNonArchimedeanHyperbolicPlane},  and in general from \cite[Theorem 1.5]{Bouzoubaa_CompRealSpectrumCharVarSOn1}.
Note that the latter considers the hyperboloid model of hyperbolic space over $\F$ (instead of the projective model),  but since these two models (up to multiplying $d_\B$ by a factor $1/2$) are related over $\R$ by a semi-algebraic homeomorphism,  the same holds true over any real closed field $\F$ by the Tarski--Seidenberg transfer principle \cite[\S 1.4]{BochnakCosteRoy_RealAlgebraicGeometry}.
Since for real closed valued fields $\Lambda$ is dense in $\R$,  we obtain that the metric completion of $X_\B$ is an $\R$-tree, i.e.\ a geodesic $0$-hyperbolic metric space,  see \cite{Imrich_MetricPropertiesTreeLikeSpaces} or \cite[\S 2,  Theorem 4.14]{Chiswell_IntroductionLambdaTrees}.
For the relation between $X_\B$ in dimension two and the boundary points of Thurston's compactification of Teichm\"uller space,  we refer to \cite[\S 7]{Brumfiel_RSCTeichmullerSpace}.

\smallskip
We now turn to the symmetric space of $\SL(n,\F)$.
Let $\Symm(n,\F)$ denote the vector space of symmetric $n \times n$-matrices.
Then
\[\PP\Symm_{>0}(n,\F) \coloneqq \{[M] \in \PP\Symm(n,\F)\mid M \textnormal{ is positive definite}\}\]
is open,  convex and bounded in an affine chart.
Note that $\PP\Symm_{>0}(n,\R)$ is the projective model for the symmetric space of $\SL(n,\R)$.

The group $\SL(n,\F)$ acts linearly on $\Symm(n,\F)$ by $g_*M \coloneqq  gMg^t$ for all $g \in \SL(n,\F)$ and $M \in \Symm(n,\F)$,  where $g^t$ denotes the transpose of the matrix $g$.
This action preserves positive definite matrices and descends to a projective transformation of $\PP\Symm_{>0}(n,\F)$,  which preserves the cross-ratio and hence also the Hilbert pseudo-metric.
As in the real case this action is transitive,  see e.g.\ \cite[Theorem 9]{Kaplansky_LinearAlgebraGeometry}.\compl{Here we need only Euclidean fields.}
The interior $\simplex^o$ of the simplex spanned by the elementary diagonal matrices $E_{ii} = \diag(0,\ldots,0,1,0,\ldots,0)$, with a $1$ at position $i$ for $i =1,\ldots,n$,  isometrically embeds in $\PP\Symm_{>0}(n,\F)$, i.e.\ $d_{\PP\Symm_{>0}(n,\F) |_{\simplex^o \times \simplex^o}} = d_{\simplex^o}$, since $\simplex^o$ is the intersection of $\PP\Symm_{>0}(n,\F)$ with the projectivization of the subspace of diagonal matrices (\Cref{lem:PropertiesdOmega}~(\ref{lem: intersectingWithSubspace})).

Recall that $\mathcal{O}$ is the valuation ring of $\abs{\cdot}$,  and that we write $\overline{x}$ for the class in $X_\Omega$ of $x\in \Omega$.
\begin{proposition}[$\SL(n,\F)$-symmetric space]
\label{propo:SymSpace}
The orbit map descends to a map 
\[f \from \SL(n,\F)/\SL(n,\mathcal{O}) \to X_{\PP\Symm_{>0}(n,\F)},  \quad \overline{g}\mapsto \overline{[gg^t]},\]
that is well-defined,  bijective and $\SL(n,\F)$-equivariant.
Furthermore,  for all diagonal matrices $a=\diag(a_1,\ldots,a_n) \in \SL(n,\F)$,  the map $f$ satisfies
\begin{align}
\label{eqn:distSymSpaceSLnF}
d_{\PP\Symm_{>0}(n,\F)}(f(\overline{\Id}), f(\overline{a})) =  \max_{i\neq j} \log \tfrac{|a_i^2|}{|a_j^2|}.
\end{align}
\end{proposition}
\begin{proof}
Set $\Omega\coloneqq\PP\Symm_{>0}(n,\F)$.
Since $\SL(n,\F)$ acts transitively on $\SL(n,\F)/\SL(n,\mathcal{O})$ it suffices to check that for all $g \in \SL(n,\mathcal{O})$ we have $d_{\Omega}([gg^t], [\Id]) = 0$.
Using the Cartan decomposition over real closed fields,  see e.g.\ \cite[Proposition 4.4]{BurgerIozziParreauPozzetti_RSCCharacterVarieties2},  we write $g$ as $g=kak'$ with $k,k' \in \SO(n,\F) \subset \SL(n,\mathcal{O})$ and $a=\diag(a_1,\ldots,a_n) \in \SL(n,\mathcal{O})$ diagonal. 
Then by \Cref{propo:Simplex} and the considerations right before \Cref{propo:SymSpace} we have
\begin{align*}
d_{\Omega}([gg^t], [\Id]) = d_{\Omega}([ka^2k^t], [\Id])  = d_{\simplex^o}([a^2], [\Id])=\log \max_{i\neq j} \tfrac{|a_i^2|}{|a_j^2|}.
\end{align*}
Since $a$ is diagonal and invertible,  it follows that $a_i \in \mathcal{O}^\times$,  and thus $|a_i|=1$ for all $i=1,\ldots,n$.
Thus $d_{\Omega}([gg^t], [\Id]) = 0$,  which proves that $f$ is well-defined.
The same computation also shows that $f$ satisfies \eqref{eqn:distSymSpaceSLnF}.

A similar argument shows that $\SL(n,\mathcal{O})$ is exactly the stabilizer of $\overline{[\Id]}$ in $X_{\Omega}$,  thus $f$ is injective.
The map $f$ is $\SL(n,\F)$-equivariant and hence surjective,  since $\SL(n,\F)$ acts transitively on $\Omega$ and hence also on $X_{\Omega}$.
\end{proof}

\begin{remarks}
\noindent\begin{enumerate}
\item In fact $X_{\PP\Symm_{>0}(n,\F)}$ is
the \emph{metric shadow} $\mathcal{B}_{\SL(n,\F)}=\PP\Symm_{>0}(n,\F)/\sim_N$ 
defined in \cite[\S
5.2]{BurgerIozziParreauPozzetti_RSCCharacterVarieties2},  
endowed with the $\F$-valued multiplicative norm $N$ given by $N(a_1,\ldots,a_n) = \max_{i\neq j} \tfrac{|a^2_i|_\F}{|a^2_j|_\F}$. 
Even more,  the $\F$-valued multiplicative Hilbert pseudo-distance $D_{ \PP\Symm_{>0}(n,\F)}$ (\Cref{subs:FvaluedMultHilbertDistance}) satisfies $D_{ \PP\Symm_{>0}(n,\F)}=N \circ \delta_\F$,  where $\delta_\F$ is the Cartan projection of $\SL(n,\F)$.
\item
The quotient $\SL(n,\F)/\SL(n,\mathcal{O})$ has the structure of an affine $\Lambda$-building in the sense of Bennett \cite{Bennett_AffineLambdaBuildingsI},  called the \emph{affine $\Lambda$-building associated to $\SL(n,\F)$},  see \cite{KramerTent_AffineLambdaBuildingsUltrapowers,Appenzeller_Thesis}.
It can be thought of as a non-Archimedean analogue of a symmetric space,  and the higher rank analogue of a $\Lambda$-tree.\footnote{
The quotient $\SL(n,\F)/\SL(n,\mathcal{O})$ can be endowed with an $\SL(n,\F)$-invariant function $d$ that satisfies $d(\overline{\Id}, \overline{a}) = 2 \log \max_{i\neq j}  \tfrac{|a_i|}{|a_j|}$ for a diagonal matrix $a=\diag(a_1,\ldots,a_n) \in \SL(n,\F)$.
\Cref{propo:SymSpace} then implies that $d$ is an $\SL(n,\F)$-invariant metric on $\SL(n,\F)/\SL(n,\mathcal{O})$.}
\end{enumerate}
\end{remarks}

\section{Hilbert geometry of non-Archimedean integral polytopes}
\label{section:Polytopes}

The goal of this section is to determine the Hilbert metric space associated to a non-Archimedean integral polytope.
For this we first prove a general result for convex sets sandwiched between two simplices whose vertices have coefficients of the same ``size'' (\Cref{subsection:FlagSimplexSandwichLemma}).
Secondly we show that for polytopes,  such a configuration can always be obtained (\Cref{subsection:ApproximationPolytopesSimplices}).
We begin by introducing in detail the objects in question.

\subsection{Preliminaries on polytopes}
\label{subsection:PreliminariesPolytopes}
Let $\FF$ be any ordered field and $V$ a $d$-dimensional vector space over $\F$.
A subset $\polyt \subset \PP V$ is a \emph{polytope} if $\polyt$ is the image under an affine chart $f_\eb \from \F^{d-1} \to \PP V$ of an affine polytope $A$ in $\FF^{d-1}$, namely the convex hull of a finite number of points.
For example,  simplices as defined in \Cref{subsection:ModelFlats} are polytopes.
A (closed) \emph{face} of a polytope $\polyt$ is the image under an affine chart of the intersection $\sigma=A \cap H$ of the affine polytope $A \subset \FF^{d-1}$ defining $\polyt$ and an affine hyperplane $H \subset \FF^{d-1}$ such that $A \cap H \neq \emptyset$ and $A$ lies in one of the closed half-spaces bounded by $H \cap \F^{d-1}$.
The notion of face is independent of the choice of affine chart containing $\polyt$.
The set of faces of $\polyt$ is denoted by $\Faces(\polyt)$, and it is partially
ordered by inclusion.
The dimension $\dim(\sigma)$ of a face $\sigma$ is the smallest dimension of a
projective subspace, denoted $\langle \sigma \rangle$, containing it. 
The $0$-dimensional faces of $\polyt$ are called its \emph{vertices}. 
Note that with this definition the polytope itself is a face.
The faces of polytopes are again polytopes,  and the faces of simplices are again (lower-dimensional) simplices.

A \emph{flag} of a polytope $\polyt$ is an increasing sequence of faces of $\polyt$.
The set of flags of $\polyt$ is denoted by $\Flags(\polyt)$, and the set of those containing $\polyt$ is denoted by $\Flags_*(\polyt)$.
If $F=(\sigma_1,\ldots,\sigma_k)$ and $F'= (\sigma'_1,\ldots,\sigma'_{k'})$ are two flags of $\polyt$,  we define their intersection $F\cap F'$ on the underlying subsets of faces.
The set of flags of $\polyt$ (as well as $\Flags_*(\polyt)$) is partially ordered by inclusion,  namely $F \prec  F'$ if and only if $F=F\cap F'$.
A flag $F$ is called \emph{maximal} if maximal for this order.
We call $\Flags(\polyt)$ the \emph{flag complex} associated to $P$.
We would like to give a different description of the flag complex in terms of the barycentric subdivision of $\polyt$ in a an affine chart.
This approach allows us to do computations in an affine chart,  however they do depend on a choice of such (even though the conclusions will not).

Recall that if $\simplex$ is a simplex associated to a basis $\eb=(e_1,\ldots,e_d)$ of $V$ as in \Cref{subsection:ModelFlats}, i.e.\ $\simplex=\simplex_\eb=\PP( \{ \sum_{i=1}^d x_i e_i \mid x_i \geq 0\})$,  then $\PP(e_i)$ are the vertices of $\simplex$.
We now define flag simplices,  following \cite[Definition 13]{VernicosWalsh_FlagApproxConvexBodiesVolumeGrowthHilbertGeom}.

\begin{definition}[Flag simplex]
A simplex $\simplex$ is a \emph{flag simplex} of a polytope $\polyt$ if there is a maximal flag $F=(\sigma_1,\ldots,\sigma_d)$ of $\polyt$ such that the relative interior of each face $\sigma_i$ contains exactly one vertex $s_i$ of $\simplex$.
\end{definition}

A special class of flag simplices of simplices are so-called barycentric simplices.
They play a crucial role in the flag simplex sandwich lemma (\Cref{lem:Intro:FlagSimplexSandwichLemma}).

\begin{definition}[Barycentric simplex]
A simplex $\simplex'$ is a \emph{barycentric simplex} of a simplex $\simplex$ \emph{with respect to a point} $s\in \Int{\simplex}$,  if there exists a basis $\eb=(e_1,\ldots,e_d)$ of $V$ such that $\simplex=\simplex_\eb$,
$s= \PP(e_1+\ldots+e_d)$ and $\simplex'=\simplex_{\eb'}$ with $\eb'=(e_1,e_1+e_2,\ldots,e_1+\ldots+e_d)$.
\end{definition}

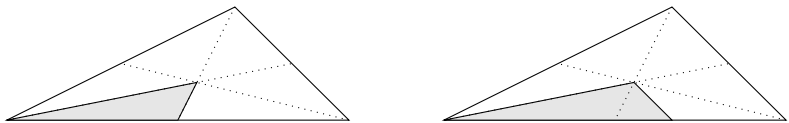
\begin{figure}[h]
\begin{tikzpicture}[scale=1.5]
\filldraw[fill=gray!0] (0,0)--(3,0)--(2,1)--cycle;
\filldraw[fill=gray!20] (1.6666,0.3333)--(0,0)--(1.5,0)--cycle;
\draw[dotted]  (0,0)--(2.5,0.5);
\draw[dotted]  (3,0)--(1,0.5);
\draw[dotted]  (2,1)--(1.5,0);
\end{tikzpicture}
\hspace{1cm}
\begin{tikzpicture}[scale=1.5]
\filldraw[fill=gray!0] (0,0)--(3,0)--(2,1)--cycle;
\filldraw[fill=gray!20] (1.6666,0.3333)--(0,0)--(2,0)--cycle;
\draw[dotted] (0,0)--(2.5,0.5);
\draw[dotted] (3,0)--(1,0.5);
\draw[dotted] (2,1)--(1.5,0);
\end{tikzpicture}
\caption{A barycentric simplex (left) and a flag simplex that is not a barycentric simplex (right).}
\end{figure}

\noindent If we do not specify the point $s$,  we just say that $\simplex'$ is a barycentric simplex of $\simplex$.
In particular,  a barycentric simplex is a flag simplex.

Note that barycentric simplices are exactly the image under an affine chart of maximal simplices in the barycentric subdivision of the affine simplex.
Given an  affine polytope $\polyt$ (the convex hull of finitely many points in $\F^{d-1}$), we denote by $b_\polyt$ its \emph{barycenter}, namely 
\[b_P\polyt = \tfrac{1}{k} \sum_{i=1}^k v_i \in \F^{d-1},\]
where $v_1,\ldots,v_k$ are the vertices of $\polyt$ (which is well defined as $\FF$ has characteristic $0$).

\begin{definition}[Barycentric subdivision]
The \emph{barycentric subdivision} $\Barc(\polyt)$ of $\polyt$ is the simplicial complex defined as follows.
To every flag $F=(\sigma_1,\ldots,\sigma_{k}) \in \Flags(\polyt)$ we associate the $(k-1)$-simplex $\simplex_F$ whose vertices are the barycenters of the $\sigma_i$.
Then $\Barc(\polyt)$ is the set of simplices $\simplex_F$,  $F\in\Flags(\polyt)$, ordered by inclusion. 
\end{definition}

The union of the $\simplex_F$ is equal to $\polyt$ and if $F, F' \in \Flags(\polyt)$ are two flags of $\polyt$ then 
$\simplex_{F} \cap \simplex_{F'} = \simplex_{F' \cap F}$.
Note that $\Barc(\polyt)$ is isomorphic to $\Flags(\polyt)$ as an abstract simplicial complex since we have $\simplex_{F'} \subseteq \simplex_{F} \iff F' \prec F$.
Note that $F \in \Flags_*(\polyt)$  if and only if $\simplex_F$ contains $b_\polyt$ as a vertex,  and we write $\Barc_*(\polyt)$ for the set of all barycentric simplices that contain $b_\polyt$ as a vertex.
If $\polyt$ is a simplex,  say $\polyt=\simplex_\eb$,  and $F \in \Flags_*(\polyt)$,  we have an explicit description of $\simplex_F$.
Indeed,  for $D \subset \Delta$ we defined in \Cref{subsection:ModelFlats}
\[\simplex^D_{\eb} = \PP\big( \big\{ x=\sum_{i=1}^d x_i e_i \mid x_1\geq  \ldots \geq x_{d}>0\text{ and }x_j=x_{j+1} \, \forall \,  j \in \Delta \setminus D \big\}\big).\]
With this notation we have $\simplex_F \cap \simplex_{\eb}^o=\simplex_{\sigma(\eb)}^{D(F)}$ for some  
permutation $\sigma$ of the order of the basis vectors,  or equivalently the vertices of $\simplex_\eb$,  where $\sigma(\eb)=(e_{\sigma(1)},\ldots,e_{\sigma(d)})$.

\subsection{%
\texorpdfstring{%
Approximation lemmas and proof of  \Cref{lem:Intro:FlagSimplexSandwichLemma}}%
{Approximation lemmas and proof of  Lemma D}}
\label{subsection:FlagSimplexSandwichLemma}
The goal of this section is to prove approximation lemmas for convex sets sandwiched in between two simplices whose vertices have coefficients of the same ``size''.

Let $\F$ be a non-Archimedean ordered valued field.
Recall that $\mathcal{O} = \{ x \in \F \mid |x| \leq 1\}$ denotes the valuation ring of $\abs{\cdot}$,  and $\mathcal{O}^\times = \{ x \in \F \mid |x| = 1\}$ its invertible elements.
We have seen in \Cref{subsection:ModelFlats} how to associate to a basis $\eb=(e_1,\ldots,e_d)$ of $V$ an open simplex $S^o_\eb=\PP( \{ \sum_{i=1}^d x_i e_i \mid x_i > 0\})$ and a map
\[ \log_{\eb} \from S_\eb^o \to \Aa_\Lambda, \quad \PP\big(\sum_{i=1}^d x_i e_i\big) \mapsto
  [(\log |x_1|, \ldots, \log|x_d|)].\]
We define the standard parabolic subgroup $P_D(d,\F)$ of
$\GL(d,\F)$ associated to a subset $D \subset \Delta=\{1,\ldots,d-1\}$,  say $D=\{d_1<\ldots<d_k\}$. 
Then $P_D(d,\F)$ is the group of invertible $d \times d$ block upper
triangular matrices with blocks of sizes $d_l-d_{l-1}$ for all
$l=1,\ldots,k$, where we set $d_0 \coloneqq 0$ and $d_{k+1}\coloneqq d$.
We also define $P_D(d,\mathcal{O}) \coloneqq P_D(d,\F) \cap \GL(d,\mathcal{O})$.

Given two simplices that agree on a common barycentric simplex,  we can give a sufficient condition on when the two logarithmic maps agree on this simplex.
\begin{lemma}[Base change in $P_D(d,\mathcal{O})$]
\label{lem:BaseChangeGL(O)}
Let $\eb=(e_1,\ldots,e_d)$ and $\eb'=(e'_1,\ldots,e'_d)$ be two bases of $V$, 
and let $M = \Mat_{\eb'}(\eb)$ be the matrix of $\eb$ in $\eb'$.
If $M \in P_D(d,\mathcal{O})$ and $\simplex^D_{\eb}=\simplex^D_{\eb'}$ for some $D \subseteq \Delta$,  then $\log_{\eb}=\log_{\eb'}$ on $\simplex^D_{\eb}=\simplex^D_{\eb'}\subset S^o_{\eb} \cap \subset S^o_{\eb'}$.
\end{lemma}
\begin{figure}[H]
\begin{tikzpicture}[scale=2.5]
\filldraw[fill=gray!20, semitransparent] (0,0) --(1,0) --  (1/2,3/2) -- cycle;
\draw[dotted] (0,0)--({(1+1/2)/2},{1.5/2});
\draw[dotted] (1,0)--(1/4,{1.5/2});
\draw[dotted] (1/2,{1.5})--(1/2,0);
\draw (.6,1.5) node{$\simplex_{\eb}$} ;

\filldraw[fill=gray!60, semitransparent] (0,0) --(0,1) --  (3/2, 1/2) -- cycle;
\draw[dotted] (0,0)--({(0+3/2)/2},{(1+1/2)/2});
\draw[dotted] (0,1)--(3/4,{0.5/2});
\draw[dotted] (3/2, 1/2)--(0,1/2);
\draw[very thick] (0,0)--(1/2,1/2);
\draw (-.4,1/4) node{$\simplex^D_{\eb}=\simplex^D_{\eb'}$} ;
\draw[->] (-.1,.24) .. controls +(0.1,0.1) and +(-0.1,0.1) .. (.18,.23);
\draw (1/2,1/2) node{$\bullet$} ;
\draw (1.4,.62) node{$\simplex_{\eb'}$} ;
\end{tikzpicture}
\begin{tikzpicture}[scale=2]
\filldraw[fill=gray!20, semitransparent] (0,0) --(1.5,0) --  ({1/2*1.5},3/2) -- cycle;
\draw[dotted] (0,0)--({(1+1/2)/2*1.5},{1.5/2});
\draw[dotted] ({1*1.5},0)--({1/4*1.5},{1.5/2});
\draw[dotted] ({1/2*1.5},{1.5})--({1/2*1.5},0);
\draw ({.2*1.5},1.5) node{$\simplex_{\eb'}$} ;

\filldraw[fill=gray!60, semitransparent] ({-1/4*1.5},0) --({3/4*1.5},0) --  ({1/2*1.5}, {2}) -- cycle;
\draw[dotted] ({-1/4*1.5},0)--({(9/16*3/4+1/4*1/2)/(13/16)*1.5},{(1/4*2)/(13/16)});
\draw[dotted] ({3/4*1.5},0)--({(-3/16*1/4+1/4*1/2)/(7/16)*1.5},{(1/2)/(7/16)});
\draw[dotted] ({1/2*1.5}, {2})--({1/2*1.5},0);
\draw[very thick] ({1/2*1.5},0)--({1/2*1.5},1/2);
\draw ({1/2*1.5},1/2) node{$\bullet$} ;
\draw ({.9*1.5},.62) node{$\simplex_{\eb}$} ;
\draw (1.9,1/4) node{$\simplex^D_{\eb}=\simplex^D_{\eb'}$} ;
\draw[->] (1.5,.24) .. controls +(-0.1,0.1) and +(0.1,0.1) .. (.8,.23);
\end{tikzpicture}
\caption{Configuration for $D=\{1\}$ (left) and $D=\{2\}$ (right) in \Cref{lem:BaseChangeGL(O)} in the case $d=3$.}
\label{fig_BaseChangeGL(O)}
\centering
\end{figure}
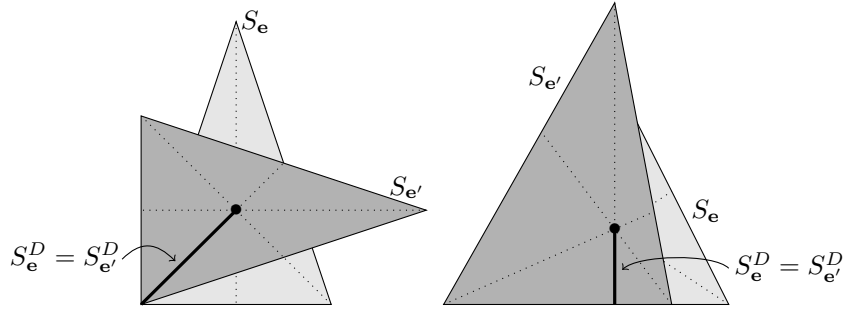
\begin{proof}
Let $D=\{d_1<\ldots<d_k\}$ and $d_{k+1} \coloneqq d$,  and set $v_j \coloneqq e_1+\ldots+e_{d_j}$ and $v'_j \coloneqq e'_1+\ldots+e'_{d_j}$ for all $j =1,\ldots,k+1$.
We first remark that $\simplex^D_{\eb}=\simplex^D_{\eb'}$ and $M \in P_D(d,\F)$ is equivalent to asking that $\PP(v_j)=\PP(v'_j)$ for all $j=1,\ldots,k+1$.
Thus there exists $\lambda_j \in \F$,  $\lambda_j \neq 0$ with $v_j = \lambda_j v'_j$ for all $j=1,\ldots,k+1$.
In other words,  $\lambda_j$ is an eigenvalue of $M$.
Since $M \in \GL(d,\mathcal{O})$,  it follows that $\lambda_j \in \mathcal{O}^\times$,  i.e.\ $\abs{\lambda_j}=1$ for all $j =1,\ldots,k+1$.

Let now $p =\PP(\sum_{j=1}^{k+1} x_j v_j) \in \simplex^D_{\eb}$ with $x_j >0$.
Then $p =\PP(\sum_{j=1}^{k+1} x_j \lambda_j v'_j) \in  \simplex^D_{\eb'}$.
Written in the basis $\eb$ we have $p=\PP(\sum_{i=1}^{d} y_i e_i)$ with $y_i=\sum_{j=1}^{k+1} x_j  \delta_i^j$ where 
\[\delta_i^j =\begin{cases} 1,  \textnormal{ if } i \leq d_j,  \\ 0,  \textnormal{ if } i > d_j. \end{cases}\]
We also get $p=\PP(\sum_{i=1}^{d} y'_i e'_i)$ with $y'_i=\sum_{j=1}^{k+1} \lambda_j x_j \delta_i^j$.
Hence taking absolute values and using the ultrametric triangle inequality gives 
\[\abs{y_i}=\max_{j=1,\ldots,k+1}\big\{|x_j \delta_i^j|\big\}=\max_{j=1,\ldots,k+1}\big\{|\lambda_j x_j \delta_i^j|\big\} \geq \abs{y'_i}, \]
as $\abs{\lambda_j}=1$ for all $j =1,\ldots,k+1$.
Reversing the roles of $\eb$ and $\eb'$ yields the reverse inequality.
Thus we have
\[\log_{\eb}(p) = (\log(\abs{y_i}))_{i=1}^d=(\log(\abs{y'_i}))_{i=1}^d=\log_{\eb'}(p).\qedhere\]
\end{proof}

If the base change matrix is non-negative we can even say more.

\begin{lemma}[Non-negative base change in $P_D(d,\mathcal{O})$]
\label{lem:BaseChangeUO}
Let $\eb=(e_1,\ldots,e_d)$ and $\eb'=(e'_1,\ldots,e'_d)$ be two bases of $V$,  and let $M = \Mat_{\eb'}(\eb)$ be the matrix of $\eb$ in $\eb'$.
If $M \in P_{D}(d,\mathcal{O})$ and $M$ is non-negative,  i.e.\ $M$ has only non-negative coefficients,  
then $\log_{\eb} = \log_{\eb'}$ on $S^D_{\eb}\subseteq S^D_{\eb'}$.
\end{lemma}

\begin{remark}
\label{rem:GeometricInterpretationBaseChangeLemma}
One can interpret the above lemma,  as well as the assumptions on $M$,  geometrically; see \Cref{fig_GeomInterpretationSandwichLemma}.
Indeed,  asking for $M$ to be non-negative is equivalent to the fact that the simplex $S_{\eb}$ is contained in the simplex $\simplex_{\eb'}$.
Furthermore,  $M$ being block upper triangular means that there are
two flags $F=(\sigma_1,\ldots,\sigma_k, \simplex_\eb)$ of $\simplex_\eb$ and $F'=(\sigma'_1,\ldots,\sigma'_k,\simplex_{\eb'})$ of
$\simplex_{\eb'}$
such that $\langle \sigma_i \rangle =\langle \sigma'_i \rangle$ with 
$d_i=\dim \langle \sigma_i \rangle +1$ for all $i=1,\ldots, k$.
The subset $\simplex_{\eb}^D$ on which $\log_{\eb}$ and $\log_{\eb'}$ agree,  corresponds to the \xfchanged{intersection of $\simplex^o_{\eb}$ with the barycentric simplex}
of $\simplex_{\eb}$ corresponding to the flag $F$.

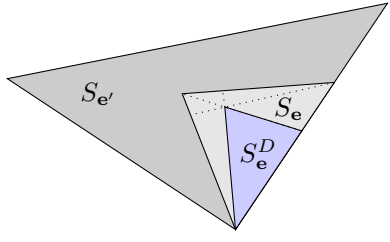
\begin{figure}[h]
\begin{tikzpicture}[scale=1]
\filldraw[fill=gray!40] (0,2) --(3,0) --  (5,3) -- cycle;
\draw (1.2,1.8) node{$\simplex_{\eb'}$} ;

\filldraw[fill=gray!20] (3,0) --(4.3,1.95) --  (2.3,1.8)-- cycle;
\draw (3.7,1.6) node{$\simplex_{\eb}$} ;

\filldraw[fill=blue!20] (3,0) -- (3.872,1.308) -- (2.86,1.625) -- cycle;
\draw (3.3,1) node{$\simplex_{\eb}^D$} ;

\draw[dotted] (3.872,1.308)--(2.3,1.8);
\draw[dotted] (4.3,1.95)--(2.41,1.52);
\draw[dotted] (2.86,1.625)--(2.83,1.85);

\end{tikzpicture}
\caption{Geometric interpretation of \Cref{lem:BaseChangeUO}.}
\label{fig_GeomInterpretationSandwichLemma}
\centering
\end{figure}
\end{remark}

\begin{proof}
Let $p = \PP(\sum_{j=1}^d x_j e_j) \in S^D_{\eb}$ with $x_j >0$ for all $j=1,\ldots,d$.
Then $p = \PP(\sum_{i=1}^d x'_i e'_i)$ with $x'_i = \sum_{j=1}^d m_{ij} x_j \in \F$.
By the assumption that $M$ has only non-negative coefficients and $x_j >0$ for all $j$,  we have $x'_i \geq 0$ for all $i=1,\ldots,d$.
We claim that $|x_i|=|x'_i|$ for all $i=1,\ldots,d$.

Fix thus $l \in \{1,\ldots, k\}$ and $i \in \{d_l+1,\ldots, d_{l+1}\}$ 
and write $c_l \coloneqq |x_j|$ for all $j \in \{d_l+1,\ldots, d_{l+1}\}$. 
We also set $n_l\coloneqq d_l-d_{l-1}$.
We have
\begin{align*}
 |x'_i| &= \max_{1\leq j \leq d} \{ |m_{ij}| \cdot |x_j| \} \geq \max_{d_l+1\leq j \leq d_{l+1}} \{ |m_{ij}| \cdot |x_j| \} \\
&= c_l \cdot \max_{d_l+1\leq j \leq d_{l+1}} \{ |m_{ij}| \} = c_l,
\end{align*}
because the $(n_{l+1})\times (n_{l+1})$-block matrix 
on the diagonal is in $\GL(n_{l+1}\mathcal{O})$,  
hence in every row of this submatrix there is at least one coefficient in $\mathcal{O}^\times$.
On the other hand we have
\begin{align*}
|x'_i| &= \max_{1\leq j \leq d} \{ |m_{ij}| \cdot |x_j| \} = \max_{d_{l+1}\leq j \leq d} \{ |m_{ij}| \cdot |x_j| \} \\
&\leq  \max_{d_{l+1}\leq j \leq d} \{ |x_j| \} = |x_{d_{l+1}}| = c_l,
\end{align*}
where the first equality holds since $M$ is block upper triangular.  
The inequality is true since $M$ has coefficients in $\mathcal{O}$,  and the order-compatibility of $\abs{\cdot}$ implies that $|x_j| \leq |x_{d_{l+1}}|$ for all $j \geq d_{l+1}$.

Thus $|x_i|=|x'_i|$ for all $i=1,\ldots,d$.
In particular we obtain that the $x'_i$ are strictly positive,  hence $p \in S^o_{\eb'}$,  and thus
\[\log_{\eb'}(p) = [(\log|x'_i|)_{i=1}^d]=[(\log|x_i|)_{i=1}^d] =\log_{\eb}(p).\qedhere\]
\end{proof}

We now apply the above lemma to prove an approximation property for convex sets sandwiched in between two simplices.
Recall that we denote by $\pi$ the projection map from the convex set $\Omega$ to the associated Hilbert metric space $X_\Omega$.

\begin{lemma}[Flag simplex sandwich lemma, \Cref{lem:Intro:FlagSimplexSandwichLemma}]
\label{lem:FlagSimplexSandwichLemma}
Let $\Omega\subset \PP V$ be an open convex subset.
Let $\simplex$,  $\simplex_\eb$ and $\simplex_{\eb'}$ be three simplices such that $\simplex \subset \simplex_\eb \subset \overline{\Omega} \subset \simplex_{\eb'}$, $\simplex$ is a flag simplex of both $\simplex_\eb$ and $\simplex_{\eb'}$,  and $\simplex$ is a barycentric simplex of $\simplex_\eb$.
If $\Mat_{\eb'}(\eb) \in \GL(d,\mathcal{O})$,  then
\begin{enumerate}
\item $\log_{\eb}=\log_{\eb'}$ on $\simplex \cap \simplex_\eb^o$,  and
\item $d_\Omega = d_{\simplex^o_\eb} = d_{\simplex^o_{\eb'}}$ on $\simplex \cap \Omega \times \simplex \cap \Omega$.
\end{enumerate}
In particular,  $\pi(\simplex \cap \Omega) \subset X_\Omega$ is isometric to $\Aap_\Lambda$.
\end{lemma}

\begin{figure}[h]
\begin{tikzpicture}[scale=1.2]
\filldraw[fill=gray!40] (0,2) --(3,0) --  (5,3) -- cycle;
\draw (1,1.8) node{$\simplex_{\eb'}$} ;

\filldraw[fill=gray!30] (3,0)-- (4.5,2.25)  .. controls +(-.5,.2) and +(0,0) ..   (2.5,2.3)-- (1.5,1.875) .. controls +(-.1,-.2) and +(0,0) ..  cycle;
\draw (2,1.7) node{$\overline{\Omega}$} ;

\filldraw[fill=gray!20] (3,0) --(4.3,1.95) --  (2.3,1.8)-- cycle;
\draw (3.7,1.6) node{$\simplex_\eb$} ;

\filldraw[fill=blue!20] (3,0) -- (3.872,1.308) -- (2.86,1.625) -- cycle;
\draw (3.2,1) node{$\simplex$} ;

\draw[dotted] (3.872,1.308)--(2.3,1.8);
\draw[dotted] (4.3,1.95)--(2.41,1.52);
\draw[dotted] (2.86,1.625)--(2.83,1.85);
\end{tikzpicture}
\caption{Sandwiching $\overline{\Omega}$ between two simplices.}
\label{fig_SandwichLemmaProof}
\centering
\end{figure}

\begin{proof}
The first claim follows almost immediately from \Cref{lem:BaseChangeUO}.
To verify the conditions in \Cref{lem:BaseChangeUO} we use the considerations in \Cref{rem:GeometricInterpretationBaseChangeLemma}.
Indeed,  since $\simplex_\eb \subset \simplex_{\eb'}$ we have that $M \coloneqq \Mat_{\eb'}(\eb)$ is non-negative.  
Furthermore,  $M$ is upper triangular as $\simplex$ is a flag simplex of both $\simplex_{\eb}$ and $\simplex_{\eb'}$,  hence $M \in P_D(d,\F)$ with $D=\{1,\ldots,d-1\}$.
By assumption we have that $M \in P_D(d,\mathcal{O})$.
Thus we can apply \Cref{lem:BaseChangeUO},  which tells us that $\log_{\eb}=\log_{\eb'}$ on $\simplex_{\eb}^D=S \cap \simplex^o_{\eb}$.

Since $\simplex^o_{\eb} \subset \Omega \subset \simplex^o_{\eb'}$,  we already know that $d_{\simplex^o_{\eb'}} \leq d_\Omega \leq d_{\simplex^o_{\eb}}$ by \Cref{lem:PropertiesdOmega}~(\ref{lem:PropertiesdOmega2:Sandwich}).
However on $S\cap \Omega$ the pseudo-distances $d_{\simplex^o_{\eb'}}$ and $d_{\simplex^o_{\eb}}$ agree,  since $\log_{\eb} =\log_{\eb'}$ on $S \cap \simplex^o_{\eb}$,  which we know to be isometries by \Cref{propo:Simplex}.
Hence $d_{\simplex^o_{\eb'}}=d_\Omega=d_{\simplex^o_{\eb}}$ on $S\cap \Omega \times S \cap \Omega$.
\end{proof}

\subsection{Approximation of polytopes by simplices}
\label{subsection:ApproximationPolytopesSimplices}
The next lemma tells us that if $\Omega$ is the interior of a polytope $P$ over any ordered field we can always find simplices sandwiching $\overline{\Omega}=\polyt$ as in \Cref{lem:Intro:FlagSimplexSandwichLemma},  such that the barycenter (in some affine chart) is contained in the inner simplex.
The proof is inspired by \cite[Lemma 16]{VernicosWalsh_FlagApproxConvexBodiesVolumeGrowthHilbertGeom},  in which the authors prove a weaker version of this claim in the case $\F=\R$.

Let $\F$ now be any ordered field and $V$ a $d$-dimensional vector space over $\F$.

\begin{lemma}
\label{lem:SandwichLem}
Let $\polyt \subset \PP V $ be a polytope. 
Fix an affine chart containing $\polyt$ and let $b_\polyt$ be the  barycenter
and $\Barc(\polyt)$ be the barycentric subdivision of $\polyt$ in this affine chart.
Let $\simplex$ be a maximal simplex of $\Barc(\polyt)$.
Then there exist two simplices $\simplex_{\eb}, \simplex_{\eb'} \subset \PP V$ such that $\simplex_{\eb} \subset \polyt \subset \simplex_{\eb'}$, $\simplex$ is a flag simplex of both $\simplex_{\eb}$ and $\simplex_{\eb'}$,  and $\simplex$ is a barycentric simplex of $\simplex_{\eb}$ with respect to $b_\polyt$.
\end{lemma}

\begin{figure}[h]
\begin{tikzpicture}[scale=1.2]
\filldraw[fill=gray!40] (0,2) --(3,0) --  (5,3) -- cycle;
\draw (1,1.8) node{$\simplex_{\eb'}$} ;

\filldraw[fill=gray!30] (3,0) --(4.5,2.25) --  (4,2.5)-- (1.5,1.875) -- (1.3,1.5) -- cycle;
\draw (2,1.7) node{$\polyt$} ;

\filldraw[fill=gray!20] (3,0) --(4.3,1.95) --  (2.3,1.8)-- cycle;
\draw (3.8,1.6) node{$\simplex_{\eb}$} ;

\filldraw[fill=blue!20] (3,0) -- (3.872,1.308) -- (2.86,1.625) -- cycle;
\draw (3.2,0.9) node{$\simplex$} ;

\draw (2.86,1.625) node{$\bullet$} ;
\draw (2.84,1.38) node[right]{$b_\polyt$} ; 

\draw[dotted] (3.872,1.308)--(2.3,1.8);
\draw[dotted] (4.3,1.95)--(2.431,1.55);
\draw[dotted] (2.86,1.625)--(2.83,1.85);

\end{tikzpicture}
\caption{The simplices $\simplex_{\eb},\simplex,\simplex_{\eb'}$ and the polytope $\polyt$ for $d=3$.}
\label{fig_SandwichLemma}
\centering
\end{figure}

\begin{proof}
We can assume that $\polyt \subset \F^{d-1}$ is $(d-1)$-dimensional,  and convex hulls are always taken in this fixed affine chart.
We proceed by induction on $d$.
If $d=2$, then $\polyt$ is a segment,  hence a simplex and we can take $\Sin=\polyt=\Sout$.
Let now $d>2$ and assume it holds true for $d-2$.
We first prove the existence of $\Sin$.
The vertices of $\simplex$ except $b_\polyt$ form a flag simplex $\simplex_\sigma$ of a face $\sigma$ of $\polyt$.
By induction hypothesis there exists a simplex $\simplex' \subset \sigma$ such that $\simplex_\sigma$ is a barycentric simplex of $\simplex'$ with respect to the barycenter $b_\sigma$ of $\sigma$.
Choose any $q'\in \polyt$ on the (affine) line through $b_\sigma$ and $b_\polyt$ outside of the segment between $b_\sigma$ and $b_\polyt$.
\footnote{We can take $q'$ to be the barycenter of the vertices of $\polyt$ not in $\sigma$.}
Then $\Sin \coloneqq \conv(\simplex',q')$ is such that $\Sin \subset \polyt$ is a simplex and $\simplex$ is a barycentric simplex of $\Sin$ with respect to $b_\polyt$,  because $q'$,  $b_\polyt$ and $b_\sigma$ are aligned.

\begin{figure}[H]
\begin{tikzpicture}[scale=1]
\filldraw[fill=gray!20] (0,0)--(3,0)--(2,1)--cycle;
\filldraw[fill=gray!40] (1.6666,0.3333)--(0,0)--(1.5,0)--cycle;
\draw (0,0)--(2,4);
\draw (3,0)--(2,4);
\draw (2,1)--(2,4);
\draw (0,0)--(2.5,0.5);
\draw (3,0)--(1,0.5);
\draw (2,1)--(1.5,0);
\draw (1.6666,0.3333) node{$\bullet$};
\draw (1.6666,0.3333)--(2,4);
\draw (1.75,1.25) node{$\bullet$};
\draw (1.75,1.25) node[left]{$b_\polyt$};
\draw (1.76,0.6) node[right]{$b_\sigma$};
\draw (2,4) node[above]{$q$};
\draw (0.8,0) node[below]{$\simplex_\sigma$};
\draw (1,2) node[left]{$\Sin$};
\end{tikzpicture}
\caption{The induction step for $\Sin$.}
\label{fig_InductionStepSin}
\end{figure}

We now turn to the proof of the existence of $\Sout$.
As in the construction of $\Sin$,  the vertices of $S$ except $b_\polyt$ form a flag simplex $\simplex_\sigma$ of a face $\sigma$ of $\polyt$.
By induction hypothesis there exists a simplex $\simplex''$ such that $\sigma \subset \simplex''$ and $\simplex_\sigma$ is a flag simplex of $\simplex''$.
Let $o$ be the vertex of $\polyt$ that is also a vertex of $S$.
Give $\F^{d-1}$ a vector space structure by putting $o$ as the origin.
We can assume that $\sigma$ and $\simplex''$ lie in the subspace $\F^{d-2} \times \{0\} \subset \F^{d-1}$, that $\simplex''$ is the standard simplex in $\F^{d-2}$,  i.e.\  $\simplex''=\{(x_1,\ldots,x_{d-2}) \in (\F_{\geq 0})^{d-2} \mid  \sum_{i=1}^{d-2} x_i \leq 1 \}$,  and that $\polyt \subset \F^{d-2} \times \F_{\geq 0} $ by convexity.
For $u \in (\F_{>0})^{d-2}$ the linear map 
\[L_u \from \F^{d-2} \times \F \to \F^{d-2} \times \F,  \; (x,t) \mapsto (x+tu, t)\]
fixes $\F^{d-2}\times \{0\}$ pointwise.
Since $\polyt$ is bounded,  there exists $u_0$ such that $L_{u_0}(\polyt) \setminus \sigma \subset (\F_{> 0})^{d-1}$.
Scaling yields that there exist $\alpha \in \F_{\geq 1}$ and $q'' \in \F^{d-1}$ such that
\[L_{u_0}(\polyt) \subset \conv(\alpha\simplex'',  q'' ).\]
Then $\Sout \coloneqq L_{u_0}^{-1}(\conv(\alpha\simplex'',  q'' ))$ is a simplex that contains $\polyt$ and that contains $\simplex$ as a flag simplex.

\begin{figure}[H]
\begin{tikzpicture}[scale=1.3]
\draw (5.9,0.1)--(7.9,4.1)--(13.9,4.1)--(11.9,0.1)--cycle;
\draw (11.9,0.1) node[right]{$\F^{d-2}\times \{0\}$};
\draw (7.2,4.5) node[left]{$\{0\}\times\F$};

\draw (7.3,0.1) -- (9.3,4.1);
\draw (6.1,0.5)--(12.1,0.5);
\draw[->] (7.5,0.5) -- (7.5,4.5);

\filldraw[fill=gray!40] (7.5,0.5)--(11.7,0.5)--(9.1,3.7)--cycle;
\draw (10.7,1.2) node{$\simplex''$};

\filldraw[fill=gray!20] (7.5,0.5)--(8.7,0.5)--(9.5,1)--(10.5,1.8)--(9.5,1.8)--(8.3,1.3)--cycle;
\draw (9.5,1.5) node{$\sigma$};

\filldraw[fill=blue!20] (7.5,0.5)--(8.1,0.5)--(9.2,1.3)--cycle;
\draw (8,.6) node{\footnotesize{$\simplex_\sigma$}};

\draw[dotted] (11.7,0.5)--(13.2,0.5)--(9.6,4.8)--(9.3,4.1);
\draw (12.3,2) node[above]{$\alpha\simplex''$};

\draw (7.5,0.5)--(7.7,1.5);
\draw (8.7,0.5)--(8.9,1.6);
\draw (9.5,1)--(10,2);
\draw (10.5,1.8)--(11,2.8);
\draw (9.5,1.8)--(9.7,2.8);
\draw (8.3,1.3)--(8.5,2.3);

\fill[fill=gray!20,semitransparent] (7.5,0.5)--(7.7,1.5)--(8.9,1.6)--(8.7,0.6);
\fill[fill=gray!20,semitransparent] (8.9,1.6)--(8.7,0.6)--(9.5,1)--(10,2);
\fill[fill=gray!20,semitransparent] (9.5,1)--(10,2)--(11,2.8)--(10.5,1.8);
\fill[fill=gray!20,semitransparent] (11,2.8)--(10.5,1.8)--(9.5,1.8)--(9.7,2.8);
\fill[fill=gray!20,semitransparent] (9.5,1.8)--(9.7,2.8)--(8.5,2.3)--(8.3,1.3);
\fill[fill=gray!20,semitransparent] (8.5,2.3)--(8.3,1.3)--(7.5,0.5)--(7.7,1.5);

\draw (7.5,0.5) node{$\bullet$};
\draw (7.45,0.35) node[left]{$o$};
\draw (11,3) node[left]{$L_{u_0}(\polyt)$};
\draw (8,4.5) node{$\bullet$};
\draw (8,4.5) node[above]{$q''$};

\draw[dotted] (7.5,0.5)--(8,4.5)--(12.3,0.9);
\draw[dotted] (8,4.5)--(9.6,4.8); 

\end{tikzpicture}
\caption{The induction step for $\Sout$.}
\label{fig_InductionStepSout}
\end{figure}
\end{proof}

\subsection{The geometric realization of $P$ over a Weyl chamber}
\label{subsection:DefinitionModelSpace}
Let $\F$ be an ordered valued field,  and $V$ a $d$-dimensional $\F$-vector space.
The goal of this subsection is to construct,  given a polytope $\polyt\subset \PP V$,  the geometric realization of $\Flags_*(\polyt)$ (or equivalently $\Barc_*(\polyt)$ in a fixed affine chart) modeled on the model Weyl chamber $\Aap_\Lambda$ (see \Cref{subsection:ModelFlats}) for $\Lambda = \log\abs{\F^\times} \subset \R$ the value group. 
Recall that we denote by $\Flags_*(\polyt)$ the set of all flags of $\polyt$ that contain $\polyt$,  and by $\Barc_*(\polyt)$ the set of all barycentric simplices of $\polyt$ that contain $b_\polyt$ (in a fixed affine chart).
Let us now make this more precise.

We use the notations from \Cref{subsection:ModelFlats}.
Recall that $\Aa_\Lambda=\Lambda^d/\Lambda(1,\ldots,1)$ is the $\Lambda$-model flat and $\Aap_\Lambda=\{[\alpha] \in \Aa_\Lambda \mid \alpha_1 \geq \ldots \geq \alpha_d\}$ a $\Lambda$-model Weyl chamber.
For a subset $D \subset \Delta=\{1,\ldots,d-1\}$ we have $\Aa_{D,\Lambda} = \{ [\alpha] \in \Aa_\Lambda \mid \alpha_i=\alpha_{i+1}$ for all $i \in \Delta \setminus D\}$,  and we define $\Aap_{D,\Lambda} \coloneqq \Aa_{D,\Lambda} \cap \Aap_\Lambda$,  called a \emph{face} of $\Aap_\Lambda$.
For a flag $F = (\sigma_1, \ldots,\sigma_k,P) \in \Flags_*(\polyt)$ containing the face $\polyt$ we define its \emph{type} $D(F)$  as
\[ D(F) \coloneqq \{d_{\sigma_1},\ldots,d_{\sigma_k}\}  \subseteq \Delta,\]
where for a face $\sigma$ of $\polyt$ we set $d_\sigma = \dim(\sigma)+1$.
Now the underlying set of $K$ is obtained by gluing copies of $\Aap_\Lambda$ according to the combinatorial data of the flag complex of $\polyt$.
More precisely,  we define the following.

\begin{definition}[$\Aap_\Lambda$-geometric realization of $\Flags_*(\polyt)$]
We set
\[\Model \coloneqq \Model(\polyt,\Lambda) \coloneqq \Big(\bigsqcup_{F \in \Flags_*(\polyt)} \Aap_{D(F),\Lambda} \times \{F\}\Big)_{\big/{\sim}},\]
where $\sim$ is the equivalence relation generated by $(x,F) \sim (x,F')$ if $F \prec F'$.
\end{definition}

\noindent Note that when $\Lambda=\R$ the space $\Model$ is a polyhedral fan of Weyl chambers.
We write $\overline{(x,F)}$ for the equivalence class of $(x,F)$.
Denote by $o=\ov{(0,F)}$ the cone point of $\Model$.

We have a projection $\projM \from \bigsqcup_{F \in \Flags_*(\polyt)} \Aap_{D(F),\Lambda} \times \{F\} \to \Model$,  and we define for all $F \in \Flags_*(\polyt)$
\[\projM_F \from \Aap_{D(F),\Lambda} \hookrightarrow\bigsqcup_{F \in \Flags_*(\polyt)} \Aap_{D(F),\Lambda} \times \{F\} \to \Model\]
as the composition of $\alpha$ with the inclusion.
Then $\projM_F$ is a bijection onto its image $K_F$,  and $K=\cup_{F \in \Flags_*(\polyt)}K_F$.
One should think of the maps $\projM_F$ as charts to
the model space $K$.

The next step is to define a metric on $\Model$.
For this we follow a general construction for defining metrics on disjoint unions and quotients of metric spaces as in \cite[Chapter I.5, 5.18, 5.19]{BridsonHaefliger_MetricSpacesNonPosCurv}.
{We quickly recall it here.
If $(X_i,d_i)_{i \in I}$ is a family of metric spaces,  and $X \coloneqq \bigsqcup_{i \in I} X_i \times \{i\}$ is their disjoint union,  one can define a distance function $d$ on $X$ by setting
\[d((x,i),(x',i'))=\begin{cases}d_i(x,x') &\textnormal{ if } i=i',\\
\infty &\textnormal{ if } i\neq i'.\end{cases}\]
If $(X,d)$ is a metric space and $\sim$ an equivalence relation on $X$,  we can define the \emph{quotient pseudo-metric} $\overline{d}$ on the set of equivalence classes $\overline{X}=X/\sim$ as follows.
For $\overline{x}, \overline{y}\in \overline{X}$,
a \emph{chain} joining $\overline{x}$ to $\overline{y}$ is a sequence $C=(x_1,y_1,\ldots,x_n,y_n)$ (for some $n \in \N$) of points of $X$ satisfying $x_1 \in \overline{x}$,  $y_n \in \overline{y}$,  and $y_k \sim x_{k+1}$ for all $k=1,\ldots, n-1$.
Its \emph{length} is $l(C) \coloneqq \sum_{k=1}^n d(x_k, y_k)$. 
We set
\[\overline{d}(\overline{x},\overline{y}) = \inf_{C \textnormal{ chain from }\overline{x} \textnormal{ to }\overline{y}} l(C).\]

Let us get back to our setting.}
On each {$\Aap_{D(F),\Lambda}$} we have the metric $\dhex$,  and we put the quotient pseudo-metric $\dquot$ on $\Model$.
In fact we have the following.

\begin{proposition}
The quotient pseudo-metric $\dquot$ on $\Model$ is a metric.
\end{proposition}

\begin{proof}
We apply \cite[Corollary 5.28]{BridsonHaefliger_MetricSpacesNonPosCurv} to show that $\dquot$ is a metric.
For this we need to verify the conditions in \cite[Lemma 5.27]{BridsonHaefliger_MetricSpacesNonPosCurv}, i.e.\  we need to show that for every $\overline{(x,F)} \in \Model$ there exists $\epsilon(\overline{(x,F)})=\epsilon>0$ such that
\begin{enumerate}
\item 
\label{item:IndEpsBalls}
for all $(x,F), (x,F') \in \overline{(x,F)}$ and all $(z,F) \in B((x,F), \varepsilon)$,  $(z,F') \in B((x,F'),\varepsilon)$ with $(z,F)\sim (z,F')$,  we have 
\[\dmodel((x,F), (z,F)) = \dmodel((x,F'), (z,F'));\textnormal{ and}\]
\item 
\label{item:UnionEquivClasses}
$X_\varepsilon \coloneqq \bigcup_{(x,F)\in \overline{(x,F)}} B((x,F), \varepsilon)$ is a union of equivalence classes.
\end{enumerate}
Note that (\ref{item:IndEpsBalls}) is satisfied for any $\varepsilon>0$,  since the distance on $\Aap_{D(F),\Lambda}$ is just the restriction of the distance $\dhex$ on $\Aap_\Lambda$ for any $F \in \Flags_*(\polyt)$.

If $x=0$,  we can choose any $\varepsilon>0$.
Assume now that $x\neq 0$, and choose $(x,F') \in \overline{(x,F)}$ with $F'$ maximal, i.e.\ $\Aap_{D(F'),\Lambda}=\Aap_\Lambda$.
Set 
\[\varepsilon \coloneqq \min_{ D\subset \Delta,\,x \notin \Aap_{D,\Lambda}} \tfrac{1}{2} \dhex(x, \Aap_{D,\Lambda}) >0\]
to be half of the minimal distance of $x$ to the faces of $\Aap_\Lambda$ not containing $x$.
Then $\varepsilon$ does not depend on the choice of $(x,F)$ in its equivalence class,  as any other choice of maximal flag leads to the same $\varepsilon$.

Let us verify that $\varepsilon$ satisfies (\ref{item:UnionEquivClasses}).
Let $(x,F) \in \overline{(x,F)}$ and $(z,F)\in B((x,F),\varepsilon)$.
We show that any $(z,F') \sim (z,F)$ is contained in $X_\varepsilon$,  which shows that $X_\varepsilon$ is a union of equivalence classes.
Since $(z,F') \sim (z,F)$ we have that $z \in \Aap_{D(F \cap F'),\Lambda}$.
By the choice of $\varepsilon$,  it follows that $\dmodel((x,F), (z,F)) <\min_{ D\subset \Delta,\,x \notin \Aap_{D,\Lambda}} \tfrac{1}{2} \dhex(x, \Aap_{D,\Lambda})$. 
In particular,  if $z \in \Aap_{D,\Lambda}$ for some $D \subset \Delta$,  then also $x \in \Aap_{D,\Lambda}$.
Thus we get that $x \in \Aap_{D(F \cap F'),\Lambda}$,  and hence $(z,F') \in X_\varepsilon$.
\end{proof}

From now on we denote the metric on $\Model$ obtained this way by $\dmodel$.
Then for all $F \in \Flags_*(\polyt)$ the map $\projM_F$ defined above is an isometry onto its image $K_F$.

\subsection{%
\texorpdfstring{Proof of \Cref{thm:Intro:NAPolytopalHilbertMetricSpace}}%
{Proof of Theorem C}}

\label{subsection:ProofHilbertMetricSpacePolytope}
Let $\K$ be an ordered field. 
Given an ordered field extension $\F$ of $\K$,  we define for an affine polytope $\Kpolyt \subset \K^{d-1}$ its \emph{$\F$-extension} 
\[\Kpolyt_\F \coloneqq \conv_\F(\Kpolyt)\subset \F^{d-1}\]
as the $\F$-convex hull of $\Kpolyt$.
Then $\Fpolyt \coloneqq \Kpolyt_\F$ is an affine polytope in $\FF^{d-1}$ with the same vertices as $\Kpolyt$,
using the natural inclusion $\K^{d-1} \subset \F^{d-1}$.
The polytope $\Fpolyt$ is called an \emph{integral} polytope over $\K$.
Note that $\Kpolyt=\Fpolyt\cap \K^{d-1}$.
If $\K$ and $\F$ are real closed field,  this definition agrees with the $\F$-extension of $\Kpolyt$ for semi-algebraic sets in the sense of real algebraic geometry \cite[Definition 5.1.2]{BochnakCosteRoy_RealAlgebraicGeometry}.

\xfchanged{Using affine charts we can make sense of integral polytopes in projective spaces and their $\F$-extensions. 
For example,  if $\simplex_\eb$ is the simplex associated to a basis $\eb$ of a finite-dimensional $\K$-vector space,  then $(\simplex_\eb)_\F=\PP(\{\sum_{i=1}^d x_i e_i \mid x_i \in \F,  x_i \geq 0\})$.}

Let now $\F$ be a non-Archimedean ordered valued field with value group $\Lambda=\log\abs{\F^\times}$ and valuation ring $\mathcal{O}=\{x \in \F \mid |x| \leq 1\}$.

\begin{theorem}[\Cref{thm:Intro:NAPolytopalHilbertMetricSpace}]
\label{thm:NAPolytopalHilbertMetricSpace}
Let $\Fpolyt \subset \F^{d-1} \subset \PP(\F^d)$ be an integral polytope over a subfield $\KK \subset \mathcal{O}$,  and let $\Omega \subset \F^{d-1} \subset \PP(\F^d)$ be its interior.
Then there exists a map $\Psi \from \Omega \to \Model(\Fpolyt,\Lambda)$ that induces a global isometry
\[ \overline{\Psi} \from (X_\Omega,d_\Omega)\to (\Model(\Fpolyt,\Lambda),\dmodel).\]
Moreover,  $\Psi$ maps the barycenter of $\Fpolyt$ to the cone point $o \in K$, 
and,  if $F$ is a maximal flag $F \in \Flags(\Fpolyt)$,  the restriction
of $\Psi$ to the associated barycentric simplex $\Omega_F\coloneqq \Omega \cap (S_F)_\F$ 
is of the form
\xfchanged{\[ \function{\Psi_F\coloneqq \Psi_{|\Omega_F}}{\Omega_F}{\Model_F\subset \Model(\Fpolyt,\Lambda) \;,}%
{\PP(x)}{\projM_F(\log_{\eb}(x))=\projM_F([\log|x_1|,\ldots,\log|x_d|])}\]
}for some basis $\eb$ of $V$ with
$\Omega_F=\PP\big(\big\{ x=\sum_{i=1}^d x_i e_i \bigm| x_1\geq \cdots \geq x_d >0,  x_i \in \F
\big\}\big)$.
\end{theorem}

The proof of \Cref{thm:Intro:NAPolytopalHilbertMetricSpace} follows in the first step the strategy of \cite[Theorem 1]{Vernicos_HilbertGeometryConvexPolytopes}.
However,  the ultrametric property together with the assumption $\K \subset \mathcal{O}$ allows to prove that the maps $\Psi_F$ descend to local isometries $\overline{\Psi}_F \from \pi(\Omega_F) \to K_F$ (instead of just bi-Lipschitz maps in the Archimedean case).
In the second step,  we show that the maps are well-defined on the intersection of two maximal flags.
Thirdly and lastly,  we glue the local isometries together to a global isometry. 

\begin{proof}
Let $\Kpolyt \coloneqq \Fpolyt \cap \K^{d-1}$ so that $\Kpolyt_\F = \Fpolyt$.

\textbf{Step 1: Construction of local isometries for maximal flags.}
Let $F$ be a maximal flag of $\Kpolyt$ (or equivalently of $\Fpolyt$) and $\simplex_F$ the
corresponding maximal simplex of the barycentric subdivision of $\Kpolyt$.
By \Cref{lem:SandwichLem} there exist two simplices $\Sin \subset \Kpolyt \subset \Sout$ in $\K^{d-1}$ such that $\simplex_F$ is a flag simplex of both $\Sin$ and $\Sout$,  and $\simplex_F$ is a barycentric simplex of $\Sin$ with respect to $b_\Kpolyt$.
The same holds true for their $\F$-extensions,  i.e.\ $(\Sin)_\F \subset \Fpolyt \subset (\Sout)_\F$ in $\F^{d-1}$ with $(S_F)_\F$ a flag simplex of both $(\Sin)_\F$ and $(\Sout)_\F$,  and $(S_F)_\F$ a barycentric simplex of $(\Sin)_\F$ with respect to $b_{\Fpolyt}=b_\Kpolyt$.
Recall that $(S_F)_\F$ consists of the points $\PP(x)$ 
of the form $\PP(\sum_{i=1}^d x_i e_i)$ with $x_1\geq \ldots \geq x_d \geq 0$,  $x_i \in \F$.
By \Cref{propo:Simplex} the map 
\[\Psi_F \coloneqq \alpha_F \circ \log_{\eb} \from \Omega_F \to \Model(\Fpolyt,\Lambda)\eqqcolon K\] is surjective onto $K_F$.

We claim that $\Psi_F$ preserves distances and induces hence an isometry from $\pi(\Omega_F)$ onto its image $K_F \subset K$.
Since the vertices of $(\Sin)_\F$ and $(\Sout)_\F$ are in $\K^{d-1} \subset \mathcal{O}^{d-1} \subset \F^{d-1} \subset \PP V$,  the matrix of $\eb$ in $\eb'$ is in $\GL(d,\K) \subset \GL(d,\mathcal{O})$.
Thus we can apply \Cref{lem:Intro:FlagSimplexSandwichLemma} to
$(S_F)_\F \subset (\Sin)_\F \subset \Fpolyt \subset (\Sout)_\F$ to conclude that $\Psi_F$ preserves distances.
Hence $\Psi_F$ induces an isometry $\overline{\Psi}_F$ from $\pi(\Omega_F) \subset X_\Omega$ to $K_F\subset K$.

\textbf{Step 2: Well-defined on the intersection of maximal flags.}
For two maximal flags $F_1$ and $F_2$ of $\Fpolyt$, we claim that $\Psi_1=\Psi_2$ on $\Omega_{1} \cap \Omega_{2}$,  where $\Psi_i\coloneqq \Psi_{F_i}$ and $\Omega_{i} \coloneqq \Omega_{F_i}$ for $i=1,2$.

\begin{figure}[h]
\begin{tikzpicture}[scale=1.2]
\fill[fill=gray!20] (1,2) --(1.5,3.5)--  (3,3) -- (2+2/3,2)  -- cycle;
\draw (-0.2,3.8)--(0,4);
\draw(0,4)--(3,3);
\draw (2+1/3,1)--(3,3);
\draw (2,0.8)--(2+1/3,1);
\draw (0,4) node[above]{$\Fpolyt$} ;
\draw (1,2) node{$\bullet$} ;
\draw (1,2) node[below]{$b_{\Fpolyt}$} ;
\draw[very thick]  (1,2)--(3,3);
\draw (1,2)--(1.5,3.5);
\draw (1,2)--(2+2/3,2);
\draw (1.9,2.9) node{$\Omega_{1}$} ;
\draw (2.3,2.3) node{$\Omega_{2}$} ;
\draw (0.4,2.5) node{$\Omega_1 \cap \Omega_{2}$} ;
\draw[->] (1,2.4) .. controls +(0.1,0.1) and +(-0.1,0.1) .. (1.4,2.3);

\fill [black!40,path fading=north,fading transform={rotate=-60}]
(4,3) -- (6.5,3) -- (5.25,5.16506) --cycle;
\draw[->] (2.5,3.5) .. controls +(0.3,0.3) and +(-0.3,0.3) .. (4,3.5);
\draw (3.5,4) node{$\Psi_1$} ;
\draw (4,3) node{$\bullet$} ;
\draw[very thick] (4,3)--(6,3);
\draw[dotted, very thick] (6,3)--(6.5,3);
\draw (4,3)--(5,4.73205); 
\draw[dotted] (5,4.73205)--(5.25,5.16506);
\draw (5,3.5) node{$\Aap_\Lambda$};
\draw (6.5,4) node{$K_{F_1}$};

\fill [black!40,path fading=south,fading transform={rotate=60}]
(4,2.5+0.5) -- (6.5,2.5+.5) -- (5.25,0.33494+.5) --cycle;
\draw (4,2.5+.5) node{$\bullet$} ;
\draw[very thick] (4,2.5+.5)--(6,2.5+.5);
\draw[dotted, very thick] (6,2.5+.5)--(6.5,2.5+.5);
\draw (4,2.5+.5)--(5,3.5-2.73205+.5); 
\draw[dotted] (5,3.5-2.73205+.5)--(5.25,0.33494+.5);
\draw[->] (3,2+.5) .. controls +(0.2,-0.2) and +(-0.2,-0.2) .. (4,2+.5);
\draw (3.5,1.6+.5) node{$\Psi_{2}$} ;
\draw (5,2+.5) node{$\Aap_\Lambda$};
\draw (6.5,1+.5) node{$K_{F_2}$};

\end{tikzpicture}
\caption{Intersection of two maximal flags.}
\end{figure}
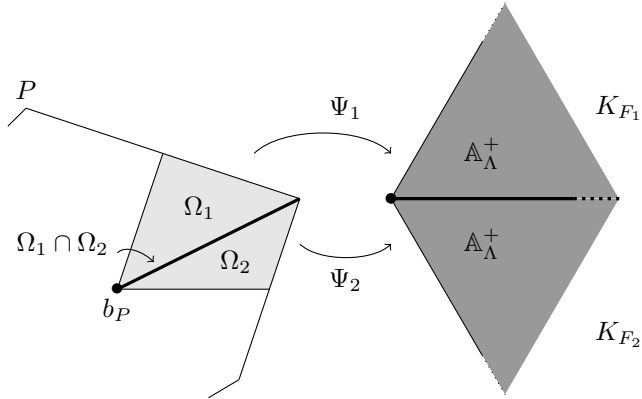

\noindent To see this,  we use \Cref{lem:BaseChangeGL(O)}.
For both $i=1,2$ denote by $\simplex_{\eb_i}$ the simplex constructed in Step 1 to define $\Psi_{i}$.
Let $F \coloneqq F_1 \cap F_2 \in \Flags_*(\Fpolyt)$ be the intersection of the two flags and $D=D(F)$.
Then $\simplex_{\eb_1}^D=\simplex_{\eb_2}^D$.
The same holds true for their $\F$-extensions.
Since $(S_F)_\F$ is a flag simplex of both $(\simplex_{\eb_1})_\F$ and $(\simplex_{\eb_2})_\F$,  the matrix $M \coloneqq \Mat_{\eb_1}(\eb_2)$ of $\eb_2$ in $\eb_1$ is in $P_D(d,\F)$.
Using as above that the vertices of $(\simplex_{\eb_1})_\F$ and $(\simplex_{\eb_2})_\F$ are in $\K^{d-1} \subset \mathcal{O}^{d-1} \subset \F^{d-1} \subset \PP V$,  we have that $M \in \GL(d, \mathcal{O})$,  hence in $P_D(d,\mathcal{O})$.
By \Cref{lem:BaseChangeGL(O)} we conclude $\log_{\eb_1}=\log_{\eb_2}$ on $\Omega_1 \cap \Omega_2$,  which finishes Step 2.
In particular $\Psi(\Omega_1 \cap \Omega_2) = \Model_{F_1}\cap \Model_{F_2}$.

\textbf{Step 3: Gluing the local isometries to a global isometry.}
We claim that
\[ \Psi=(\Psi_F)_{F \in \Flags_*(\Fpolyt)} \from \Omega=\bigcup_{F \in \Flags_*(\Fpolyt)} \Omega_F \to  \bigcup_{F \in \Flags_*(\Fpolyt)} K_F=K\]
preserves distances globally,  thus induces an isometry $\overline{\Psi} \from (X_\Omega,d_\Omega) \to (K,\dhex)$.
Since $\Psi=\Psi_F$ on $\Omega_F$ and $\Psi_F$ preserves distances, we already know that if $x,y \in \Omega_F$,  then $d_\Omega(x,y)=\dmodel(\Psi(x),\Psi(y))$.
Let thus $x \in \Omega_F$ and $y \in \Omega_{F'}$.
Consider the projective line segment from $x$ to $y$.
On it one finds finitely many points $x=x_0,  x_1, \ldots, x_k, x_{k+1} = y$,  such that any two consecutive points $x_i, x_{i+1}$ lie in \xfchanged{$\Omega_{F_i}$ for some maximal flag $F_i \in \Flags_*(\Fpolyt)$}.
Thus
\[d_\Omega(x,y)=\sum_{i=0}^{k}d_\Omega(x_{i},x_{i+1}) 
=\sum_{i=0}^{k}\dmodel \xfchanged{(\Psi_{F_i}(x_{i}),\Psi_{F_i}(x_{i+1}) ) }
\geq \dmodel(\Psi(x),\Psi(y)).\]
On the other hand, 
for every chain  $z_0=\Psi(x), z_1,\ldots,z_l,z_{l+1}=\Psi(y)$ connecting $\Psi(x)$ and $\Psi(y)$ in $K$,  denoting $F_{i}$ the flag such that $z_i,z_{i+1}\in K_{F_i}$,
there exists $x_0=x$,  $x_i\in \Omega_{F_i}\cap \Omega_{F_{i-1}}$,
$i=1,\ldots l$,  $x_{l+1} = y$
from $x$ to $y$ such that $\Psi(x_i)=z_i$ for $i=0,\ldots l+1$. 
Then
\begin{align*}
d_\Omega(x,y) \leq \sum_{i=0}^{l} d_\Omega(x_i,x_{i+1})= \sum_{i=0}^{l}  \dmodel(z_i,z_{i+1}),
\end{align*}
since $x_i,x_{i+1} \in \Omega_{F_i}$ and $\Psi$ is an isometry on $\Omega_{F_i}$.
This proves $\dmodel(\Psi(x),\Psi(y))=d_\Omega(x,y)$.
\end{proof}

\section{Ultralimits of Hilbert geometries as non-Archimedean Hilbert geometries}
\label{section:UltralimitsHilbertGeometry}
The goal of this section is to prove \Cref{thm:Intro:UltralimitsHilbertGeometries},  \Cref{thm:Intro:AsymptoticConesPolytopalHilbertGeometries},  and \Cref{thm:Intro:DegenerationsConvProjStr}.
Before doing so we introduce the necessary background on ultrafilters in Sections \ref{subsection:UltrafiltersProducts}-\ref{subsection:UlimInLinearGroup}.

\subsection{Ultrafilters and ultraproducts}
\label{subsection:UltrafiltersProducts}
We refer to \cite[I]{Bourbaki_GenTopo} for the background on
ultrafilters in this section.

For all the section, we fix a non-principal ultrafilter $\omega$ on $\NN$.
\compl{In fact $\omega$ will vary in the proof of \Cref{prop:UltralimitsAndGH}.}
Recall that $\omega$ may be seen as a finitely additive probability measure on
$\NN$ with values in $\{0,1\}$, null on finite sets, 
in particular, any assertion depending
on an integer $n\in\NN$ is either true for $\omega$-almost all $n$ or false for $\omega$-almost all $n$.
%
A sequence $(x_n)_{n\in\NN}$ in a  fixed topological space $X$ 
is said to converge to a point $x\in X$ \emph{according to $\omega$} 
if for every neighborhood $V$ of $x$, one has $x_n\in V$ for
$\omega$-almost all $n$. We will then say that $x$ is an {\em $\omega$-limit} of 
$(x_n)$ and denote $\lim_\omega x_n=x$. 
If $X$ is Hausdorff the $\omega$-limit
is unique if it exists. 
Note that if $\lim_\omega x_n=x$ then $x$ is a limit point of the sequence
$(x_n)$, namely there is a subsequence converging to $x$ is the usual
sense.
An important feature is that every sequence in a compact space has an
$\omega$-limit.
The $\omega$-limits of sequences of real numbers are taken
in the compact space $[\minfty,\pinfty]$.

Given a sequence of sets $(E_n)$, the \emph{ultraproduct} of the
$E_n$ following the ultrafilter $\omega$ is the quotient 
\[\prod_\omega E_n \coloneqq \prod_{n\in\NN} E_n /_{=_\omega}\]
where $(x_n)=_\omega(y_n)$ 
if $x_n=y_n$ for $\omega$-almost all $n$.

\subsection{Ultralimits of metric spaces}
\label{subsection:UltralimitsMetricSpaces}
A reference for this section is \cite{KlLe97}.
Given a sequence of metric spaces
$(E_n,d_n)$, 
we have an associated  pseudo-distance 
\[\dom: \prod_n E_n \times \prod_n E_n \to [0,\infty]\] 
given by
\[\dom((x_n),(y_n)) \coloneqq \lim_\omega d_n(x_n,y_n) \;.\]

Given a sequence of \emph{observation points} $o=(o_n) \in \prod_n E_n$, 
the {\em ultralimit} 
\[(E_\omega,d_\omega,o_\omega)=\Ulim_\omega(E_n,d_n,o_n)\]
of the sequence of pointed metric spaces $(E_n,d_n,o_n)$
is defined as the  quotient metric space of $(\prod_n E_n, \dom)$
based at $o$, namely 
\[E_\omega \coloneqq \{x \in \prod_n E_n \colon \dom(o,x) < \infty\}_{/\dom(x,y)=0}\]
with base-point the projection $o_\omega$ of  $o$.
In the case where 
$E_n$ is a fixed metric space $(E,d)$ and $d_n=\tfrac{1}{\lambda_n}d$, 
for $(\lambda_n)$ a sequence  of real numbers  such that $\lambda_n\geq 1$ and 
$\lambda_n \to \infty$ (called a {\em scaling sequence}), the ultralimit
$\Ulim_\omega(E_n,d_n,o_n)$ is called the \emph{asymptotic cone} of
$(E,d)$ {\em with respect to  $(\lambda_n)$ and  $(o_n)$}.

When $E_n$ is endowed with an isometric action $\rho_n \from \Gamma\to\Isom(E_n)$
of a finitely generated group $\Gamma$, such that 
\[\lim_\omega \dn(o_n,\rho_n(\gamma)o_n) < \infty \]
for every $\gamma$ in a finite symmetric generating set $F$ of $\Gamma$,
then there is an induced isometric action $\rho_\omega$ of $\Gamma$
on $E_\omega$, defined by
\[\rho_\omega(\gamma) (\ulim x_n) =\ulim \rho_n(\gamma) (x_n) \;.\]

Sequences $x=(x_n)$ in $\prod_n E_n$ such that $\dom(o,x) =\lim_\omega d_n(o_n,x_n) < \infty$ are called 
{\em $\omega$-bounded}.
The class in $E_\omega$ of a $\omega$-bounded sequence
$(x_n)$ will be called its {\em ultralimit} in $E_\omega$ and 
denoted by $\ulim_{E_\omega} x_n$ 
or simply $\ulim x_n$ if there is no ambiguity.
Note that changing the sequence on a $\omega$-null set of indices 
does not change the ultralimit, in particular the ultralimit is
well-defined even if the sequence is defined only on a $\omega$-full set.

We now recall the definition and some basic properties 
of ultralimits of subsets. 
Given a sequence of subsets $Y_n \subset E_n$, we denote by
$\ulim Y_n$ 
the subset of $E_\omega$  consisting of ultralimits $\ulim y_n$ of  
$\omega$-bounded sequences $(y_n) \in \prod_nE_n$ 
such that $y_n \in Y_n$ for $\omega$-almost all $n$, 
and call it the {\em ultralimit} of the sequence $(Y_n)$.
Note that it is not empty if and only if there exists an $\omega$-bounded
sequence such that $o_n'\in Y_n$ for $\omega$-almost all $n$.
In this case
$\ulim Y_n$ is isometric to the ultralimit of $(Y_n, \dn, o_n')$. 
It is easily seen that $Y_n$ and its closure have the same ultralimit
\[\ulim{\overline{Y_n}}=\ulim Y_n\;.\]
Indeed,  if $x_\omega=\ulim x_{n}$ with  $x_n\in \overline{Y_n}$, then 
taking
$y_n\in Y_n$ such that $d_n(x_n,y_n) \leq \frac{1}{n}$
we get $x_\omega=\ulim y_n \in \ulim Y_n$.
%
More generally the distance to a subset 
behaves well with respect to ultralimits. 

\begin{proposition} 
\label{prop:ulimDistToSubset}
Let  $Y_n \subset E_n$ be a sequence of subsets.
Let $x_\omega =\ulim x_n \in E_\omega$. Then
\begin{enumerate}
\item 
  \label{it:ulimDistToSubset}
$\dom(x_\omega,\ulim Y_n)  =   \lim_\omega d_n( x_n , Y_n )$, 
where the distance to the empty subset is $\infty$.
\item Suppose that $\ulim Y_n$ is non-empty.
There exists $y_\omega\in \ulim Y_n$ such that
  \label{it:ulimDistToSubsetRealized}
$\dom(x_\omega,y_\omega)= \dom(x_\omega,\ulim Y_n)$.
\end{enumerate}
\end{proposition}
\begin{proof}
Let $y_\omega \in \ulim Y_n$ and $y_n \in Y_n$ an $\omega$-bounded sequence  with $y_\omega =\ulim y_n$.
Then we have that 
\[\dom(x_\omega,y_\omega)=\lim_\omega d_n(x_n,y_n) \geq \lim_\omega d_n(x_n,Y_n)\]
which shows that $\dom(x_\omega,\ulim Y_n) \geq \lim_\omega d_n(x_n,Y_n)$.
Moreover, if $\lim_\omega d_n( x_n , Y_n ) < \infty$
then  taking 
$y_n\in Y_n$ such that $d_n(x_n,y_n) \leq d_n(x_n,Y_n)+\frac{1}{n}$
we get $y_\omega \in \ulim Y_n$ such that 
$\dom(x_\omega,y_\omega)= \lim_\omega d_n(x_n,Y_n)$.
\end{proof}
Note that this also shows that $\ulim Y_n$ is always a closed subset of $E_\omega$,  since $\dom(x_\omega,\ulim Y_n)=0$ implies that $x_\omega\in \ulim Y_n$,  as the distance is always realized.

The following proposition relates ultralimits and
pointed Gromov--Hausdorff limits and is well-known.
 We refer to \cite[Section 8.1]{BBI}  for the notion of
 pointed Gromov--Hausdorff convergence.

\begin{proposition}[Link with Gromov--Hausdorff limits]
\label{prop:UltralimitsAndGH}
Let $(E_n,d_n,o_n)$, ${n\in\NN}$, and $(E,d,o)$ 
be proper pointed metric spaces.
The following are equivalent
\begin{enumerate}
\item 
\label{it:cvGH}
$(E_n,d_n,o_n)$ is converging to $(E,d,o)$ 
in the pointed Gromov--Hausdorff sense;
\item
\label{it:conAs}
$(E_\omega,d_\omega,o_\omega)$ is isometric to $(E,d,o)$  for all non-principal
ultrafilters $\omega$.
\footnote{By an isometry of pointed metric spaces we mean that
  the base-point is sent to the base-point.}
\end{enumerate}
\end{proposition}

\begin{proof}
The forward direction (\ref{it:cvGH}) $\Rightarrow$ (\ref{it:conAs}) 
follows from \cite[Lemma 2.4.3]{KlLe97}.

The converse statement is harder to find in the literature 
so we include a proof.
Let $\eps >0$ and $r\in\RR_{\geq 0}$ and let  $\ov{B}_r(o)$
be the closed ball
of radius $r$ centered at $o$ in $E$ and $\ov{B}_r(o_n)$ be the closed ball of radius $r$ in $E_n$.

We want to show that for $n$ big enough
we have  the following property, which we denote $P(n)$ for future reference: 
there exists a map $f_n \from \ov{B}_r(o_n) \to E$ sending $o_n$ to $o$ such that 
\[\operatorname{dis}(f_n)\coloneqq \sup_{x_1 ,x_2\in \ov{B}_r(o_n)}|d(f_n (x_1), f_n (x_2)) − d_n(x_1 , x_2 )| < \epsilon\]
and the $\epsilon$-neighborhood of the set $f_n(\ov{B}_r(o_n))$ contains the ball $\ov{B}_{r−\epsilon}(o)$ (see \cite[Section 8.1]{BBI}).

Note that 
it is enough to show that $P(n)$ is true for $\omega$-almost all $n$ for every non-principal ultrafilter $\omega$. 
Indeed if the subset $I$ of $n\in\NN$ such that $P(n)$ is false is infinite, then there exists a non principal ultrafilter $\omega$ on $\NN$ containing $I$. 

Let $\omega$ be a non-principal ultrafilter on $\NN$.
By assumption $(E,d,o)$  is isometric to $(E_\omega,d_\omega,o_\omega)$,  and we identify it to the latter in the following.

We construct the maps $f_n$ using $\eps$-nets.
Let $B_\omega$ be the ultralimit of $\ov{B}_r(o_n)$ in $E_\omega$.
It is a closed subset of $E_\omega$, included in $\ov{B}_r(o_\omega)$, hence compact
since $E_\omega$ is proper.
Thus $B_\omega$ contains a finite $\eps/4$-net 
$$Y_\omega=\{y_\omega^{(i)},\ i=0, \ldots, N\}$$
with $y_\omega^{(0)}=o_\omega$.
For $i=0$ to $N$, choose a sequence $(y_n^{(i)}) \in \prod_n\ov{B}_r(o_n)$ 
representing $y_\omega^{(i)}$. 
We may suppose that $y_n^{(0)}=o_n$ for all $n \in \N$.  

We first show that for $\omega$-almost all $n$ the set 
\[Y_n = \{y_n^{(i)},\ i=0, \ldots, N\}\subset \ov{B}_r(o_n)\]
is an $\eps/3$-net in $\ov{B}_r(o_n)$: 
if not, for $\omega$-almost all $n$ we can choose  $z_n \in \ov{B}_r(o_n)$
such that $d_n(z_n, Y_n) > \eps/3$.
Then the ultralimit $z_\omega =\ulim z_n$  is in $B_\omega$ and
$d_\omega(z_\omega, Y_\omega)=\lim_\omega d_n(z_n, Y_n) \geq \eps/3$, 
yielding a contradiction as $Y_\omega$ is an $\eps/4$-net of $B_\omega$.

Now there is an $\eps/3$-correspondence between $Y_n$ 
and $Y_\omega$, more precisely
there exists $I \subset \N$ with $\omega(I)=1$ such that for all $n\in I$
\begin{equation}
  \label{eq:disY}
| d_\omega(y_\omega^{(i)}, y_\omega^{(j)}) - d_n(y_n^{(i)}, y_n^{(j)}) | 
\leq \eps/3
\ \ \text{for all } (i,j) \in \{0,\ldots, N\}^2,\end{equation}
since  by definition of $d_\omega$ we have
\[d_\omega(y_\omega^{(i)}, y_\omega^{(j)}) = \lim_\omega d_n(y_n^{(i)}, y_n^{(j)})\]
for all $(i,j) \in \{0,\ldots, N\}^2$. 

We finally construct the maps $f_n \from \ov{B}_r(o_n) \to E_\omega$ for $n\in I$
and show that they satisfy $P(n)$.  
The construction consists in
projecting on the closest point in $Y_n$ and sending it to the corresponding point
in $Y_\omega$.
Namely, for $x\in \ov{B}_r(o_n)$, choose $i \in \{0,\ldots,N\}$ with $d_n(x,y_n^{(i)})$
minimal,  and let $f_n(x) \coloneqq y_\omega^{(i)}$.
In particular, $f_n(o_n)=o_\omega$.
Since $Y_n$ is an $\eps/3$-net in $\ov{B}_r(o_n)$ we have by (\ref{eq:disY})
that 
\[|d_\omega(f_n (x_1), f_n (x_2)) − d_n(x_1 , x_2 )| \leq \epsilon\]
for all $x_1 ,x_2 \in \ov{B}_r(o_n)$.
Moreover, $f_n(\ov{B}_r(o_n))$ is an $\eps/3$-net in $Y_\omega$,  which is a $\eps/4$-net in $B_\omega$, thus the $\epsilon$-neighborhood of the set $f_n(B_r(o_n))$ contains $B_\omega$.
This concludes since $B_\omega=\ulim \ov{B}_r(o_n)$ contains the ball
$\ov{B}_{r'}(o_\omega)$ for every $r'<r$, in particular for $r'=r-\eps$.
\compl{Since if $d_\omega(o_\omega,x_\omega)<r$ then $d_n(o_n,x_n)<r$ for $\omega$-almost
all $n$ (in fact it contains $\ov{B}_r(o_\omega)$)}%
\end{proof}

\subsection{The ultralimit field $\RRom$ (Robinson field)}
\label{subsection:RobinsonField}
We refer to \cite[\S 3.3]{Parreau_CompEspReprGroupesTypeFini} 
for more details.

The {\em asymptotic cone} $\RRom$ of $\RR$ 
{\em with respect to the scaling sequence $(\lambda_n)$}  
is a non-Archimedean real closed  field, that
may be defined as the ultralimit
of the sequence of valued fields 
$(\RR, \abs{\cdot}^{1/\lambda_n})$,
seen as metric spaces pointed 
at $o_n=0$, 
endowed with term-by-term operations and order, 
and with the absolute value
$\absom{\cdot}:\RRom\to\RR_{\geq 0}$ defined by
\[\absom{\ulim t_n} \coloneqq \lim_\omega \abs{t_n}^{1/\lambda_n}\;.\]
Note that this is a non-Archimedean real closed field, 
whose absolute value $\absom{\cdot}$  
takes all values in $\RR_{\geq 0}$.
We have a canonical order-preserving embedding $\R \hookrightarrow \RRom$ by sending a real number $t$ to the equivalence class of the constant sequence $(t)_{n\in \N}$,  which we also denote by $t$.
%
Given a sequence $(t_n)$ in $\RR$, 
we set $\ulim{t_n}=\infty$ if
$\lim_\omega\abs{t_n}^{1/\lambda_n}=\infty$, 
so that every sequence in $\RR$ has a well defined 
ultralimit in $\RRom \cup \{\infty\}$.

\subsection{Ultralimits in $\RR^d$ }
\label{subsection:UlimRd}
We now turn to ultralimits of sequences in  $V=\RR^d$.
We will denote by $\RRom^d$ the $\RRom$-vector space $(\RRom)^d$.
A sequence $x=(x_n)$  in $\RR^d$ is \emph{$\omega$-bounded} if all its
coordinates are $\omega$-bounded in $\RR$.
We then define its ultralimit $x_\omega=\ulim x_n \in (\RRom)^d$  
by taking ultralimits of coefficients, namely
\[x_\omega=(x_{\omega,1},\ldots,x_{\omega,d}) \in (\RRom)^d\] 
where $x_n=(x_{n,1},\ldots,x_{n,d}) \in \RR^d$ and  $x_{\omega,i} = \ulim x_{n,i}   \in \RRom$.

\begin{remark}
\label{rem:UlimVSExtension}
If $X\subseteq \R^d$ is semi-algebraic,  we have $\ulim X \subseteq X_{\RRom} \subseteq \RRom^d$,  where 
$X_{\RRom}$ is the \emph{$\RRom$-extension} of $X$, i.e.\ the set of solutions in $\RRom^d$ of the same polynomial equalities and inequalities defining $X$ \cite[Definition 5.1.2]{BochnakCosteRoy_RealAlgebraicGeometry}.

In the case that $X$ is an affine polytope or a closed ball for the $\RRom$-valued $l^2$-norm $\|\cdot\|_{2,\RRom}$ (see \Cref{subsection:SymSpaces})
one has in fact equality,  i.e.\  $\ulim X=X_{\RRom}$.
\compl{
If for all $x \in X$ they satisfy $P(x)>0$ or $Q(x)=0$ for some polynomials,  then the same holds true for the ultralimits,  which shows $\ulim X \subseteq X_{\RRom}$.
In the case of the affine polytope,  one has that $\ulim X$ is $\RRom$-convex and thus contains the $\RRom$-convex hull of $X$,  which agrees with $X_{\RRom}$ by the considerations in \Cref{subs:ConvexSets}.
For the closed ball $B$ of radius $1$ (the other radii are analogous),  take an $\omega$-bounded sequence $(v_n) \in (\R^d)^\N$ with $v_\omega=\ulim v_n \in B_{\RRom}$.
If $\sum_{i=1}^d (v_\omega)_i^2< r^2$ (in $\RRom$),  we get that for $\omega$-almost all $n \in \N$,  we have $\sum_{i=1}^d (v_n)_i^2 \leq r^2$.
Thus $v_n \in B$ for $\omega$-almost all $n \in \N$,  and hence $v_\omega \in \ulim B$.
If $\sum_{i=1}^d (v_\omega)_i^2= r^2$,  we do not necessarily get that for $\omega$-almost all $n \in \N$,  we have $\sum_{i=1}^d (v_n)_i^2 \leq r^2$.
Consider thus $v'_n \coloneqq \tfrac{v_n}{\|v_n\|} \in B$ for all $n \in \N$.
We claim that $\ulim v'_n =\ulim v_n=v_\omega$,  and thus $v_\omega \in B$.
Indeed
\[
\lim_\omega \|v_n-v'_v\|^{1/\lambda_n} = \lim_\omega |\|v_n\|-1|^{1/\lambda_n}=0,
\]
where the last equality follows from the assumption $\sum_{i=1}^d (v_\omega)_i^2= r^2$ (this is an equality in $\RRom$).
}
\end{remark}

Denoting by $\norm{\cdot}$ the standard Euclidean norm on $\RR^d$,
observe that
\[\norm{x}\coloneqq\lim_\omega\norm{x_n}^{1/\lambda_n}=\max_i \abs{x_{\omega,i}} \;.\] 
In particular $x=(x_n)$ is $\omega$-bounded if and only if 
$\norm{x}=\lim_\omega\norm{x_n}^{1/\lambda_n}<\infty$,  and
$\ulim x_n=\ulim y_n$ if and only if  
$\norm{y-x}=0$.
Hence  $(\RRom)^d$ endowed with the $l^\infty$-norm 
is isomorphic as a normed vector space to 
the asymptotic cone of the normed vector space $(V, \norm{\cdot})$
with respect to the scaling sequence $(\lambda_n)$
as defined in \cite[Proposition
3.13]{Parreau_CompEspReprGroupesTypeFini},  see Proposition 3.14 therein.

\subsection{Ultralimits in the linear group}
\label{subsection:UlimInLinearGroup}
A sequence $(g_n)$ in $\GL(V)\simeq \GL(d,\RR)$ is 
{\em $\omega$-bounded in $\GL(V)$} if
\[\lim_\omega\norm{g_n}^{1/\lambda_n} <\infty 
\text{ and } \lim_\omega \norm{g_n^{-1}}^{1/\lambda_n} <\infty \;.\]
where $\norm{\cdot}$ is any norm on $\End(V)$.
Then 
we may define the  ultralimit $g_\omega \in M(d,\RRom)$ 
of $(g_n)$ by taking ultralimits on  coefficients
\[(g_\omega)_{ij}=\ulim (g_n)_{ij}\]
and we have $g_\omega \in \GL(d,\RRom)$
\cite[Proposition 5.1]{Parreau_InvariantWeaklyConvexCocompactSubspacesSurfaceGroupsA2Buildings}.
For any $\omega$-bounded sequence $(x_n)$ in $V$ we then have 
$g_\omega (\ulim x_n)= \ulim g_n(x_n)$.

We now relate the eigenvalues of  $g_\omega \in \GL(d,\RRom)$ with the eigenvalues of the elements in the sequence $(g_n)$,  where $g_n$ is an $\omega$-bounded sequence in $\GL(V)$ with $\ulim g_n=g_\omega$.
Denote by $\lambda_1(g_n), \ldots, \lambda_d(g_n) \in \C$ the eigenvalues of $g_n$ counted with multiplicity and sorted by descending modulus $|\lambda_1(g_n)|\geq  \ldots \geq |\lambda_d(g_n)|$.
Then the sequence $(\lambda_1(g_n))$ in $\C$ is $\omega$-bounded, since $|\lambda_1(g_n)|\leq \norm{g_n}$ for any operator norm $\norm{\cdot}$ on $\End(V)$.
The same holds true for the sequences $(\lambda_i(g_n))$ for all $i=2,\ldots,d$.
Thus 
\[\ulim\lambda_1(g_n), \ldots, \ulim\lambda_d(g_n) \in \C_\omega, \]
are defined, where $\C_\omega$ is the Robinson field over $\C$,  which is defined as the ultralimit of the sequence of pointed metric spaces $(\C, \abs{\cdot}^{1/\lambda_n}, 0)$ as in \Cref{subsection:RobinsonField}.
In fact $\C_\omega=\RRom[\sqrt{-1}]$.
We remark furthermore,  that  $\ulim\lambda_1(g_n), \ldots, \ulim\lambda_d(g_n)$ satisfy 
\[|\ulim\lambda_1(g_n)|_{\RRom} \geq \ldots \geq |\ulim\lambda_d(g_n)|_{\RRom} \in \RRom,\]
where $\abs{x_\omega +y_\omega\sqrt{-1} }_{\RRom}\coloneqq \sqrt{x_\omega^2+y_\omega^2} \in (\RRom)_{\geq 0}$ for $x_\omega +y_\omega\sqrt{-1} \in \C_\omega$.
We now want to relate $\ulim\lambda_1(g_n), \ldots, \ulim\lambda_d(g_n)$ to the eigenvalues of $g_\omega$.

As a first observation we have that if $v_n \in V$ is a sequence of eigenvectors of $g_n$ of unit norm for the eigenvalue $\lambda_n$,  then $(v_n)$ is $\omega$-bounded and $\ulim v_n \in \Vom$ is an eigenvector of $g_\omega$ for the eigenvalue $\ulim \lambda_n$ (in particular,  non-zero).
On the other hand we have the following.

\begin{proposition}
\label{propo:UlimEV}
The ultralimits $\ulim\lambda_1(g_n), \ldots, \ulim\lambda_d(g_n) \in \C_\omega$ are the eigenvalues of $g_\omega$ counted with multiplicity.
\end{proposition}
\begin{proof}
It suffices to show that the characteristic polynomial of $g_\omega$ is equal to $\prod_{i=1}^d(T-\ulim \lambda_i(g_n))$,  in other words $\det(T\Id-g_\omega)=\prod_{i=1}^d(T-\ulim \lambda_i(g_n))$.
Since $\C_\omega$ is infinite, it suffices to check the equality by evaluating for any point $T=t_\omega \in \C_\omega$.
Let thus $t_\omega=\ulim t_n$ for an $\omega$-bounded sequence $(t_n)$ in $\C$.
Then
\begin{align*}
\det(t_\omega\Id-g_\omega)&=\ulim \det(t_n\Id-g_n)\\
&=\ulim \prod_{i=1}^d(t_n-\lambda_i(g_n))= \prod_{i=1}^d(t_\omega-\ulim \lambda_i(g_n)),
\end{align*}
where the first equality holds since the determinant is a polynomial expression.
\end{proof}

\subsection{The ultralimit projective space $\PVom$}
\label{subsection:UlimProjectiveSpace}
Ultralimits of projective spaces were introduced in
\cite{Parreau_InvariantWeaklyConvexCocompactSubspacesSurfaceGroupsA2Buildings},
to which we refer for more details.

The projective space $\PP V$ on $V=\RR^d$ is endowed with 
the spherical metric induced by the canonical norm $\norm{\cdot}$, 
which is  defined by
\[\dPV(x,y) \coloneqq \inf_{\tilde{x},\tilde{y}}\norm{\tilde{y}-\tilde{x}}\]
where the infimum is taken over all unit vectors $\tilde{x}$ and $\tilde{y}$ lifting $x$ and $y$.

The {\em asymptotic cone} $\PVom$ of $\PP V$ with respect to the scaling
sequence $(\lambda_n)$ is defined as the ultralimit of the sequence of
metric spaces $(\PP V, \dPV^{1/\lambda_n})$ with respect to any sequence of
base points.
Note that the ultralimit does not depend on the choice of base points as these metric spaces are uniformly bounded.
%
The asymptotic cone $\PVom$ identifies canonically with the projective space $\PP(\Vom)$ over
$\Vom=(\RRom)^d$, in such a way that
\[ \PP(\ulim x_n)= \ulim \PP(x_n) \]
for any $\omega$-bounded sequence $(x_n)$ in $V$ with $\ulim x_n \neq 0$ in $\Vom$. 

\begin{proposition}[Ultralimits and affine charts]
\label{prop:ultralimits:affineCharts}
\label{prop:UlimAffineChartsBounded}
Let $\eom=(e_{\omega,1},\ldots,e_{\omega,d})$ be a basis of $\Vom$ and let
$f_\omega:\RRom^{d-1}\to \PVom$ be the associated  affine chart.
For $i=1,\ldots,d$, let $(e_{n,i})$  be a $\omega$-bounded sequence in $V$ with
ultralimit $e_{\omega,i}$. 
Then for $\omega$-almost all $n\in \N$,   $\en=(e_{n,1},\ldots,e_{n,d})$ is a basis of $V$, 
and we have
\[f_\omega(\ulim x_n) \coloneqq \ulim f_n(x_n) \; ,\]
where $f_n \from \RR^{d-1}\to \PV$ is the affine chart associated to
$\en$ 
(which is defined $\omega$-almost everywhere).
Moreover  given a sequence $(A_n)$ of subsets of $\PP V$,  if $\ulim
A_n$ is contained in the image of $f_\omega$,  then $A_n$ is contained in
the image of $f_n$ for $\omega$-almost all $n$, and $A_n$ is bounded in that chart.
\end{proposition}

\begin{proof}
Let $g_\omega \in \GL(\Vom)$ be the map sending the canonical basis to 
$\eom$, and let 
$g_n \in \End(V)$ be the map sending the canonical basis to 
$\en$. 
We have $g_\omega=\ulim g_n$, 
therefore by \cite[Proposition 5.1]{Parreau_InvariantWeaklyConvexCocompactSubspacesSurfaceGroupsA2Buildings}
we have $g_n \in \GL(V)$  and $\en$ is a basis of $V$ for $\omega$-almost all $n \in \N$. 
Up to precomposing $f_n$ by $g_n^{-1}$ and $f_\omega$ by $g_\omega^{-1}$, 
we may then suppose that 
$f_n$ is the standard affine chart $\RR^{d-1} \to \PP(\RR^d)$.
Then $f_\omega$ is the standard affine chart $\RRom^{d-1} \to \PP(\RRom^d)$
and for any $\omega$-bounded sequence 
$x_n=(x_{n,1} , \ldots , x_{n,d-1}) \in (\RR^{d-1})^\N$ we have
\begin{align*}
  \ulim f_n(x_n)
&= \ulim [x_{n,1} : \ldots : x_{n,d-1} : 1]\\
&= [ \ulim x_{n,1} : \ldots :\ulim x_{n,d-1} : 1]\\
&= [x_{\omega,1} : \ldots : x_{\omega,d-1} : 1]\\
&=f_\omega(x_\omega).
\end{align*}

For the last assertion, since $\omega$ is an ultrafilter  
we have either  $A_n$ is contained in the image of $f_n$ 
for $\omega$-almost all $n$, or 
$A_n$ is not contained in the image of $f_n$ for $\omega$-almost all
$n$,  i.e.\ there exists $x_n \in A_n \setminus f_n(\R^{d-1}) = A_n \cap \PP(\RR^{d-1}\times \{0\})$.
Then $\ulim x_n \in \ulim A_n$ belongs to $\ulim \PP(\RR^{d-1}\times \{0\}) = \PP(\RRom^{d-1}\times \{0\})$, contradicting
the hypothesis that $\ulim A_n$ is contained in the image of $f_\omega$.
Observe that if a sequence 
$x_n=(x_{n,1} , \ldots , x_{n,d-1}) \in (\RR^{d-1})^\N$ is not $\omega$-bounded,
$\ulim f_n(x_n)$ belongs to the hyperplane at infinity $\PP(\RRom^{d-1}\times\{0\})$ of the affine chart $f_\omega$,  \xfchanged{ since we need to normalize the last projective coordinate (which is constant equal to one) by a sequence of numbers tending to infinity.}
This proves that  there exists a
constant $C\in \RR_{>0}$ such that $A_n\subset[-C^{\lambda_n},C^{\lambda_n}]^{d-1}$ for
$\omega$-almost all $n$.
\end{proof}

We now relate proximality and ultralimits.
Let $\F$ be a real closed field and $V$ a finite-dimensional $\F$-vector space.
An element $g \in \GL(V)\cong \GL(d,\F)$ is \emph{proximal} if $g$ has a unique generalized eigenvalue $\lambda_1(g)$ (in $\F[\sqrt{-1}]$) with largest modulus (in $\F$),  where eigenvalues are counted with multiplicity \cite[\S 4.4]{BurgerIozziParreauPozzetti_RSCCharacterVarieties2}.
This generalizes the definition in the real case, see e.g.\ \cite[\S 4.1]{BenoistQuint_RandomWalksReductiveGroups}.
Furthermore,  if $g$ is proximal then $\lambda_1(g) \in \F$, and $g$ has a unique attracting fixed point in $\PP V$ and in $\PP V^*$ (corresponding to the repelling hyperplane of $g$) \cite[Corollary 4.11 (1)]{BurgerIozziParreauPozzetti_RSCCharacterVarieties2}.
We say that $g$ is \emph{biproximal} if both $g$ and $g^{-1}$ are proximal.
In this case $g$ has a unique attracting and repelling fixed point in $\PP V$ and in $\PP V^*$.

The goal of the next proposition is to relate the ultralimits of the attracting and repelling fixed points to the attracting and repelling fixed points of the ultralimit.

\begin{proposition}[Proximality and ultralimits]
\label{propo:UlimProximalElements}
Let $(g_n)$ be an $\omega$-bounded sequence of elements in $\GL(V)\cong\GL(d,\R)$ such that $g_\omega =\ulim g_n \in \GL(d,\RRom)$ is proximal with attracting fixed point $x_\omega  \in \PP \Vom$.
Then for $\omega$-almost all $n \in \N$,  $g_n$ is proximal,  and,  denoting by $x_n \in \PP V$ the attracting fixed point of $g_n$ in $\PP V$,  we have $x_\omega = \ulim x_n$.
\end{proposition}
\begin{proof}
Denote by $\lambda_1(g_n), \ldots, \lambda_d(g_n) \in \C$ the generalized eigenvalues of $g_n$ sorted by descending modulus $|\lambda_1(g_n)|\geq  \ldots \geq |\lambda_d(g_n)|$.
Then $\ulim\lambda_1(g_n), \ldots, \ulim\lambda_d(g_n) \in \C_\omega$ are the generalized eigenvalues of $g_\omega$,  see \Cref{propo:UlimEV}.
Since $g_\omega$ is proximal,  we obtain $\ulim \lambda_1(g_n)>|\ulim\lambda_2(g_n)|_{\RRom}$.
Thus $\lambda_1(g_n)>|\ulim\lambda_2(g_n)|$ for $\omega$-almost all $n\in \N$.
This proves the first claim.

Let $v_n \in V$ be a lift of $x_n$ of norm one for $\omega$-almost all $n\in \N$ (for the others we can choose arbitrary vectors in $V$).
Then $\ulim v_n \in \Vom$ is an eigenvector of $\ulim \lambda_1(g_n)$,  and thus $x_n=\PP(\ulim v_n)=x_\omega$.
\end{proof}

\subsection{Ultralimits of convex projective subsets} 
\label{subsection:UlimConvexProjSets}

Let $(\Omega_n)_{n\in\NN}$ be a sequence of convex subsets in $\PP V$. 
We consider the  ultralimit projective space $\PVom$ defined in \Cref{subsection:UlimProjectiveSpace}.

First observe that the ultralimit $\ulim \Omega_n$ of the $\Omega_n$ is a convex subset of $\PVom$. 
This follows from the fact that the ultralimit of a sequence of segments $[x_n,y_n] \subset V$ is the
segment $[x_\omega,y_\omega] \subset V_\omega$, where $x_\omega=\ulim x_n$ and $y_\omega=\ulim y_n$.
We now prove that the boundary behaves well with respect to ultralimits.

\begin{proposition}[Boundary of ultralimits]
\label{prop:ultralimitsOfConvex:boundary}
Let $(\Omega_n)_{n\in\NN}$ be a sequence of convex subsets in
$\PP V$.
Suppose that $\ulim \Omega_n$ is contained in an affine chart.
Then 
\[\partial \ulim \Omega_n = \ulim \partial\Omega_n \;.\]
\end{proposition}

\begin{remarks}
    \noindent\begin{enumerate}
\item  The inclusion $\partial\ulim Y_n \subset \ulim \partial Y_n$ is true for any
    sequence of subsets $(Y_n)$  in $\PP V$ (see the proof below).
\item If $\ulim \Omega_n$ is not contained in an affine chart,  then 
the inclusion $\partial\ulim \Omega_n \subset \ulim \partial\Omega_n$ 
may be strict:
for example with $\Omega_n$ the ball of center $0$ and radius $n^{\lambda_n}$ in
the affine chart  $\R^{d-1} \subset \PV$, we have $\ulim \Omega_n =\PVom$ and  
$\ulim \partial \Omega_n$ is the hyperplane at infinity but $\partial \ulim \Omega_n$ is empty.
  \end{enumerate}
\end{remarks}

\begin{proof}
We first prove that $ \partial \ulim \Omega_n \subset \ulim \partial\Omega_n$. 
Let $x_\omega \in \partial \ulim \Omega_n$. Since $\ulim \Omega_n$ is closed, 
there is a sequence $x_n \in \Omega_n$ such that $x_\omega=\ulim x_n$. 
We have $\dom(x_\omega, (\ulim \Omega_n)^c)=0$.
Since $(\ulim \Omega_n)^c \subset \ulim (\Omega_n^c)$,
\compl{if $\ulim x_n \notin \ulim \Omega_n$ then $x_n \notin \Omega_n$ for $\omega$-almost all
$n$ (else we would have $x_n \in \Omega_n$ for $\omega$-almost all $n$)}%
this implies that
$\dom(x_\omega, \ulim (\Omega_n^c) )=0$.
Thus since $\ulim (\Omega_n^c)$ is closed, we have $x_\omega \in \ulim
  (\Omega_n^c)$, hence we can choose a sequence $y_n \in \Omega_n^c$ such that
$x_\omega=\ulim y_n$
Now choose $z_n \in \partial \Omega_n $ such that $d(x_n,z_n)\leq d(x_n,y_n)$.
Then $\dom(x_\omega,\ulim z_n) \leq \dom(x_\omega,x_\omega)=0$,
therefore $x_\omega=\ulim z_n \in \ulim \partial\Omega_n$.

We now prove the reverse inclusion.
Since $\ulim \Omega_n$ is contained in an affine chart,
by Proposition \ref{prop:ultralimits:affineCharts}
we may reduce to the case where
$\Omega_n \subset \RR^{d-1}$ for all $n \in \N$ (possibly changing $\Omega_n$ on a
$\omega$-null subset of indices).
Let $x_\omega \in \ulim \partial\Omega_n \subset \RRom^{d-1}$. 
Then $x_\omega \in \ulim \ov{\Omega_n}= \ulim \Omega_n$. 
We have to show that $x_\omega$ is not an interior point of $\ulim \Omega_n$.
Fix some $\eps>0$ in $\RR$.
Let  $x_n \in \partial\Omega_n$ be a sequence such that $x_\omega=\ulim x_n$. 
For each $n$, since $\Omega_n$ is convex we can choose a point $y_n\in\RR^{d-1}$
such that 
$d( y_n , \Omega_n) = d( y_n , x_n) = \eps^{\lambda_n}$, 
(where $d$ denotes the Euclidean metric in $\RR^{d-1}$).
Then passing to ultralimits in $\RRom^{d-1}$ we get
\[\dom(y_\omega,x_\omega) = \lim_\omega d( y_n , x_n )^{1/\lambda_n} = \eps \] 
where $y_\omega=\ulim y_n$, and 
\[\dom(y_\omega,\ulim \Omega_n) = \lim_\omega d( y_n , \Omega_n )^{1/\lambda_n} = \eps \] 
by \Cref{prop:ulimDistToSubset}\eqref{it:ulimDistToSubset},  
hence $y_\omega \in (\ulim \Omega_n)^c$.
In particular the ball of radius $\eps$ centered at $x_\omega$  
meets the complement
of $\ulim \Omega_n$, finishing the proof.
\end{proof}

We now study the behavior of the intersection of a convex subset with a line with respect to ultralimits.
First notice that in contrast with the real case,
the intersection of the ultralimit convex
with a line may not be a closed segment, even for semi-algebraic $\Omega_n$,
as the following example shows.

\begin{example}  
\label{exa:UltralimitConvexAndLines:NotSegment}
Consider the ultralimit of the sequence of the semi-algebraic strictly convex $C^1$ subsets
\[ \Omega_n \coloneqq \{(x,y) \in \R^2 \mid x^{2n}+y^{2}<1\}\]
for the scaling sequence  $\lambda_n=n$. 
Then the ultralimit convex $\ulim \Omega_n$ is not strictly convex and the intersection of the 
line $L_\omega=\RR_\omega \times \{-1\}$  with $\ulim \Omega_n$ is   
\[L_\omega \cap \ulim \Omega_n = I \times \{-1\},\]
where $I \subset \RR_\omega$ is the subset of infinitesimal elements, namely
elements given by sequences  $(x_n)\in \RR^{\NN}$ whose $\omega$-limit in $\RR$ is $0$.
In particular,  it is not a closed segment.
\compl{
$\subseteq$: We first prove $L_\omega\cap \ulim\Omega_n \subseteq I \times \{-1\}$.
Let $(x_n,y_n) \in \Omega_n$ such that $\ulim y_n=-1$,  i.e. \ $\lim_\omega|y_n+1|^{1/n}=0$.
In particular,  $\lim_\omega y_n=-1$.
We need to show that $\lim_\omega x_n=0 \in \R$.
We have 
\begin{align*}
\lim_\omega x_n\leq \lim_\omega (1-y_n^2)^{1/2n}&=\lim_\omega (1-y_n)^{1/2n}(1+y_n)^{1/2n}\\
&=\lim_\omega (1+y_n)^{1/2n}\leq \lim_\omega(1+y_n)^{1/n}=0.
\end{align*}

$\supseteq$: We now prove $L_\omega\cap \ulim\Omega_n \supseteq I \times \{-1\}$.
It is clear that $I \times \{-1\} \subseteq L_\omega$.
Thus it is left to show that $I\times \{-1\} \subset \ulim \Omega_n=\ulim \overline{\Omega_n}$.
Let $x_\omega =\ulim x_n \in I$, i.e.\ $\lim_\omega x_n =0$.
Then $\lim_\omega x_n^{2n}=0$ and for $n$ large enough $x_n^{2n}<1$.
Let $y_n\coloneqq -(1-x_n^{2n})^{1/2}$ be a sequence in $\R$.
By definition $y_n^2+x_n^{2n}=1$,  thus $(x_n,y_n) \in \overline{\Omega_n}$ (for $n$ large enough).
We claim that $\ulim y_n=-1$.
We compute
\[
\lim_\omega |y_n+1|^{1/n}= \lim_\omega |1-(1-x_n^{2n})^{1/2}|^{1/n}=\lim_\omega |x_n^n|^{1/n}=
0,\]
since $1-(1-x_n^{2n})^{1/2}$ goes to $0$ at the same speed as $x_n^{n}$.
}
\begin{figure}[h]
\begin{tikzpicture}[scale=.8]
\begin{axis}
[axis lines=middle, domain=-1:1, samples=100, axis on top=true, xmin=-1.2, xmax=1.2, ymax=1.2, ymin=-1.2, height=7cm, width=7cm, xticklabel=\empty, yticklabel=\empty, cycle list={}, tick style={draw=none}]
\addplot+[mark=none,fill=gray!10,draw=black] ({x},{sqrt(1-x^100)}) \closedcycle;
\addplot+[mark=none,fill=gray!10,draw=black] ({x},{-sqrt(1-x^100)}) \closedcycle;

\addplot+[mark=none,fill=gray!30,draw=black] ({x},{sqrt(1-x^10)}) \closedcycle;
\addplot+[mark=none,fill=gray!30,draw=black] ({x},{-sqrt(1-x^10)}) \closedcycle;

\addplot+[mark=none,fill=gray!50,draw=black] ({x},{sqrt(1-x^4)}) \closedcycle;
\addplot+[mark=none,fill=gray!50,draw=black] ({x},{-sqrt(1-x^4)}) \closedcycle;

\addplot+[mark=none,fill=gray!70,draw=black] ({x},{sqrt(1-x^2)}) \closedcycle;
\addplot+[mark=none,fill=gray!70,draw=black] ({x},{-sqrt(1-x^2)}) \closedcycle;
\end{axis}
\end{tikzpicture}
\caption{The subsets $\Omega_n$ for $n=1,2,5,50$.}
\end{figure}
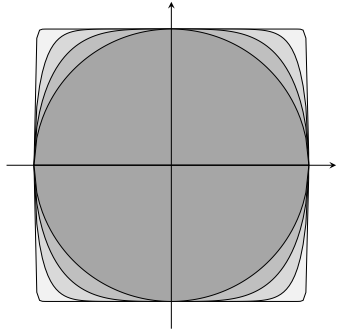
\end{example}

We now show that intersections with lines behave well 
with respect to ultralimits 
when the ultralimit line 
meets the interior of the ultralimit convex.

\begin{proposition}[Ultralimits of convex subsets are well-bordered]
\label{prop-UltralimitConvexAndLines}
Suppose that $\ulim \Omega_n$ is contained in an affine chart.
Let $(L_n)$ be a sequence of projective lines in $\PP V$ and $L_\omega$ its ultralimit.

Suppose that the line $L_\omega$ meets the interior $\Omega_\omega$ of $\ulim \Omega_n$. 
Then for $\omega$-almost all $n$ the line $L_n$ meets $\Omega_n$  
and, denoting $L_n \cap \Omega_n=]a_n,b_n[$ the intersection 
and $a_\omega=\ulim a_n$, $b_\omega=\ulim b_n$ the ultralimits of the endpoints,
we have 
\[ L_\omega \cap \ulim \Omega_n=[a_\omega,b_\omega] \text{ and } L_\omega \cap \Omega_\omega = ]a_\omega,b_\omega[ \;.\]
In particular,  the open convex $\Omega_\omega$ is well-bordered.
\end{proposition}

\begin{proof}
Since $\ulim \Omega_n$ is contained in an affine chart,
by \Cref{prop:ultralimits:affineCharts}
we may reduce to the case where
$\Omega_n \subset \RR^{d-1}$ for all $n \in\N$ and $\ulim \Omega_n \subset \RRom^{d-1}$.

We have
\[\ulim (L_n \cap \ov{\Omega_n}) \subset (\ulim L_n) \cap (\ulim \ov{\Omega_n}) 
= L_\omega \cap  \ulim \Omega_n\]
and $\ulim [a_n,b_n]=[a_\omega,b_\omega]$, 
thus  $[a_\omega,b_\omega] \subset L_\omega \cap \ulim \Omega_n$.

We now show  that $[a_\omega,b_\omega]= L_\omega \cap \ulim \Omega_n$.
Let $I=L_\omega \cap \ulim \Omega_n$ and denote by $J$ the relative interior of $I$. 
Since $\ulim \Omega_n$ is convex and $L_\omega$ meets its interior $\Omega_\omega$,
we have  $J \subset L_\omega \cap \Omega_\omega$.
By \Cref{prop:ultralimitsOfConvex:boundary} we have that
$a_\omega,b_\omega \notin \Omega_\omega$, hence $a_\omega,b_\omega \notin J$.
As $[a_\omega,b_\omega] \subset I$ and $I$ is convex, this implies that $[a_\omega,b_\omega] = I$. 
Hence $]a_\omega,b_\omega[=J \subset L_\omega \cap \Omega_\omega$ which concludes since $L_\omega \cap \Omega_\omega \subset I$.
\end{proof}

\subsection{%
\texorpdfstring%
{Proof of \Cref{thm:Intro:UltralimitsHilbertGeometries}}%
{Proof of Theorem A}
}
\label{subsection:ProofUlimHilbMetrics}
Let $(\Omega_n)$ be a sequence of non empty open convex subsets in
the projective space $\PP V$.
Let $o_n \in \Omega_n$ and let $(\lambda_n)$  be a {\em scaling sequence}, 
namely a sequence  of real numbers  such that $\lambda_n\geq 1$ and 
$\lambda_n \to \infty$. 

Let 
\[X_\omega \coloneqq \Ulim_\omega{(\Omega_n,\tfrac{1}{\lambda_n} d_{\Omega_n},o_n)}\] 
be the ultralimit of the sequence of pointed metric spaces
$(\Omega_n,\tfrac{1}{\lambda_n} d_{\Omega_n},o_n)$, where $d_{\Omega_n}$ 
is the Hilbert metric on $\Omega_n$.

Let $\PVom$  be the asymptotic cone  of the projective space $\PP V$
for the scaling sequence $(\lambda_n)$.
We denote $\Omega_\omega$ the interior of the convex subset  $\ulim \Omega_n \subset\PVom$.
It is an open convex subset of $\PVom$ (possibly empty).

We have the following characterization of points in $\Omega_\omega$: they
correspond to sequences which do not go to the boundary too quickly.
More precisely we have
\begin{equation}
\label{eq:UltralimitConvexCriterionInteriorPoint}
x_\omega \in \Omega_\omega
\Leftrightarrow
\lim_\omega d(x_n, \partial\Omega_n)^{1/\lambda_n}>0 \;.
\end{equation}
Indeed, by Proposition \ref{prop:ultralimitsOfConvex:boundary} 
we have
\[\dom(x_\omega, \partial\Omega_\omega) 
=\dom(x_\omega, \ulim \partial\Omega_n) 
= \lim_\omega d(x_n, \partial\Omega_n)^{1/\lambda_n} \;.\]

Let $\oom=\ulimPVom{o_n} \in \ulim \Omega_n$ be the ultralimit of the sequence
$(o_n)$ in $\PVom$. 

\begin{theorem}[\Cref{thm:Intro:UltralimitsHilbertGeometries}]
\label{thm:UltralimitsHilbertGeometries}
Suppose that $\ulim \Omega_n$ is contained in an affine chart
and $\oom \in \Omega_\omega$. 
\begin{enumerate}
\item 
The map
\[\Phi \from
\begin{array}[t]{cll}
{\prod_n \Omega_n} & \rightarrow & {\ulim \Omega_n} \\
{x=(x_n)} & \mapsto & {\ulimPVom x_n}
\end{array}\]   
sends the pseudo-distance $\dom$ 
to the Hilbert pseudo-distance. 
In particular it induces 
an isometry $\ov{\Phi}\from \Xom \to X_{\Omega_\omega}$, satisfying
$\ov{\Phi}(\ulimXom x_n)=\overline{\ulimPVom x_n}$.

\item 
Let $\Gamma$ be a finitely generated group, and $\rho_n \from \Gamma\to\PGL(V)$ be a representation preserving $\Omega_n$ for each $n \in \N$.
Suppose that for every $\gamma$ in a finite symmetric generating set of $\Gamma$, we have
\[\lim_\omega\, \norm{\rho_n(\gamma)}^{1/\lambda_n} < \infty \]
for any fixed norm $\norm{\cdot}$ on $\End(V)$.
Then $(\rho_n)$ induces a well-defined isometric action of $\Gamma$ on $\Xom$ and on $X_{\Omega_\omega}$, for which the map $\ov{\Phi}$ is equivariant.
\end{enumerate}
\end{theorem}

\begin{remarks}
\label{rem:UltralimitConvexCriterion}
 \noindent \begin{enumerate}
  \item 
\label{rem:UltralimitConvexCriterionAffineChart}
The hypotheses of the theorem are always
satisfied up to replacing $\Omega_n$ by $g_n(\Omega_n)$ for some $g_n \in
\PGL(V)$. 
Indeed we may then suppose that $\Omega_n \subset \RR^{d-1}$, $o_n=0$ and $\ov{B}(0,1)\subset \Omega_n \subset \ov{B}(0,r)$ for some fixed $r \in \RR_{>0}$
(\cite{Benzecri_VarietiesLocalementAffinesProjectives},  see also \cite[\S 9]{Marquis_AroundGroupsHilbertGeometry}).
Using the notations from \Cref{subsection:VectorSpacesNAOrderedFields} and \Cref{rem:UlimVSExtension},  we have
 $\ov{B}_{\|\cdot\|_{2,\RRom}}(0,1) \subset \ulim \Omega_n \subset \ov{B}_{\|\cdot\|_{2,\RRom}}(0,r) \subset \RRom^{d-1}$.
Thus $\oom=0$ is in $\Omega_\omega$.
\item 
Note that $\Omega_\omega$ is empty if and only if $\ulim \Omega_n$ is contained in a proper projective subspace.  
In particular, in the setting of the second point of the theorem, the
hypothesis that $\Omega_\omega$ is non-empty is always satisfied if $\rho_\omega$ is
non-parabolic,
i.e.\ $\rho_\omega(\Gamma)$ does not lie in a proper parabolic subgroup of $\PSL(d,\RRom)$.
\end{enumerate}
\end{remarks}

\begin{proof}[Proof of \Cref{thm:UltralimitsHilbertGeometries}]
Recall the construction of $\Xom$:  
we have a pseudo-distance $\dom$ on $\prod_n \Omega_n$ with values in $[0,\infty]$ 
defined by
 \[\dom(x,y) \coloneqq \lim_\omega\tfrac{1}{\lambda_n} d_{\Omega_n}(x_n,y_n) \]
for $x=(x_n)$ and $y=(y_{n})$ in $\prod_n \Omega_n$,
and $\Xom$ is defined as the quotient metric space, of $\omega$-bounded sequences,  namely the
quotient of the subspace 
\[\etp{\prod_n \Omega_n} \coloneqq \big\{x \in \prod_n \Omega_n \mid \dom (o,x) < \infty \big\}\]
by the equivalence relation $\dom(x,y)=0$, endowed with the metric
induced by $\dom$.
It is enough to  show that  the map
\[ \Phi \from
\begin{array}[t]{cll}
{\prod_n \Omega_n} & \rightarrow & {\PVom} \\
{x=(x_n)} & \mapsto & {x_\omega=\ulimPVom x_n}
\end{array}\]   
sends the subspace $\etp{\prod_n \Omega_n}$ of $\omega$-bounded sequences 
to $\Omega_\omega$ and 
preserves the pseudo-distances, i.e.\ for all $x$ and $y$ in $\prod_n \Omega_n$ we have
\[\dom(x,y)=d_{\Omega_\omega}(x_\omega,y_\omega) \;.\]
Let thus $x=(x_n)$ and $y=(y_n)$ in $\prod_n\Omega_n$.
Let $L_n \subset \PP V$ be a line through $x_n$ and $y_n$,   and let $a_n,b_n \in L_n \cap \partial \Omega_n$ with $a_n,x_n,y_n,b_n$ in that order.
Denote by $a_\omega,x_\omega,y_\omega,b_\omega \in \ulim \Omega_n$
the corresponding ultralimits in $\PVom$, which are on the line 
$L_\omega \coloneqq \ulim {L_n}$ and in the same order. 
Note that some of these points may collapse, for example we may have
$y_\omega=b_\omega$.
We first show that the pseudo-distance $\dom(x,y)$ can be expressed in
terms of the cross-ratio $\CR{a_\omega,x_\omega,y_\omega,b_\omega}$ when $x_\omega \in \Omega_\omega$.

Suppose that $x_\omega \in \Omega_\omega$.
Since $L_\omega \cap \Omega_\omega$ is non-empty,  and $\ulim \Omega_n$ is included in some affine chart,  \Cref{prop-UltralimitConvexAndLines} implies that
\[ L_\omega \cap \Omega_\omega = ]a_\omega,b_\omega[ \;.  \]
In particular $ a_\omega, x_\omega, b_\omega$ are three distinct points
and by \cite[Proposition 5.4]{Parreau_InvariantWeaklyConvexCocompactSubspacesSurfaceGroupsA2Buildings} we have
\[ \CR{a_\omega,x_\omega,y_\omega,b_\omega}=\ulim \CR{a_n,x_n,y_n,b_n} \] 
in $\RRom \cup \{\infty\}$.
Taking logarithm of absolute values and $\omega$-limits it follows that 
\[ \dom(x,y)
=\lim_\omega\tfrac{1}{\lambda_n} d_{\Omega_n}(x_n,y_n) 
= \log \abs{\CR{a_\omega,x_\omega,y_\omega,b_\omega}} \]
in $[0,\infty]$. 
In particular for $x_\omega,y_\omega \in \Omega_\omega$ we have 
\[d_{\Omega_\omega}(x_\omega,y_\omega) = \dom(x,y)  \; ,\]
since \(d_{\Omega_\omega}(x_\omega,y_\omega)=\log \abs{\CR{a_\omega,x_\omega,y_\omega,b_\omega}}\) by \Cref{prop-UltralimitConvexAndLines}.

Taking $(x_n)=(o_n)=o$ we  get that
$\dom(o,y) < \infty$ 
if and only if 
 $y_\omega \neq b_\omega$, 
 which is equivalent to
$y_\omega \in \Omega_\omega$ since $y_\omega \in [a_\omega,b_\omega]$ and $]a_\omega,b_\omega[ \subset \Omega_\omega$. This
shows that  
\begin{equation}
\label{eq:om-boundedIffUltralimitInInterior}
\lim_\omega\tfrac{1}{\lambda_n} d_{\Omega_n}(o_n,y_n) < \infty \text{ if and only if }
y_\omega \in \Omega_\omega \;.
\end{equation}
In particular,  the map $\Phi$ sends the set $\etp{\prod_n \Omega_n}$ of $\omega$-bounded sequences to $\Omega_\omega$, 
which concludes the proof of the first point.

For the second point, recall that we want to prove that 
$(\rho_n)$ induces well-defined isometric actions of $\Gamma$ on
$\Xom$ and on $X_{\Omega_\omega}$. We first deal with  the action on  $X_{\Omega_\omega}$.

Observe that by assumption 
$(\rho_n(\gamma))$ is $\omega$-bounded in $\GL(V)$ with respect to  the scaling sequence $(\lambda_n)$
for any $\gamma\in F$, hence for any $\gamma\in\Gamma$.
Thus $(\rho_n(\gamma))$ has a well-defined ultralimit $\rho_\omega(\gamma)$ in
$\PGL(\Vom)$, whose action on $\PVom$ satisfies
\[\rho_\omega(\gamma) (\ulim x_n) =\ulim \rho_n(\gamma) (x_n)\]
for every sequence $(x_n)$ in $\PP V$
(see \Cref{subsection:UlimInLinearGroup}). 
Since $\rho_n(\gamma)$ preserves $\Omega_n$, we have that $\rho_\omega(\gamma)$
preserves $\ulim \Omega_n$,  hence its interior $\Omega_\omega$ and the associated
Hilbert pseudo-distance,
thus induces an isometry of $X_{\Omega_\omega}$.

We now turn to the action on $\Xom$.
Since
$\rho_\omega(\gamma)(\oom)=\ulim_{\PVom} \rho_n(\gamma)(o_n)$ is in $\Omega_\omega$, 
we have
$\lim_\omega\tfrac{1}{\lambda_n} d_{\Omega_n}(o_n , \rho_n(\gamma) o_n) < \infty$
by \eqref{eq:om-boundedIffUltralimitInInterior}.
Hence it induces an isometry $\rho_\omega'(\gamma)$
of $\Xom$ defined by
\[\rho_\omega'(\gamma) (\ulim_{\Xom} x_n) \coloneqq \ulim_{\Xom} \rho_n(\gamma) (x_n)\; .\]
The equivariance of the map $\ov{\Phi}$
follows from the first point, since
\[\ov{\Phi}(\ulimXom \rho_n(\gamma) (x_n))=\overline{\ulimPVom \rho_n(\gamma) (x_n)}=\rho_\omega(\gamma) (\overline{\ulimPVom x_n})\;.\qedhere \]
\end{proof}

\subsection{%
\texorpdfstring{Asymptotic cones of real polytopes and proof of \Cref{thm:Intro:AsymptoticConesPolytopalHilbertGeometries}}
{Asymptotic cones of real polytopes and proof of Theorem E}}
\label{subsection:AsymptoticConeRealPolytope}

In this section we combine \Cref{thm:Intro:UltralimitsHilbertGeometries} and \Cref{thm:Intro:NAPolytopalHilbertMetricSpace} to show that all asymptotic cones of a real polytope $\polyt$ with Hilbert metric 
are isometric to the geometric realization $\Model=\Model(\polyt,\RR)$ of $\Flags_*(\polyt)$ modeled on $\Aap$.
Denote $o=\ov{0}$ the cone point of $\Model$.

Recall from \Cref{subsection:UltralimitsMetricSpaces} that given a metric space $(E,d)$,  a scaling sequence
is a sequence $(\lambda_n)$ of real numbers  
such that $\lambda_n\geq 1$ and $\lambda_n \to \infty$.
The asymptotic cone of $(E,d)$ with respect to $(\lambda_n)$,  observation points $(o_n)$ in $E$ and ultrafilter $\omega$ is the $\omega$-ultralimit of 
the sequence of pointed metric spaces $(E,\tfrac{1}{\lambda_n}d, o_n)$.

\begin{theorem}[\Cref{thm:Intro:AsymptoticConesPolytopalHilbertGeometries}]
\label{thm:AsymptoticConesPolytopalHilbertGeometries}
Let $\polyt \subset \RR^{d-1} \subset\PP(\RR^d)$ be a bounded real convex polytope with non-empty interior $\Omega  \subset \RR^{d-1} \subset\PP(\RR^d)$.
Let $\omega$ be a non-principal ultrafilter on $\NN$,  and let $(X_\omega,d_\omega,o_\omega)$ be the asymptotic cone of $(\Omega, d_\Omega)$ for the scaling sequence $(\lambda_n)_{n \in \N}$ and fixed observation points $o_n=x_0\in \Omega$.
Then $(X_\omega,o_\omega)$ is isometric to $(\Model,o)$.
\end{theorem}

\begin{proof}
We first apply Theorem \ref{thm:Intro:UltralimitsHilbertGeometries} to identify $X_\omega$ with the metric space associated to $\polyt_{\FF}$ with its Hilbert geometry for some non-Archimedean real closed field $\FF$.
Let $\ulim \Omega$ be the ultralimit of $\Omega$ (seen as a constant
sequence of subsets) in the asymptotic cone $\PVom$ of the projective space
$\PP V$ for the scaling sequence $(\lambda_n)$ and ultrafilter $\omega$; refer to \Cref{subsection:UlimConvexProjSets}.
%
Since $\Omega$ is bounded in the standard affine chart,  $\ulim \Omega$  is contained in the standard affine chart.
Moreover, since  $d(o , \Omega^c)>0$ {(in particular constant)},  we have 
$\lim_\omega d( o ,\Omega^c )^{1/\lambda_n} = 1>0$.
Hence  $\oom = \ulim o$ is in the interior $\Omega_\omega$ of {$\ulim \Omega$} by \Cref{eq:UltralimitConvexCriterionInteriorPoint} {in \Cref{subsection:ProofUlimHilbMetrics}.}
Hence we can apply Theorem \ref{thm:Intro:UltralimitsHilbertGeometries} 
and we get an isometry $\ov{\Phi} \from \Xom \to X_{\Omega_\omega}$ satisfying
$\ov{\Phi}(\ulimXom x_n)=\overline{\ulimPVom x_n}$.

The ultralimit $\ulim \Omega =\ulim \polyt$ of $\polyt$ is the $\RRom$-extension of $\polyt$ (\Cref{rem:UlimVSExtension}),
where $\RRom=\Ulim_\omega(\RR, \abs{\cdot}^{1/\lambda_n},0)$ is the
associated Robinson field.
In particular $\ulim \Omega= \ulim \polyt= \polyt_{\RRom}$.
Moreover $\ov{\Phi}$ sends $\oom$ to the barycenter $b_\omega$ of $\polyt_{\RRom}$: indeed denoting $b$ the barycenter of $\polyt$ we have that $\ulimXom o =
\ulimXom b$,  since $d_\Omega(o,b)$ is constant and hence $\lim_\omega \tfrac{1}{\lambda_n}d_\Omega(o , b ) = 0$.
Furthermore,  we have $\ulimPVom b=b_\omega$ since the ultralimit in $\PVom$ of a constant sequence preserves the barycenter.
Putting everything together we get
\[\ov{\Phi}(\oom)=\ov{\Phi}(\ulimXom o)=\ov{\Phi}(\ulimXom b)=\ulimPVom b=b_\omega.\]
We conclude by applying \Cref{thm:Intro:NAPolytopalHilbertMetricSpace} with
$\FF=\RRom$ and remarking that $\RRom$ has surjective valuation.
\end{proof}

\subsection{%
\texorpdfstring%
{Convex real projective structures and proof of  \Cref{thm:Intro:DegenerationsConvProjStr}}%
{Convex real projective structures and proof of Theorem B}
}
\label{subsection:ConvexRealProjStr}

The goal of this section is to use the above result to associate to a closed point in the boundary of the real spectrum compactification of the space of convex real projective structures a properly convex open subset of the projective space over a Robinson field.
%

Let $M$ be a closed connected oriented topological manifold of dimension $d-1 \geq 2$ with hyperbolic fundamental group $\Gamma=\pi_1(M)$.
Denote by $\PS(M)$ the space of marked convex real projective structures on $M$.
Via the holonomy representations the space $\PS(M)$ can be identified with a subset of the $\PSL(d,\R)$-character variety of reductive representations modulo conjugation
\[\mathcal{X}(\Gamma, d) \coloneqq \Hom_\textnormal{red}(\Gamma,\PSL(d,\R))/\PSL(d,\R).\]
where two representations are identified if they are conjugate.
This is true,  since the holonomy representation of a marked convex real projective structure is irreducible \cite{Vey_AutomAffinesOuvertsConvexesSaillants}.
It turns out that $\PS(M)$ is a union of connected components of $\mathcal{X}(\Gamma,d)$ \cite{Koszul_DeformationsConnexionsLocalementPlates, Benoist_ConvexesDivisiblesIII},  and thus a closed and semi-algebraic subset of $\mathcal{X}(\Gamma,d)$.
\compl{This is true if and only if $\Gamma$ has trivial virtual center, e.g.\ if $M$ is hyperbolic.}

As such $\PS(M)$ admits a compactification using the real spectrum, denoted by $\overline{\PS(M)}^{\RSp}\subset \overline{\mathcal{X}(\Gamma,d)}^{\RSp}$; for more details we refer to \cite{BurgerIozziParreauPozzetti_RSCCharacterVarieties2}.
%
The space $\overline{\mathcal{X}(\Gamma,d)}^{\RSp}$ is endowed with a topology that turns it into a compact space,  but that is not necessarily Hausdorff.
Restricting to the set of closed points,  yields indeed a compact Hausdorff space \cite[Proposition 7.1.25 (ii)]{BochnakCosteRoy_RealAlgebraicGeometry}.
In \cite[Theorem 1.1]{BurgerIozziParreauPozzetti_RSCCharacterVarieties2} the authors prove that there is a bijective correspondence 
\begin{align*}
\partial^{\RSp}\mathcal{X}(\Gamma,d)&\coloneqq \overline{\mathcal{X}(\Gamma,d)}^{\RSp}\setminus \mathcal{X}(\Gamma,d)\\
&\cong \left\{ (\rho, \F_\rho) \ \middle\vert 
\begin{array}{l}
	\rho \from \Gamma \to \PSL(d,\F_\rho) \textrm{ reductive homomorphism,}\\
\F_\rho \supseteq \R \textrm{ real closed,  non-Archimedean,  minimal}
\end{array}
\right\} {\Big/ \sim} \,,
\end{align*}
where two pairs $(\rho,\F_\rho)$ and $(\rho',\F'_{\rho'})$ are \emph{equivalent}, if there exists a field isomorphism $\sigma \from \F \to \F'$ such that $\sigma \circ \rho$ and $\rho'$ are $\PSL(d,\F_{\rho'})$-conjugate.
The field $\F_\rho$ is called \emph{minimal} if $\rho$ cannot be $\PSL(d,\F_\rho)$-conjugated into a representation $\rho \from \Gamma \to \PSL(d,\K)$ for $\K \subset \F_\rho$ a proper real closed subfield.

\begin{definition}
\label{definition:RepresentingPtsinRSC}
We say that a reductive representation $\rho' \from \Gamma \to \PSL(d,\F')$ \emph{represents} a point $\alpha=[(\rho,\F_\rho)] \in \overline{\mathcal{X}(\Gamma,d)}^\RSp$ if $\F'\supseteq \R$ is real closed and there exist a real closed field $\F$ and two field homomorphisms $\sigma' \from \F' \to \F$ and $\sigma_\rho \from \F_\rho \to \F$ such that $\sigma' \circ \rho'$ and $\sigma_\rho \circ \rho$ are conjugate in $\PSL(d,\F)$.
\end{definition}

We can now prove \Cref{thm:Intro:DegenerationsConvProjStr}.

\begin{theorem}[\Cref{thm:Intro:DegenerationsConvProjStr}]
\label{thm:DegenerationsConvProjStr}
Let $\alpha \in \partial^{\RSp} \PS(M)$ be a closed point.
Then there exist a non-Archimedean,  real closed valued field $\F$ and a representation $\rho \from \Gamma \to \PSL(d,\F)$ representing $\alpha$,  such that $\rho(\Gamma)$ preserves a non-empty,  open,  properly convex,  well-bordered subset $\Omega \subset\PP(\F^d)$,  and $\Gamma$ acts on the associated Hilbert metric space $X_\Omega$ without global fixed point.
\end{theorem}

\begin{proof}
We construct $\rho$, $\F$ and $\Omega$ as ultralimits of the corresponding real objects.
Choose any reductive representation $\rho' \from \Gamma \to \PSL(d,\F')$
representing $\alpha$ with $\F'$ non-Archimedean,  real closed and minimal.

Pick a non-principal ultrafilter $\omega$ and a sequence of scalars $\lambda=(\lambda_n) \in \R^\N$ tending to infinity.
Since $(\rho',\F')$ represents a closed point,  there exist an order-preserving field injection $i\from \F' \hookrightarrow \RRom$ into the Robinson field with respect to $\lambda$ and $\omega$,  and a sequence of representations $\rho_n \from \Gamma \to \PSL(d, \R)$ 
such that $i \circ \rho'$ and $\rho_\omega$,  the ultralimit of $\rho_n$ (see \Cref{subsection:UlimInLinearGroup} and proof of \Cref{thm:Intro:UltralimitsHilbertGeometries}),  are $\PSL(d,\RRom)$-conjugate \cite[Theorem 7.16]{BurgerIozziParreauPozzetti_RSCCharacterVarieties2}.
In other words $(\rho_\omega, \RRom)$ represents $\alpha$.

For $n\in \N$,  let $\Omega_n \subset \PP V$ be the properly convex domain divided by $\rho_n(\Gamma)$.
By the considerations at the beginning of \Cref{subsection:UlimConvexProjSets},  the ultralimit $\ulim \Omega_n \subset \PVom$ is a closed convex subset preserved by $\rho_\omega(\Gamma)$.
We have that the interior $\Omega_\omega$ of $\ulim\Omega_n$ is open,  convex and well-bordered,  see \Cref{prop-UltralimitConvexAndLines}.
We are left to show that $\Omega_\omega$ is non-empty and properly convex.

We first show that any reductive representation $\rho \from \Gamma \to \PSL(d,\F)$ representing $\alpha$ is non-parabolic.
Since $\PS(M)$ is a semi-algebraic subset consisting of conjugacy classes of representations that are non-parabolic \cite{Vey_AutomAffinesOuvertsConvexesSaillants},  its preimage $\mathfrak{T}$ in $\Hom_{\textnormal{red}}(\Gamma,\PSL(d,\R))$ is a closed,  $\PSL(d,\R)$-invariant semi-algebraic subset consisting of non-parabolic representations.
Since $\rho \from \Gamma \to \PSL(d,\F)$ also represents a point in the real spectrum compactification of $\mathfrak{T}$ \cite[Corollary 7.8]{BurgerIozziParreauPozzetti_RSCCharacterVarieties2},  we conclude that $\rho$ is non-parabolic \cite[Proposition 6.6]{BurgerIozziParreauPozzetti_RSCCharacterVarieties2}. 
Since $(\rho_\omega,\RRom)$ represents $\alpha$,  this implies that $\rho_\omega$ is non-parabolic,  and hence $\ulim \Omega_n$ is not contained in a proper projective subspace.  
In particular,  its interior $\Omega_\omega$ is a non-empty.

To prove that $\Omega_\omega$ is properly convex, i.e.\ bounded in an affine chart,  we look at the dual $\Omega_n^*$ of $\Omega_n$ in $\PV^*$.
Observe that the ultralimit $\ulim(\Omega_n^*) \subset \PP(\Vom^*)$ is a convex subset preserved by the contragradient representation $\rho_\omega^{-\top}$.
Thus $\ulim(\Omega_n^*)$ has non-empty interior since $\rho_\omega^{-\top}$ is also non-parabolic. 
In particular $\ulim(\Omega_n^*)$ contains a $(d-1)$-simplex
$S_{\eb_\omega^*}=\PP(\sum_{i=1}^d  x_i e_{\omega,i}^{*} \mid x_i \in \RRom, \, x_i \geq0)$,  where
$\eb_\omega^*=(e_{\omega,i}^{*})_{i}$ is a basis of $\Vom^{*}$.
Taking for each $i$ a sequence $(e_{n,i}^{*})_n$ in $\Omega_n^{*}$
with ultralimit $e_{\omega,i}^* \in \ulim(\Omega_n^*)$, we have that  $\eb_n^{*}=(e_{n,i}^{*})_i$ is a
basis of $V^{*}$ for $\omega$-almost all $n$, 
and the associated simplex $S_{\eb_n^*}$ is included in $\Omega_n^{*}$ and
has ultralimit  $S_{\eb_\omega^*}$.
Then for $\omega$-almost all $n$, the convex $\Omega_n$ is included in the simplex $S_{\eb_n}$,  where  $\eb_n$  is the dual basis of $\eb_n^{*}$.  
Passing to ultralimits, it is easy to see that 
the ultralimit $\eb_\omega$  of $\eb_n$ is the dual basis of $\eb_\omega^{*}$,
hence $\ulim \Omega_n$ is contained in the simplex $\ulim S_{\eb_n}=S_{\eb_\omega}$.
In particular $\ulim \Omega_n$ is contained in an affine chart and
bounded in it,  hence also its interior $\Omega_\omega$.
\compl{This proves also that the interior of $\ulim(\Omega_n^*)$ is included in $(\ulim \Omega_n)^*$, hence equal since the latter is open and included in $\ulim(\Omega_n^*)$.}

\smallskip
To show that $\Gamma$ acts on $X_\Omega$ without global fixed point,  
we show that the translation length 
\[\ell_{X_{\Omega_\omega}}(\gamma) \coloneqq \inf_{x \in X_{\Omega_\omega}} d_\Omega(x,\rho_\omega(\gamma)x)\]
of an element $\gamma \in \pi_1(M)$ is bounded from below by the logarithm of the ratio of the absolute values of the largest and smallest eigenvalue of $\rho_\omega(\gamma)$.

Let $\gamma\in\pi_1(M)$ be non-trivial.
For every $n \in \NN$,  $\rho_n(\gamma)$ is positively biproximal \cite[Proposition 5.1]{Benoist_ConvexesDivisiblesI}.
We denote by $x_n^+, x_n^- \in \PP V$ and $\varphi_n^+, \varphi_n^- \in \PP V^*$ the attracting respectively repelling fixed points of $g_n\coloneqq \rho_n(\gamma)$.
Then $x_n^+, x_n^- \in \partial \Omega_n$ and $\varphi_n^+, \varphi_n^- \in \partial \Omega_n^*$.

We use the same argument as for non-parabolicity to show that any representation representing $\alpha$ is positively biproximal.
Since positive biproximality is a semi-algebraic condition satisfied by every $\rho \in \mathfrak{T}$,  which is a closed $\PSL(d,\R)$-invariant semi-algebraic set, the same property is satisfied for every representation representing a point in the real spectrum compactification of $\mathfrak{T}$ (and thus also of $\PS(M)$) \cite[Proposition 6.15]{BurgerIozziParreauPozzetti_RSCCharacterVarieties2}.
In particular $g_\omega \coloneqq \rho_\omega(\gamma)$ is positively biproximal,  and
its two attractive and repulsive fixed points $\varphi_\omega^\pm$ in $\PP \Vom$ are the ultralimits of those of $g_n$ by \Cref{propo:UlimProximalElements}.
This implies that  $\varphi_\omega^\pm$ is in $\Omega_\omega^*$. 
Indeed, $\varphi_\omega^\pm=\ulim \varphi_n^\pm$ do not change of sign on the convex $\ulim \Omega_n$, hence do not vanish on its interior $\Omega_\omega$.

Thus for every $y_\omega \in \Omega_\omega$ we have
\[d_{\Omega_\omega}(y_\omega,g_\omega y_\omega) \geq \log \left|\CR{\varphi_\omega^-,y_\omega,g_\omega y_\omega,\varphi_\omega^+}\right|= \log \frac{|\lambda_1(g_\omega)|}{|\lambda_d(g_\omega)|}\eqqcolon \ell_\infty(g_\omega), \] 
which implies that $\ell_{X_{\Omega_\omega}}(g_\omega) \geq \ell_\infty(g_\omega)$.

Recall that $\alpha$ is represented by a homomorphism $\rho' \from \pi_1(M) \to \PSL(d,\F')$ with $\F'$ a minimal field.
The field $\F'$ is endowed with a non-trivial order-compatible absolute value $\abs{\cdot}$, that is unique up to multiplication by a positive real scalar \cite[Lemma 2.7]{BurgerIozziParreauPozzetti_RSCCharacterVarieties2}.
Since $\alpha$ is a closed point there exists $\gamma \in \pi_1(M)$ with 
$|\lambda_1(\rho'(\gamma))|>|\lambda_d(\rho'(\gamma))|$, see \cite[Theorem 1.2]{BurgerIozziParreauPozzetti_RSCCharacterVarieties2}.

The inclusion $i \from \F' \hookrightarrow \RRom$ preserves absolute values (for some choice of positive scaling).
Since $i \circ \rho'$ and $\rho_\omega$ are conjugate representations,  it follows also for $\rho_\omega$ that
\[|\lambda_1(\rho_\omega(\gamma))|>|\lambda_d(\rho_\omega(\gamma))|.\]
Thus $\Gamma$ acting on $X_{\Omega_\omega}$ via $\rho_\omega$ does not have a global fixed point as $\ell_{X_{\Omega_\omega}}(g_\omega) \geq \ell_\infty(g_\omega)>0$.
\end{proof}

\bibliographystyle{alpha}
\bibliography{bibl}

\end{document}

%% file: macros.tex
 \DeclareUnicodeCharacter{2265}{\geq}
\DeclareUnicodeCharacter{2212}{-}
 \usepackage{newunicodechar} 

 \newunicodechar{ℝ}{\mathbb{R}}
 \newunicodechar{×}{\times}
 \newunicodechar{⩾}{\geq}
 \newunicodechar{⩽}{\leq}
 \newunicodechar{∈}{\in}
 \newunicodechar{⊂}{\subset}
 \newunicodechar{Γ}{\Gamma}
 \newunicodechar{π}{\pi}
 \newunicodechar{Φ}{\Phi}
 \newunicodechar{φ}{\varphi}
 \newunicodechar{ψ}{\psi}
 \newunicodechar{α}{\alpha}
 \newunicodechar{β}{\beta}
 \newunicodechar{γ}{\gamma}
 \newunicodechar{′}{'}
 \newunicodechar{⊥}{\perp}
\newunicodechar{∂}{\partial}

\renewcommand {\epsilon}{\varepsilon} 
\newcommand{\eps}{\varepsilon}        

\newcommand{\RR}{\mathbb{R}}
\newcommand{\NN}{\mathbb{N}}
\newcommand{\ZZ}{\mathbb{Z}}

\newcommand{\QQ}{\mathbb{Q}}
\newcommand{\KK}{\mathbb{K}}
\newcommand{\PP}{\mathbb{P}}

\newcommand{\FF}{\mathbb{F}}
\renewcommand{\SS}{\mathbb{S}} 
\newcommand{\Aa}{\mathbb{A}} 

\newcommand{\R}{\mathbb{R}}
\newcommand{\N}{\mathbb{N}}
\newcommand{\Z}{\mathbb{Z}}
\newcommand{\Q}{\mathbb{Q}}
\newcommand{\C}{\mathbb{C}}
\newcommand{\F}{\mathbb{F}}

\newcommand{\B}{\mathbb{B}}
\newcommand{\K}{\mathbb{K}}
\renewcommand{\L}{\mathbb{L}}

\newcommand{\ov}{\overline}

\newcommand{\abs}[1]{\left\lvert #1 \right\rvert}     
\newcommand{\norm}[1]{\lVert #1 \rVert} 

\newcommand{\pinfty}{{+\infty}}         
\newcommand{\minfty}{{-\infty}}         

\newcommand{\function}[5]{
{#1} \from
\begin{array}[t]{cll}
{#2} & \rightarrow & {#3} \\
{#4} & \mapsto & {#5}
\end{array}   
}

\DeclareMathOperator{\Hom}{Hom} 
           
\DeclareMathOperator{\Isom}{Isom}           

\DeclareMathOperator{\Mat}{Mat}

\DeclareMathOperator{\Inte}{Int}
\newcommand{\Int}[1]{\Inte{(#1)}}

\DeclareMathOperator{\Id}{Id}        
\DeclareMathOperator{\diag}{diag}   
\DeclareMathOperator{\End}{End}

\DeclareMathOperator{\GL}{GL}        
\DeclareMathOperator{\PGL}{PGL}
\DeclareMathOperator{\SL}{SL}
\DeclareMathOperator{\PSL}{PSL}

\DeclareMathOperator{\SO}{SO}

\newcommand{\dhex}{d_{\textnormal{H}}}
\newcommand{\dmodel}{d_{\textnormal{H}}}
\newcommand{\dquot}{\bar{d}}
\newcommand{\nhex}[1]{\lVert #1 \rVert_{\textnormal{H}}}

\newcommand*\from{\colon}

\newcommand{\RSp}{\textnormal{RSp}}

\newcommand{\cl}{\textnormal{cl}}
\newcommand{\Symm}{\textnormal{Sym}}
\newcommand{\PS}{\mathfrak{P}}
\newcommand{\Flags}{\textnormal{Flags}} 
\newcommand{\Barc}{\textnormal{Bar}} 
\newcommand{\Aap}{\Aa^+}
\newcommand{\Faces}{\textnormal{Faces}} 
\newcommand{\simplex}{S} 
\newcommand{\Model}{K}
\newcommand{\polyt}{P}
\newcommand{\Kpolyt}{{\mathbf{P}}}
\newcommand{\Fpolyt}{P}

\newcommand{\Sin}{\simplex_{\eb}}
\newcommand{\Sout}{\simplex_{\eb'}}
\newcommand{\projM}{\alpha}
\newcommand{\CR}[1]{\textrm{cr}\left(#1\right)}
\newcommand{\conv}{\textnormal{conv}}

\DeclareMathOperator{\ulim}{ulim}  

    \newcommand{\dn}{d_n} 

    \newcommand{\oom}{o_\omega}
    \newcommand{\Xom}{X_\omega}%

\newcommand{\absom}[1]{\abs{#1}}     %

    \newcommand{\Vom}{V_\omega}%

\newcommand{\eb}{\mathbf{e}} 
\newcommand{\en}{\mathbf{e}_n} 
\newcommand{\eom}{\mathbf{e}_\omega} 

\newcommand{\dPV}{d_{\SS}}   

\newcommand{\PV}{\PP V} 

\DeclareMathOperator{\Ulim}{Ulim}  
\newcommand{\RRom}{\RR_\omega}
\newcommand{\PVom}{\PP V_\omega}%

\newcommand{\ulimPVom}[1]{\ulim_{\PVom}#1} 
\newcommand{\ulimXom}[1]{\ulim_{X_\omega}#1} 

\newcommand{\dom}{d_\omega}%

\newcommand{\etp}[1]{{{\big(#1\big)^b}}}